\documentclass[11pt,a4paper,reqno]{amsart}
\usepackage[utf8]{inputenc}
\usepackage[T1]{fontenc}
\usepackage[english]{babel}

\usepackage{sidecap}

\usepackage{verbatim}
\usepackage{lmodern}
\usepackage{amsmath}
\usepackage{amssymb} 
\usepackage{amsthm} 
\usepackage{thmtools}
\usepackage{enumitem}
\usepackage{graphicx}
\usepackage[hidelinks]{hyperref}
\usepackage{mathtools}
\usepackage{indentfirst}
\usepackage{tabularx}
\usepackage{bbm}
\usepackage[capitalise]{cleveref}
\usepackage{stmaryrd}
\usepackage{tabularx}

\usepackage[left=1.2in,top=2cm,right=1.2in,bottom=2cm]{geometry}

\usepackage{array}
\newcolumntype{C}[1]{>{\centering\let\newline\\\arraybackslash\hspace{0pt}}m{#1}}

\usepackage{tikz} 
\usepackage{tikz-cd} 
\newcommand{\btk}{\begin{tikzcd}}
\newcommand{\etk}{\end{tikzcd}}
\usetikzlibrary{arrows,calc,decorations.markings}
\usepackage{xparse}
\newcommand{\Hom}{\mathrm{Hom}}

\usepackage{blindtext}

\makeatletter
\newcommand{\@bbify}[1]{
  \ifcsname b#1\endcsname
  \message{WARNING: Overwriting b#1 with blackboard letter!}
  \fi
  \expandafter\edef\csname b#1\endcsname
  {\noexpand\ensuremath{\noexpand\mathbb #1}\noexpand\xspace}}
\newcommand{\@calify}[1]{
  \ifcsname c#1\endcsname
  \message{WARNING: Overwriting c#1 with calligraphic letter!}
  \fi 
  \expandafter\edef\csname c#1\endcsname
  {\noexpand\ensuremath{\noexpand\mathcal #1}\noexpand\xspace}}
\newcommand{\@bfify}[1]{
  \ifcsname bf#1\endcsname
  \message{WARNING: Overwriting c#1 with bold letter!}
  \fi
  \expandafter\edef\csname bf#1\endcsname
  {\noexpand\ensuremath{\noexpand\mathbf #1}\noexpand\xspace}}
\newcounter{@letter}\stepcounter{@letter}
\loop\@bbify{\Alph{@letter}}\@calify{\Alph{@letter}}\@bfify{\Alph{@letter}}
\ifnum\the@letter<26\stepcounter{@letter}\repeat
\makeatother

\newenvironment{tz}{\begin{center}\begin{tikzpicture}}{\end{tikzpicture}\end{center}}
\tikzstyle{d}=[double distance=.3ex]
\tikzstyle{w}=[preaction={draw=white,-,line width=5pt}]

\newcounter{diagram}
\renewcommand{\thediagram}{\thetheorem}
\newenvironment{diagram}{
\setcounter{diagram}{\value{theorem}}
\refstepcounter{theorem}\refstepcounter{diagram}
\begin{center}
\normalfont{(\thediagram)}\hfill\begin{tikzpicture}[baseline=(current bounding box.center)]}
    {\end{tikzpicture}\hfill\text{ }\end{center}}

\tikzset{%
node distance=1.5cm, la/.style={scale=0.8}, lasmall/.style={scale=0.75}, over/.style={auto=false,fill=white,inner sep=1.5pt, minimum size=0, outer sep=0},
    symbol/.style={%
        draw=none,
        every to/.append style={%
            edge node={node [sloped, allow upside down, auto=false]{$#1$}}},
            
    }, pro/.style={postaction={decorate,decoration={
        markings,
        mark=at position .5 with {\node at (0,0) {$\bullet$};}
      }},
      inner sep=.9ex,
      },
      prosmall/.style={postaction={decorate,decoration={
        markings,
        mark=at position .5 with {\node at (0,0) {$\scriptstyle \bullet$};}
      }},
      inner sep=.9ex,
      },
  n/.style={double equal sign distance, -implies}, t/.style={double distance=2.5pt, -implies, postaction={draw,-}},
}

\newcommand{\pushout}[1]{\node at ($({#1})-(10pt,-10pt)$) {$\ulcorner$};}

\newcommand{\arrowdot}{
\ensuremath{\begin{tikzpicture}
\node (A) at (0,-.4) {};
\node (B) at (.4,-.4) {};
\draw[->, line width=.1ex] (0,-.6) -- (.4,-.6);
\node[shape=circle, fill=black, scale=0.35] (A) at  (.17,-.6) {};
\end{tikzpicture}
}}

\newcommand{\vcong}{\rotatebox{270}{$\cong$}}
\tikzset{%
node distance=1.5cm, la/.style={scale=0.8}, rr/.style={xshift=1.5cm},
space/.style={xshift=.5cm},
    symbol/.style={%
        draw=none,
        every to/.append style={%
            edge node={node [sloped, allow upside down, auto=false]{$#1$}}},
            
    }
}

\newcommand{\sqboundary}[4]{({#3} \; ^{{#1}}_{\substack{{#2}}} \; {#4})}

\newcommand{\colim}{\mathrm{colim}}

\newcommand{\an}{\ensuremath{\mathrm{An}}}

\newcommand{\cof}{\ensuremath{\mathrm{cof}}}

\newcommand{\nfib}{\ensuremath{\mathrm{NFib}}}

\newcommand{\inj}{\mathrm{inj}}

\newcommand{\Path}{\mathrm{Path}}
\newcommand{\id}{\mathrm{id}}

\newcommand{\Sq}{\mathrm{Sq}}

\newcommand{\Wf}{\mathcal{W}_f}
\newcommand{\dblcat}{\mathrm{DblCat}}
\newcommand{\twocat}{2\mathrm{Cat}}
\newcommand{\twogpd}{2\mathrm{Gpd}}

\newcommand{\cat}{\mathrm{Cat}}
\newcommand{\graph}{\mathrm{Graph}}

\newcommand{\dbldersch}{\mathrm{DblDerSch}}

\newcommand{\Eadj}{E_\mathrm{adj}}

\newlist{rome}{enumerate}{7}
\setlist[rome]{label=(\roman*)}

\newtheorem{theorem}{Theorem}[section]
\newtheorem{cor}[theorem]{Corollary}
\newtheorem{prop}[theorem]{Proposition}
\newtheorem{lem}[theorem]{Lemma}
\declaretheorem[name=Theorem,numbered=yes]{theoremA}

\theoremstyle{definition}
\newtheorem{defn}[theorem]{Definition}

\newtheorem{ex}[theorem]{Example}
\newtheorem{notation}[theorem]{Notation}

\theoremstyle{remark}
\newtheorem{rem}[theorem]{Remark}

\crefname{theorem}{Theorem}{Theorems}
\crefname{cor}{Corollary}{Corollaries}
\crefname{prop}{Proposition}{Propositions}
\crefname{lem}{Lemma}{Lemmas}

\crefname{defn}{Definition}{Definitions}
\crefname{terminology}{Terminology}{Terminologies}
\crefname{ex}{Example}{Examples}
\crefname{notation}{Notation}{Notations}
\crefname{descr}{Description}{Descriptions}
\crefname{constr}{Construction}{Constructions}

\crefname{rem}{Remark}{Remarks}

\title{Double categorical equivalences}

\author[L.\ Moser]{Lyne Moser}
\address{Fakultät für Mathematik, Universität Regensburg, 93040 Regensburg, Germany}
\email{lyne.moser@ur.de}

\author[M.\ Sarazola]{Maru Sarazola}
\address{School of Mathematics, University of Minnesota, Minneapolis MN, 55415, USA}
\email{maru@umn.edu}

\author[P.\ Verdugo]{Paula Verdugo}
\address{Max-Planck-Institut für Mathematik, 53111 Bonn, Germany}
\email{verdugo@mpim-bonn.mpg.de}

\begin{document}

\begin{abstract}
We present an efficient and user-friendly method for constructing any cofibrantly generated model structure on the category of double categories whose trivial fibrations are the ``canonical'' ones: the double functors which are surjective on objects, full on both horizontal and vertical morphisms, and fully faithful on squares. We show that all of these model structures are left proper and that they are localizations of the gregarious model structure introduced by Campbell. As a notable consequence, this identifies the gregarious weak equivalences as the ``canonical'' equivalences of double categories, an elusive notion thus far. Moreover, the nature of our method gives an explicit description of the fibrant objects in terms of lifting conditions. We use this to recover several known model structures, as well as construct several new examples whose homotopy theories encode a range of $2$-dimensional structures.  
\end{abstract}

\maketitle

\setcounter{tocdepth}{1}
\tableofcontents

\section{Introduction}

Double categories are 2-dimensional structures that allow us to encode situations where we have a collection of objects and two types of morphisms between them. These types of morphisms, which we visualize as ``horizontal'' and ``vertical'', might interact with each other---for instance, between $R$-modules we have monomorphisms and epimorphisms which can be composed---or they might be of a different nature---for example, between rings we have ring homomorphisms and bimodules, or between manifolds we have cobordisms and diffeomorphisms. To encode the interplay between these two types of morphisms, double categories have a collection of 2-dimensional morphisms called ``squares''.

Since their introduction by Ehresmann \cite{ehresmann}, double categories have been the focus of much work in the category theory community; see for instance \cite{brown2,grandispare1,dawson,dawsonpare,holonomy,grandispare2,fiore,kock, shulman1,shulman,garner,grandispare3,christinav,Verity}, among several others. Interestingly, in the last decade or so, double categories also started to feature prominently in other fields, a trend that has much accelerated in the past few years. This includes, among other things, a study of the different homotopy theories they can encode \cite{FP,FPP,MSV,whi}, their connection to (weak) factorization systems \cite{bourkeWFS,stepan}, their use as models for higher categories \cite{paolipronk,jaco,lyneleo} and as new frameworks for algebraic $K$-theory \cite{CZCGW,SSCGW,CKMZ:squares,segalsquares}, as a way to give rise to a notion of homology for sets \cite{ACGWhomology}, and in connection to various settings like operads \cite{pisani}, categorical databases \cite{brandonallconcepts}, type theory \cite{nimaunivalent}, and logic \cite{patterson}.

Given the usefulness of this framework, one might be surprised to find that \emph{there is no agreed-upon notion of what an equivalence of double categories should be}. This is in stark contrast to the case of $\cat$, where we have the \emph{equivalences of categories}: functors which are surjective on objects up to isomorphism, and fully faithful on morphisms. There is also a canonical notion in the category $\twocat$ of (strict) $2$-categories and $2$-functors given by the \emph{biequivalences}: $2$-functors which are surjective on objects up to equivalence, full on morphisms up to $2$-isomorphism, and fully faithful on $2$-morphisms.

\subsection*{Searching for a notion of equivalence} The subtlety with the double categorical setting lies in the fact that if we try to emulate the definitions of ($2$-)categorical equivalences, we are forced to make a choice as to what type of surjectivity we require on objects, since now there are two types of morphisms to consider. One can readily imagine several alternatives: should we ask for essential surjectivity with respect to horizontal morphisms, or vertical? Perhaps we should ask for both; in that case, should we require any special compatibility between the two types of morphisms? And so on.

Depending on the examples or the perspective we have in mind, we can reach different plausible answers to these questions. For instance, one can treat double categories as categories internal to $\cat$ and exploit the categorical equivalences, or as the algebras of a certain 2-monad over the 2-category $\cat(\graph)$; these perspectives lead to different notions of equivalences which are explored in \cite{FPP}. The vast number of meaningful options for what an equivalence of double categories could be is thrilling from a homotopical perspective, but a bit disconcerting from a categorical standpoint, where the ethos is that there should exist a canonical notion of ``sameness'' which is agnostic to any particular point of view. The first goal of this paper is to champion one such notion.

In order to tackle this question, we find it useful to work in a homotopical setting, instead of a purely categorical one. Ever since Quillen's seminal work on homotopical algebra \cite{Quillen}, and even more so in the last two decades, we have come to expect that any reasonable notion of equivalence in a category will form the class of weak equivalences of a model structure. This is indeed the case for the equivalences of categories and the biequivalences of 2-categories, which are captured as the weak equivalences in corresponding model structures \cite{rezk,Lack2Cat,LackBicat}. Hence, we expect the same to be true of any good notion of equivalence of double categories.  

A benefit of working in this setting is that the ambiguity we find with respect to horizontal and vertical morphisms goes away when we try to determine the class of trivial fibrations. In analogy with the trivial fibrations in $\cat$ and $\twocat$, these should be double functors which are surjective on objects, full on both horizontal and vertical morphisms, and fully faithful on squares; we call these \emph{canonical trivial fibrations}. Note that there are no choices to be made here, as surjectivity on objects holds strictly, i.e., up to equality. We then posit that these should be the trivial fibrations that correspond to a ``canonical'' notion of equivalence, and set out to understand all possible model structures on the category $\dblcat$ that share this class of trivial fibrations.

In doing so, we find the \emph{minimal class of equivalences} that is compatible with the canonical trivial fibrations: the \emph{gregarious double equivalences} introduced by Campbell \cite{Camp}. These are the weak equivalences of a model structure, which we prove to be left proper in \cref{gregleftproper}. More precisely, we show that any class of weak equivalences on $\dblcat$ that is part of a model structure whose trivial fibrations are the canonical ones, must contain the gregarious double equivalences. We see this as evidence that this is the least biased notion of double categorical equivalence one can formulate, and should therefore be considered the canonical equivalences. Formally, our result is as follows.

\begin{theoremA}[\cref{thm:localization}] \label{newthmA}
Every model structure on $\dblcat$ whose trivial fibrations are the canonical ones is a localization of the gregarious model structure. If it is further cofibrantly generated, it is in fact a Bousfield localization of the gregarious model structure.
\end{theoremA}

\subsection*{A recipe for constructing model structures}

The gregarious model structure of Campbell is not the only model structure in the literature with this class of trivial fibrations. In \cite{whi}, the authors construct another such model structure on $\dblcat$, that is homotopically compatible with the horizontal inclusion $\bH\colon \twocat\to \dblcat$. Its fibrant objects are the weakly horizontally invariant double categories, which makes it possible to define a nerve functor from double categories to double $(\infty,1)$-categories which is homotopical \cite{lyne}. In her doctoral thesis \cite{MSequipments}, the third-named author constructs a model structure whose fibrant objects are the equipments, which is then used to prove a result on equivalence invariance of formal category theory.

Each of these model structures is carefully constructed in an independent, ad-hoc manner. In hindsight, \cref{newthmA} tells us that these are all instances of Bousfield localizations of the gregarious model structure, giving us a unifying way to understand them as part of the same framework. However, the information encoded in such a localization can prove difficult to describe. Moreover, as is often the case in practice, these homotopy theories arise from a strong grasp of their fibrant objects instead of a complete class of weak equivalences, which gives another reason to look for a different method to obtain them. The second main goal in this paper is to provide efficient, detailed, and user-friendly methods  to construct all possible cofibrantly generated model structures on $\dblcat$ with the canonical trivial fibrations, allowing the user to provide a desired class of fibrant objects.
    
 The key to achieve this is to use the tools introduced by Guetta and the authors in \cite{fibrantly_induced} for constructing model structures. 
 As a first step, we use these methods to provide an alternative proof of the existence of the gregarious model structure due to Campbell \cite{Camp} in \cref{subsec:gregarious}. With this in hand, the general procedure for constructing cofibrantly generated model structures on $\dblcat$ (described in further detail in \cref{subsec:recipe}) goes as follows:

   \begin{enumerate}[leftmargin=0.8cm]
       \item Choose any subset $\cJ$ of canonical cofibrations\footnote{The cofibrations of all of these model structures are determined, since we fixed the trivial fibrations; descriptions can be found in \cref{gen_cofib}.} with canonically cofibrant domain which contains a set of generating trivial cofibrations of the gregarious model structure (see \cref{notn:J0} for such a set).
         \begin{itemize}[leftmargin=0.8cm]
             \item \textit{In practice, how do you choose $\cJ$?} By keeping in mind that the double functors in $\cJ$ will be trivial cofibrations, and they will determine the fibrant objects; that is, a double category $\bA$ will be fibrant if and only if $\bA\to *$ has the right lifting property with respect to the double functors in $\cJ$.
         \end{itemize}
         \item Let $(\an,\nfib)$ denote the weak factorization system of \emph{anodyne extensions} and \emph{naive fibrations} cofibrantly generated by $\cJ$, and check that any of the following equivalent conditions holds:
           \begin{rome}
        \item Every double functor $J\in \an$ between naive fibrant objects is a gregarious double equivalence.
        \item Every gregarious fibration between naive fibrant objects is in $\nfib$. 
        \item For every naive fibrant object $\bA$, the double functor $\cP(\bA)\to \bA\times \bA$ from a path object is in $\nfib$.
    \end{rome}
   \end{enumerate}
   
   \begin{theoremA}[\cref{thm:genmodel}]\label{newthmB}
        For any set $\cJ$ as above, there is a cofibrantly generated model structure on $\dblcat$ whose trivial fibrations are the canonical ones, and whose fibrant objects are determined by having the right lifting property with respect to the double functors in $\cJ$. Moreover, weak equivalences and fibrations between fibrant objects are the gregarious ones.
\end{theoremA}

In practice, this description makes it significantly easier to understand the fibrant objects than the one in \cref{newthmA}, since they are characterized via lifting conditions instead of a local condition---especially when $\cJ$ is finite. 

 Moreover, we prove that this procedure is exhaustive.

\begin{theoremA}[\cref{thm:cleverJ}] Any cofibrantly generated model structure on $\dblcat$ with the canonical trivial fibrations can be constructed using the recipe above.
\end{theoremA}

In \cref{thm:gray}, we give minimal conditions that ensure the model structures are monoidal with respect to the Gray tensor product for double categories, introduced by B\"ohm \cite{Bohm}. 

\subsection*{Applications}

 The rest of the paper is dedicated to the construction of a plethora of examples using this recipe, whose homotopy theories encode a range of relevant $2$-dimensional structures; those without citations are new to this paper. 
 \begin{itemize}[leftmargin=0.8cm]
     \item  $\dblcat_\mathrm{greg}$: the gregarious model structure \cite{Camp}; this is the ``canonical'' model structure for double categories, and is initial among model structures with the canonical trivial fibrations, in the sense that every other such model structure is a localization of it.
     \item $\dblcat_\mathrm{whi}$: the model structure for weakly horizontally invariant double categories \cite{whi}; this is compatible with the first-named author’s double $(\infty, 1)$-categorical nerve construction \cite{LyneNerve}. 
     \item $\dblcat_\mathrm{tr}$: the model structure for transposable double categories; this models the homotopy theory of 2-categories. This model structure was originally proposed in \cite{fibrantly_induced}, but there is a mistake in the proof, which we circumvent here.
     \item $\dblcat_\mathrm{h,eqp}$: the model structure for ``horizontal equipments'' \cite{MSequipments}; this provides a framework for formal category theory.
     \item $\dblcat_\mathrm{tr,ladj}$ (resp.\ $\dblcat_\mathrm{tr,adj}$), whose homotopy theory models $2$-categories whose morphisms all have left (resp.\ both left and right) adjoints.
     \item $\dblcat_\mathrm{tr,gpd}$: the model structure for transposable double groupoids; this models the homotopy theory of $2$-groupoids.\footnote{A $2$-groupoid is a $2$-category in which all morphisms are equivalences and all $2$-morphisms are invertible.}
      \item $\dblcat_{\emptyset\mathrm{/ctr}}$: the model structure for either empty or contractible double categories; this models homotopy $(-1)$-types.
     \item $\dblcat_\mathrm{ctr}$: the model structure for contractible double categories; this models homotopy $(-2)$-types, and is terminal among model structures with the canonical trivial fibrations, in the sense that it is a localization of every other such model structure.
 \end{itemize}

When appropriate, the model structures we listed also admit a transposed version, where the role of the horizontal and vertical morphisms are reversed. As expected, many of these interact with the corresponding model structures on $2$-categories (or $2$-groupoids) via Quillen adjunctions. We summarize their interactions in the non-commutative diagrams below, where red arrows denote right Quillen embeddings, blue arrows are left Quillen embeddings, green arrows are both left and right Quillen equivalences, and brown arrows are right Quillen equivalences.

    \begin{tz}
    \node[](1) {$\dblcat_\mathrm{greg}$}; 
    \node[below of=1,xshift=-3cm](2) {$\dblcat_\mathrm{whi}$}; 
    \node[below of=1,xshift=3cm](2') {$\dblcat_\mathrm{wvi}$}; 
    \node[below of=1](0) {$\twocat$}; 
    \node[below of=0](0') {$\dblcat_\mathrm{tr}$}; 
    \node[below of=0',yshift=-1.5cm](0'') {$\dblcat_\mathrm{tr,adj}$};
    \node[above of=0''](0''') {$\twocat_\mathrm{adj}$};
    \node[below of=2,xshift=-3cm](4) {$\dblcat_\mathrm{h,eqp}$}; 
    \node[below of=2',xshift=3cm](4') {$\dblcat_\mathrm{v,eqp}$}; 
     \node[below of=4,xshift=3cm](5) {$\dblcat_\mathrm{tr,ladj}$}; 
     \node[above of=5](6) {$\twocat_\mathrm{ladj}$};
     \node[below of=4',xshift=-3cm](5') {$\dblcat_\mathrm{tr,radj}$};
     \node[above of=5'](6') {$\twocat_\mathrm{radj}$};

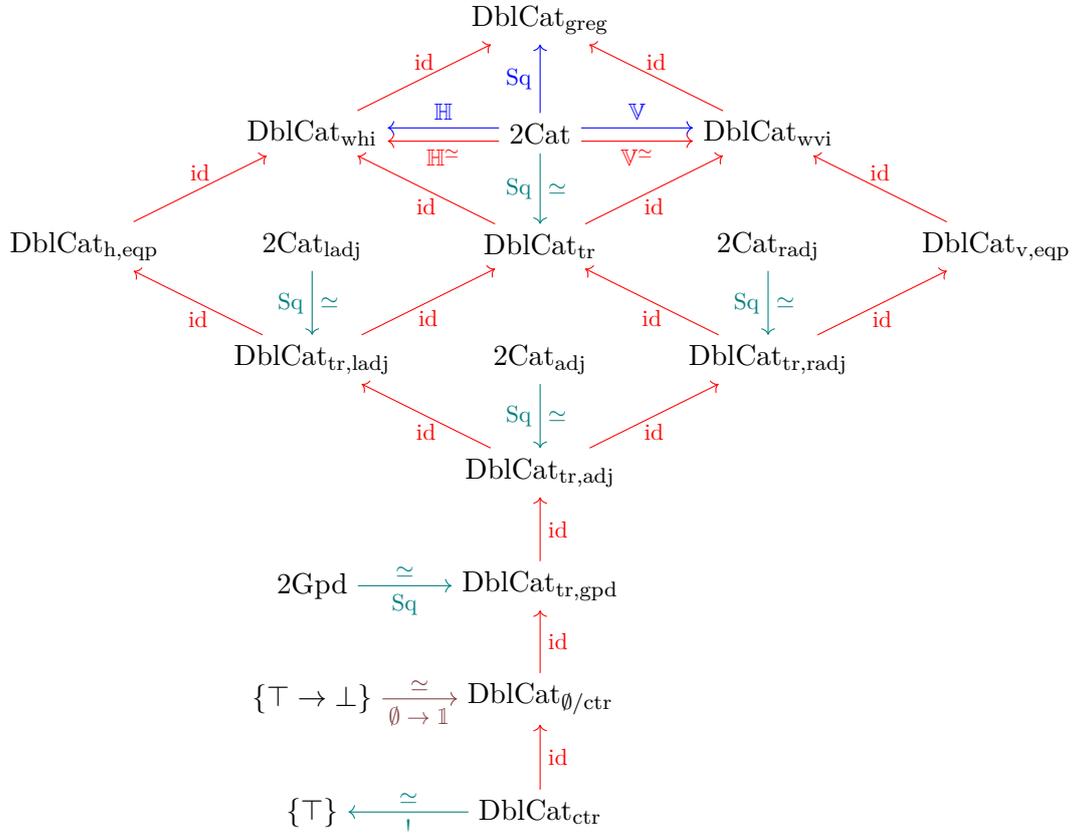
\draw[->,teal](0) to node[left,la]{$\Sq$} node[right,la]{$\simeq$} (0');
\draw[->,blue](0) to node[left,la]{$\Sq$} (1);

\draw[->,teal](6) to node[left,la]{$\Sq$} node[right,la]{$\simeq$} (5);
\draw[->,teal](6') to node[left,la]{$\Sq$} node[right,la]{$\simeq$} (5');
\draw[->,teal](0''') to node[left,la]{$\Sq$} node[right,la]{$\simeq$} (0'');

\draw[->,red]($(0.west)-(0,3pt)$) to node[below,la]{$\bH^{\simeq}$} ($(2.east)-(0,3pt)$);    
\draw[->,red]($(0.east)-(0,3pt)$) to node[below,la]{$\bV^\simeq$} ($(2'.west)-(0,3pt)$);
\draw[->,blue]($(0.west)+(0,2pt)$) to node[above,la]{$\bH$} ($(2.east)+(0,2pt)$);    
\draw[->,blue]($(0.east)+(0,2pt)$) to node[above,la]{$\bV$} ($(2'.west)+(0,2pt)$);

\draw[->,red](2) to node[above,la]{$\id$} (1); 
\draw[->,red](2') to node[above,la]{$\id$} (1); 
\draw[->,red](4) to node[above,la]{$\id$} (2);
\draw[->,red](4') to node[above,la]{$\id$} (2');
\draw[->,red](0') to node[below,la]{$\id$} (2); 
\draw[->,red](0') to node[below,la]{$\id$} (2'); 
\draw[->,red](5) to node[below,la]{$\id$} (0');
\draw[->,red](5') to node[below,la]{$\id$} (0');
\draw[->,red](0'') to node[below,la]{$\id$} (5); 
\draw[->,red](0'') to node[below,la]{$\id$} (5');
\draw[->,red](5) to node[below,la]{$\id$} (4); 
\draw[->,red](5') to node[below,la]{$\id$} (4');  

    \node[below of=0''](0') {$\dblcat_\mathrm{tr,gpd}$}; 
\draw[->,red](0') to node[right,la]{$\id$} (0''); 
    \node[below of=0'',xshift=-3cm](0) {$\twogpd$}; 
    \node[below of=0](01) {$\{\top\to \bot\}$}; 
    \node[below of=01](02) {$\{\top\}$}; 
    \node[below of=0'](0'') {$\dblcat_{\emptyset\mathrm{/ctr}}$};
    \node[below of=0''](0''') {$\dblcat_\mathrm{ctr}$}; 

\draw[->,teal](0) to node[below,la]{$\Sq$} node[above,la]{$\simeq$} (0');
\draw[->,red!50!teal](01) to node[below,la]{$\emptyset\to \mathbbm{1}$} node[above,la]{$\simeq$} (0'');
\draw[->,teal](0''') to node[below,la]{$!$} node[above,la]{$\simeq$} (02);
\draw[->,red](0'') to node[right,la]{$\id$} (0'); 
\draw[->,red](0''') to node[right,la]{$\id$} (0''); 
\end{tz}

\subsection*{Acknowledgments}

This project started while the authors were participants in the program ``Topology, representation theory and higher structures'' at the Isaac Newton Institute, and they are very grateful to the organizers for creating such a great work environment. They are also grateful to John Bourke for pointing out a mistake in the proof of the model structure for transposable double categories in \cite{fibrantly_induced}, whose correction prompted the present work, and to Tom Leinster for stimulating discussions. During the realization of this work, the first author was a member of the Collaborative Research Centre ``SFB 1085: Higher Invariants'' funded by the Deutsche Forschungsgemeinschaft~(DFG), and the second author was partially supported by the National Science Foundation grant DMS-2506116.

\section{Preliminaries: model categories}\label{subsec:MSpreliminaries}

We recall some of the specialized model categorical background that will be used throughout the paper; namely, some useful tools to construct model structures: (Bousfield) localizations in \cref{subsec:loc}, and transfer along an adjunction in \cref{subsec:transfer}. For a reader looking for an introduction to model categories, we recommend \cite{Hovey}.

\subsection{Localization}\label{subsec:loc}
We start by recalling some definitions regarding localizations of model structures, a process that allows us to add more maps to the class of weak equivalences under certain conditions.

\begin{defn}
    Let $\cM$ be a model category. A \emph{localization} of $\cM$ is another model structure $\cM_\mathrm{loc}$ on $\cM$ with the same trivial fibrations such that the identity adjunction
    \begin{tz}
\node[](1) {$\cM$}; 
\node[right of=1,xshift=1cm](2) {$\cM_\mathrm{loc}$};

\draw[->,bend left=20] ($(1.east)+(0,4pt)$) to node[above,la]{$\id$} ($(2.west)+(0,4pt)$);
\draw[->,bend left=20] ($(2.west)-(0,4pt)$) to node[below,la]{$\id$} ($(1.east)-(0,4pt)$);

\node[la] at ($(1.east)!0.5!(2.west)$) {$\bot$};
    \end{tz}
    is a Quillen pair whose derived counit is a weak equivalence in $\cM_\mathrm{loc}$. 
\end{defn}

Given a set $S$ of morphisms in a model category $\cM$, we recall how to define the left Bousfield localization of $\cM$ at $S$. 

\begin{defn}
    An object $X$ in $\cM$ is \emph{$S$-local} if it is fibrant in $\cM$ and, for every morphism $s\colon A\to B$ in $S$, the induced map between derived homs
    \[ s^*\colon \bR\Hom_\cM(B,X)\to \bR\Hom_\cM(A,X)\]
    is a weak homotopy equivalence.

    A morphism $f\colon A\to B$ in $\cM$ is an \emph{$S$-local equivalence} if, for every $S$-local object $X$ in~$\cM$, the induced map between derived homs
    \[ f^*\colon \bR\Hom_\cM(B,X)\to \bR\Hom_\cM(A,X)\]
    is a weak homotopy equivalence.
\end{defn}

\begin{defn}
    If it exists, the \emph{left Bousfield localization} of $\cM$ at $S$ is the model structure on $\cM$ such that
    \begin{itemize}[leftmargin=0.8cm]
        \item a morphism is a trivial fibration if and only if it is a trivial fibration in $\cM$, 
        \item a morphism is a weak equivalence if and only if it is an $S$-local weak equivalence. 
    \end{itemize}
\end{defn}

\begin{theorem} \label{thm:existleftBousfield}
    Let $\cM$ be a left proper, combinatorial model structure and $S$ be a set of morphisms in $\cM$. Then the left Bousfield localization of $\cM$ at $S$ exists and is such that 
    \begin{itemize}[leftmargin=0.8cm]
        \item an object $X$ is fibrant if and only if $X$ is $S$-local, 
        \item a morphism $f$ between $S$-local objects is a weak equivalence (resp.~fibration) if and only if it is a weak equivalence (resp.~fibration) in $\cM$. 
    \end{itemize}
\end{theorem}

In the combinatorial case, any localization of a model category $\cM$ can be obtained as a left Bousfield localization. A proof for the simplicial case is given in \cite[Proposition A.3.7.4]{HTT}; we present a non-simplicial version below, which will be required for our purposes.  

\begin{defn} \label{def:anodyne}
    A set $\cJ$ of morphisms in a model category $\cM$ is a set of \emph{generating anodyne extensions} if the following conditions hold:
    \begin{itemize}[leftmargin=0.8cm]
        \item an object $X$ is fibrant in $\cM$ if and only if it has the right lifting property with respect to every morphism in $\cJ$, 
        \item a morphism $f$ between fibrant objects is a fibration in $\cM$ if and only if $f$ has the right lifting property with respect to every morphism in $\cJ$. 
    \end{itemize}
\end{defn}

\begin{prop} \label{prop:locvsBousfieldloc}
    Let $\cM$ be a left proper, combinatorial model category, and suppose that $\cM_\mathrm{loc}$ is a localization of $\cM$ which is cofibrantly generated. Then $\cM_\mathrm{loc}$ coincides with the left Bousfield localization of $\cM$ at any set $\cJ$ of generating anodyne extensions of $\cM_\mathrm{loc}$. 
\end{prop}

\begin{proof}
Since the localization $\cM_\mathrm{loc}$ and the left Bousfield localization $\cM_\cJ$ of $\cM$ at $\cJ$ have the same trivial fibrations, namely those of $\cM$, it suffices to show that they have the same fibrant objects by \cite[Proposition E.1.10]{JoyalVolumeII}.

Suppose that $X$ is fibrant in $\cM_\mathrm{loc}$. Then, in particular, it is fibrant in $\cM$. Since every morphism $j\colon A\to B$ in $\cJ$ is a weak equivalence in $\cM_\mathrm{loc}$, the induced map between derived homs
    \[ j^*\colon \bR\Hom_{\cM_\mathrm{loc}}(B,X)\to \bR\Hom_{\cM_\mathrm{loc}}(A,X) \] 
    is a weak homotopy equivalence. Now, since the identity adjunction is a Quillen pair between $\cM_\mathrm{loc}$ and $\cM$, this map is weakly homotopy equivalent to the map
    \[ j^*\colon \bR\Hom_{\cM}(B,X)\to \bR\Hom_{\cM}(A,X), \]
    which is therefore a weak homotopy equivalence as well. This shows that $X$ is $\cJ$-local. 
    
    Now suppose that $X$ is $\cJ$-local. Since every cofibration in $\cJ$ is a $\cJ$-local equivalence, then $X$ has the right lifting property with respect to all generating anodyne extensions in $\cM_\mathrm{loc}$. This shows that $X$ is fibrant in $\cM_\mathrm{loc}$ and concludes the proof.   
\end{proof}

\subsection{Transfer}\label{subsec:transfer}

We now recall some available methods to transfer a model category structure along an adjunction. 

\begin{defn} Consider an adjunction as below
    \begin{tz}
\node[](1) {$\cM$}; 
\node[right of=1,xshift=.9cm](2) {$\cN$};

\draw[->,bend left=20] ($(1.east)+(0,4pt)$) to node[above,la]{$F$} ($(2.west)+(0,4pt)$);
\draw[->,bend left=20] ($(2.west)-(0,4pt)$) to node[below,la]{$U$} ($(1.east)-(0,4pt)$);

\node[la] at ($(1.east)!0.5!(2.west)$) {$\bot$};
    \end{tz}

\begin{enumerate}[leftmargin=0.8cm]
\item Let $\cM$ be a model structure. When it exists, the \emph{right-transferred model structure} on $\cN$ is such that a morphism $f$ is a weak equivalence (resp.~fibration) if and only if $Uf$ is a weak equivalence (resp.~fibration) in $\cM$.

\item Let $\cN$ be a model structure. When it exists, the \emph{left-transferred model structure} on $\cM$ is such that a morphism $f$ is a weak equivalence (resp.~cofibration) if and only if $Ff$ is a weak equivalence (resp.~cofibration) in $\cN$. 
\end{enumerate}
\end{defn}

\begin{rem}
    If the corresponding transferred model structure exists, then the original adjunction $F\dashv U$ becomes a Quillen pair. 
\end{rem}

The following result gives a criterion to recognize transferred model structures.

\begin{prop} \label{rem:inductionQE}
    Let $\cM$ and $\cN$ be model categories and $U\colon \cN\to \cM$ be a functor which is both a left and a right Quillen equivalence and admits a retraction $R\colon \cM\to \cN$. Then the model structure on $\cN$ is both left- and right-transferred from that on $\cM$ along~$U$. 
\end{prop}

\begin{proof}
    The proof works as in \cite[Lemma 4.15]{MoserNuiten}, where the functor $U$ is assumed to be fully faithful rather than have a retraction. Inspecting the proof, it suffices to show that $U$ reflects fibrations under the current hypotheses. If $X\to Y$ is a morphism in $\cN$ such that $UX\to UY$ is a fibration in $\cM$, then $UX\to UY$ has the right lifting property with respect to $UA\to UB$ for every trivial cofibration $A\to B$ in $\cN$, as $U\colon \cN\to \cM$ preserves trivial cofibrations. Applying the retraction $R\colon \cM\to \cN$ then shows that $X\to Y$ has the right lifting property with respect to every trivial cofibration $A\to B$ in $\cN$, and so it is a fibration.
\end{proof}

Finally, this last result relating left Bousfield localizations and transfer appears as \cite[Proposition 2.13]{GuettaMoser}. 

\begin{prop} \label{rem:fibindloc}
Let $\cM$ and $\cN$ be left proper, combinatorial model categories. Suppose that $\cN$ is right-transferred from $\cM$ along an adjunction 
    \begin{tz}
\node[](1) {$\cM$}; 
\node[right of=1,xshift=.9cm](2) {$\cN$};

\draw[->,bend left=20] ($(1.east)+(0,4pt)$) to node[above,la]{$F$} ($(2.west)+(0,4pt)$);
\draw[->,bend left=20] ($(2.west)-(0,4pt)$) to node[below,la]{$U$} ($(1.east)-(0,4pt)$);

\node[la] at ($(1.east)!0.5!(2.west)$) {$\bot$};
\end{tz}
Let $S$ be a set of morphisms in $\cM$ between cofibrant objects, and denote by $\cM_S$ and $\cN_{F(S)}$ the left Bousfield localizations of $\cM$ and~$\cN$ at $S$ and $F(S)$, respectively. If $F\dashv U$ is a Quillen equivalence between $\cM$ and $\cN$, then $\cN_{F(S)}$ is right-transferred from $\cM_S$ along the adjunction $F\dashv U$.
\end{prop}

\section{Preliminaries: 2-dimensional categories}\label{subsec:2dimbackground}

We now recall some key facts about $2$-categories and $2$-groupoids in \cref{lackMS}
(mostly in regards to their homotopy theory), describe the many ways to relate $2$-categories and double categories in \cref{subsec:2catvsdblcat}, the monoidal structures available for double categories in \cref{subsec:monoidalbackground}, and finally, we give a description of pushouts of double categories in \cref{subsec:pushoutsdblcat}. The reader familiar with $2$-dimensional structures is encouraged to skip to \cref{section:proper} and refer back to this material as needed.

\subsection{$2$-categorical structures and their homotopy theory}\label{lackMS} 
As mentioned in the introduction, the notion of equivalence of categories admits a $2$-categorical counterpart, which form the weak equivalences of a model structure on the category of $2$-categories. In order to define some of these relevant classes of $2$-functors, we recall what it means for a morphism in a $2$-category to be an equivalence, using this instance to establish our preferred notation.

\begin{notation}
We use the following notation.
\begin{itemize}[leftmargin=0.8cm]
\item We write $\cat$ for the category of categories and functors. 
\item We write $\twocat$ for the category of $2$-categories and $2$-functors. We refer the reader to \cite{JohYau} for a full account on $2$-categories.
\item We write $U\colon \twocat\to \cat$ for the functor that sends a $2$-category to its underlying category. 
\end{itemize}
\end{notation}

\begin{defn}\label{def:equivdata}
    A morphism $f\colon A\to B$ in a $2$-category $\cA$ is an \emph{equivalence} if there is a morphism $g\colon B\to A$ in $\cA$ and two $2$-isomorphisms $\eta\colon \id_A\cong gf$ and $\varepsilon\colon fg\cong \id_B$ in $\cA$. We refer to $(f,g,\eta,\varepsilon)$ as an \emph{equivalence data}. It is called an \emph{adjoint equivalence data} if the $2$-isomorphisms $\eta$ and $\varepsilon$ further satisfy the triangle identities. 
\end{defn}

\begin{defn}\label{defn:bieq}
A $2$-functor $F\colon \cA\to \cB$ is a \emph{biequivalence} if it is
\begin{enumerate}
    \item[(b1)] \emph{surjective on objects up to equivalence:} for every object $B\in \cB$, there is an object $A\in \cA$ and an equivalence $FA\xrightarrow{\simeq} B$ in~$\cB$,
    \item[(b2)] \emph{essentially full on morphisms:} for every pair of objects $A,C\in \cA$ and every morphism $g\colon FA\to FC$ in $\cB$, there is a morphism $f\colon A\to C$ in $\cA$ and a $2$-isomorphism $g\cong Ff$ in $\cB$, and
    \item[(b3)] \emph{fully faithful on $2$-morphisms:} for every pair of morphisms $f,h\colon A\to C$ in $\cA$ and every $2$-cell $\beta\colon Ff\Rightarrow Fh$ in $\cB$, there is a unique $2$-cell $\alpha\colon f\Rightarrow h$ in $\cA$ such that $F\alpha=\beta$.
\end{enumerate}
\end{defn}

\begin{defn}\label{Lackfib}
A $2$-functor $F\colon \cA\to \cB$ is an \emph{equifibration} if
\begin{enumerate}
    \item[(f1)] for every equivalence $g\colon FA\xrightarrow{\simeq} B$ in $\cB$, there is an equivalence $f\colon A\xrightarrow{\simeq} C$ in $\cA$ such that $Ff=g$, and
    \item[(f2)] it is \emph{locally an isofibration:} for every morphism $f\colon A\to C$ in $\cA$ and every $2$-isomorphism $\beta\colon Ff\cong g$ in $\cB$, there is a $2$-isomorphism $\alpha\colon f\cong h$ in $\cA$ such that $F\alpha=\beta$. 
\end{enumerate}
\end{defn}

The following result is a combination of \cite[Theorem 6.3]{Lack2Cat} and \cite[Theorem 4]{LackBicat}.

\begin{theorem}\label{thm:lackMS}
    There is a cofibrantly generated, left proper model structure on $\twocat$ in which the weak equivalences are the biequivalences and the trivial fibrations are the $2$-functors that are surjective on objects, full on morphisms, and fully faithful on $2$-morphisms. Furthermore, the fibrations are the equifibrations.
\end{theorem}

The model structure above can be restricted to the full subcategory of $\twocat$ consisting of all $2$-groupoids. Recall that a (1-)category is a groupoid if all morphisms are invertible. The $2$-categorical notion is analogous, requiring that both $1$- and $2$-morphisms are invertible in the weakest possible sense.

\begin{defn}
    A \emph{$2$-groupoid} is a $2$-category in which all morphisms are equivalences and all $2$-morphisms are invertible.

    We denote by $\twogpd$ the full subcategory of $\twocat$ spanned by the $2$-groupoids.
\end{defn}

The following result can be found as \cite[Theorem 3.7]{Lack2Cat}, where the notion that we refer to as 2-groupoids are called ``pseudogroupoids''.

\begin{theorem}
    There is a cofibrantly generated model structure on $\twogpd$ in which the weak equivalences are the biequivalences and the trivial fibrations are the $2$-functors that are surjective on objects, full on morphisms, and fully faithful on $2$-morphisms. Furthermore, the fibrations are the equifibrations.
\end{theorem}

Finally, we observe the following relation between the model structures on $\twogpd$ and~$\twocat$. 

\begin{prop} \label{MS2gpdvs2cat}
    The inclusion functor $\iota\colon \twogpd\to \twocat$ is fully faithful and admits both a left and a right adjoint.
    \begin{tz}
\node[](1) {$\twogpd$}; 
\node[right of=1,xshift=1.4cm](2) {$\twocat$};

\draw[->,bend right=35] ($(2.west)+(0,8pt)$) to node[above,la]{$\mathrm{loc}$} ($(1.east)+(0,8pt)$);
\draw[->] ($(1.east)$) to node[over,la]{$\iota$} ($(2.west)$);
\draw[->,bend left=35] ($(2.west)-(0,8pt)$) to node[below,la]{$\mathrm{core}$} ($(1.east)-(0,8pt)$);

\node[la] at ($(1.east)!0.5!(2.west)+(0,10pt)$) {$\bot$};
\node[la] at ($(1.east)!0.5!(2.west)-(0,10pt)$) {$\bot$};
\end{tz}
Moreover, both adjunctions are Quillen pairs, and the derived counit of $\mathrm{loc}\dashv \iota$ and the derived unit of $\iota\dashv \mathrm{core}$ are biequivalences in $\twogpd$.
\end{prop}

\begin{proof}
   The first claim regarding the adjunctions and fully faithfulness of $\iota$ is a standard verification. Furthermore, it is straightforward to see that the inclusion $\iota\colon \twogpd\to \twocat$ preserves cofibrations, fibrations, and weak equivalences. The derived unit of $\iota\dashv \mathrm{core}$ coincides with the unit, as all objects in $\twocat$ are fibrant, and so it is an isomorphism since $\iota$ is fully faithful. The statement for the derived counit of $\mathrm{loc}\dashv \iota$ then follows, as $\iota$ is homotopically fully faithful. 
\end{proof}

\subsection{Connections between $2$-categories and double categories}\label{subsec:2catvsdblcat}

The model structures we aim to construct on double categories bear a close connection to the homotopy theory of the $2$-categorical structures introduced in the previous section. With this in mind, we now recall several adjunctions that will later allow us to make this precise.

\begin{notation}
    We denote by $\dblcat$ the category of double categories and double functors. We refer the reader to \cite{Grandis} for a full account on double categories.
\end{notation}

\begin{notation} 
We will use the same notation as in \cite[\S 2]{MSV} to describe the data of a double category. If $\alpha,\beta$ are squares in a double category, we denote by $\alpha\vert\beta$ their horizontal composition, and by $\frac{\alpha}{\beta}$ their vertical composition, provided these are defined.
\end{notation}

\begin{notation}
    We say that a square $\alpha$ in a double category $\bA$ is \emph{horizontal globular} if its source and target vertical morphisms are identities, as depicted below left. We say that it is \emph{vertical globular} if its source and target horizontal morphisms are identities, as depicted below right.
    \begin{tz}
\node[](1) {$A$}; 
\node[below of=1](1') {$A$}; 
\node[right of=1](2) {$B$}; 
\node[below of=2](2') {$B$}; 

\draw[d,pro](1) to (1');
\draw[d,pro](2) to (2');
\draw[->](1) to node[above,la]{$f$} (2);
\draw[->](1') to node[below,la]{$g$} (2');

\node[la] at ($(1)!0.5!(2')$) {$\alpha$};

\node[right of=2,xshift=2cm](1) {$A$}; 
\node[below of=1](1') {$A'$}; 
\node[right of=1](2) {$A$}; 
\node[below of=2](2') {$A'$}; 

\draw[->,pro](1) to node[left,la]{$u$} (1');
\draw[->,pro](2) to node[right,la]{$v$} (2');
\draw[d](1) to (2);
\draw[d](1') to (2');

\node[la] at ($(1)!0.5!(2')$) {$\alpha$};
\end{tz}
\end{notation}

There are different adjunctions relating $\twocat$ and $\dblcat$, arising from the different ways of including $2$-categories into double categories.

\begin{defn}\label{def:bbH}
    The \emph{horizontal embedding} $\bH\colon \twocat\to \dblcat$ sends a $2$-category~$\cA$ to the double category $\bH\cA$ having the same objects as $\cA$, the morphisms of $\cA$ as horizontal morphisms, only identities as vertical morphisms, and squares given by the $2$-morphisms in $\cA$.

    Similarly, the \emph{vertical embedding} $\bV\colon \twocat\to \dblcat$ is obtained by reversing the role of the horizontal and vertical morphisms in the above description. 
\end{defn}

\begin{prop}
    The horizontal and vertical embeddings are fully faithful and admit both a left and a right adjoint.
    \begin{tz}
\node[](1) {$\twocat$}; 
\node[right of=1,xshift=1.5cm](2) {$\dblcat$};

\draw[->,bend right=35] ($(2.west)+(0,8pt)$) to node[above,la]{$L_H$} ($(1.east)+(0,8pt)$);
\draw[->] ($(1.east)$) to node[over,la]{$\bH$} ($(2.west)$);
\draw[->,bend left=35] ($(2.west)-(0,8pt)$) to node[below,la]{$\bfH$} ($(1.east)-(0,8pt)$);

\node[la] at ($(1.east)!0.5!(2.west)+(0,10pt)$) {$\bot$};
\node[la] at ($(1.east)!0.5!(2.west)-(0,10pt)$) {$\bot$};

\node[right of=2,xshift=2cm](1) {$\twocat$}; 
\node[right of=1,xshift=1.5cm](2) {$\dblcat$};

\draw[->,bend right=35] ($(2.west)+(0,8pt)$) to node[above,la]{$L_V$} ($(1.east)+(0,8pt)$);
\draw[->] ($(1.east)$) to node[over,la]{$\bV$} ($(2.west)$);
\draw[->,bend left=35] ($(2.west)-(0,8pt)$) to node[below,la]{$\bfV$} ($(1.east)-(0,8pt)$);

\node[la] at ($(1.east)!0.5!(2.west)+(0,10pt)$) {$\bot$};
\node[la] at ($(1.east)!0.5!(2.west)-(0,10pt)$) {$\bot$};
\end{tz}
\end{prop}
\begin{proof}
    The adjoint pair $\bH\dashv\bfH$ is established for instance in \cite[Proposition 2.5]{FPP}. Furthermore, the fact that $\bH$ admits a left adjoint is ensured by
the Adjoint Functor Theorem, since $\bH$ preserves all limits and colimits, and the categories involved are locally presentable. The statements for $\bV$ are analogous.
\end{proof}

\begin{rem} \label{underlyinghorver2cat}
    The right adjoint functor $\bfH\colon\dblcat\to\twocat$ takes a double category $\bA$ to its \emph{underlying horizontal $2$-category} $\bfH\bA$, whose objects are the objects of $\bA$, whose morphisms are the horizontal morphisms of $\bA$, and whose $2$-morphisms are given by the horizontal globular squares in $\bA$.

    The right adjoint functor $\bfV\colon \dblcat\to \twocat$ can be described similarly, sending a double category to its \emph{underlying vertical $2$-category}.
\end{rem}

A variation of this construction can be obtained by allowing equivalences rather than just identities in the relevant direction. 

\begin{defn}
    The \emph{homotopical horizontal embedding}  $\bH^\simeq\colon \twocat\to  \dblcat$ sends a $2$-category $\cA$ to the double category $\bH^\simeq\cA$ whose objects are the objects of $\cA$, whose horizontal morphisms are the morphisms of $\cA$, whose vertical morphisms are the adjoint equivalence data in $\cA$, and whose squares 
    \begin{tz}
\node[](1) {$A$}; 
\node[below of=1](2) {$A'$}; 
\node[right of=1](3) {$B$}; 
\node[right of=2](4) {$B'$};

\draw[->,pro] (1) to node[left,la]{$(u,v,\eta,\varepsilon)$} (2); 
\draw[->] (1) to node[above,la]{$f$} (3); 
\draw[->] (2) to node[below,la]{$f'$} (4); 
\draw[->,pro](3) to node[right,la]{$(u',v',\eta',\varepsilon')$} (4); 
 
\node[la] at ($(1)!0.5!(4)$) {$\alpha$};
\end{tz}
    are the $2$-morphisms $\alpha\colon u'f\Rightarrow f'u$ in $\cA$. 

    Similarly, the \emph{homotopical vertical embedding}  $\bV^\simeq\colon \twocat\to  \dblcat$ is obtained by reversing the role of the horizontal and vertical morphisms. 
\end{defn}

While these functors are not left adjoints (see \cite[Remark 2.17]{whi}), they are right adjoints by \cite[Proposition 2.15]{whi}. 

\begin{prop}
    The homotopical horizontal and vertical embeddings admit left adjoints.
   \begin{tz}
\node[](1) {$\twocat$}; 
\node[right of=1,xshift=1.5cm](2) {$\dblcat$};

\draw[->,bend left=20] ($(1.east)+(0,4pt)$) to node[above,la]{$L_{H}^\simeq$} ($(2.west)+(0,4pt)$);
\draw[->,bend left=20] ($(2.west)-(0,4pt)$) to node[below,la]{$\bH^\simeq$} ($(1.east)-(0,4pt)$);

\node[la] at ($(1.east)!0.5!(2.west)$) {$\bot$};

\node[right of=2,xshift=1.5cm](1) {$\twocat$}; 
\node[right of=1,xshift=1.5cm](2) {$\dblcat$};

\draw[->,bend left=20] ($(1.east)+(0,4pt)$) to node[above,la]{$L_{V}^\simeq$} ($(2.west)+(0,4pt)$);
\draw[->,bend left=20] ($(2.west)-(0,4pt)$) to node[below,la]{$\bV^\simeq$} ($(1.east)-(0,4pt)$);

\node[la] at ($(1.east)!0.5!(2.west)$) {$\bot$};
\end{tz}
\end{prop}

Finally, the last inclusion we will study allows all morphisms of the $2$-category to populate both directions of morphisms in the double category.

\begin{defn} \label{def:square}
    The \emph{square functor}  $\Sq\colon \twocat\to \dblcat$ sends a $2$-category $\cA$ to the double category $\Sq\cA$ whose objects are the objects of $\cA$, whose horizontal and vertical morphisms are the morphisms of $\cA$, and whose squares 
\begin{tz}
\node[](1) {$A$}; 
\node[below of=1](2) {$A'$}; 
\node[right of=1](3) {$B$}; 
\node[right of=2](4) {$B'$};

\draw[->,pro] (1) to node[left,la]{$u$} (2); 
\draw[->] (1) to node[above,la]{$f$} (3); 
\draw[->] (2) to node[below,la]{$f'$} (4); 
\draw[->,pro](3) to node[right,la]{$v$} (4); 
 
\node[la] at ($(1)!0.5!(4)$) {$\alpha$};
\end{tz}
are the $2$-morphisms $\alpha\colon vf\Rightarrow f'u$ in $\cA$.
\end{defn}

The following result combines  \cite[Theorem 1.7(a)]{GPAdjointsDblCats} and \cite[Proposition 8]{EhresmannIV}.

\begin{prop} \label{adj:sq}
    The square functor admits both a left and a right adjoint.
    \begin{tz}
\node[](1) {$\twocat$}; 
\node[right of=1,xshift=1.5cm](2) {$\dblcat$};

\draw[->,bend right=35] ($(2.west)+(0,8pt)$) to node[above,la]{$L_{Sq}$} ($(1.east)+(0,8pt)$);
\draw[->] ($(1.east)$) to node[over,la]{$\Sq$} ($(2.west)$);
\draw[->,bend left=35] ($(2.west)-(0,8pt)$) to node[below,la]{$\bfR$} ($(1.east)-(0,8pt)$);

\node[la] at ($(1.east)!0.5!(2.west)+(0,10pt)$) {$\bot$};
\node[la] at ($(1.east)!0.5!(2.west)-(0,10pt)$) {$\bot$};
\end{tz}
\end{prop}

In later sections we will use a description of the right adjoint of the square functor. This requires the following notion, which will also be of central importance in the examples of \cref{section:ex1,section:ex2}.

\begin{defn}\label{def:companion}
    A \emph{companion pair} in a double category $\bA$ is the data of a functor $\Sq\mathbbm{2}\to \bA$, where $\mathbbm{2}$ denotes the walking morphism. It consists of a tuple $(f,u,\varphi,\psi)$ of a horizontal morphism $f\colon A\to B$, a vertical morphism $u\colon A\arrowdot B$, and two squares $\varphi$ and $\psi$ in $\bA$ of the form
\begin{tz}
\node[](1) {$A$}; 
\node[below of=1](2) {$B$}; 
\node[right of=1](3) {$B$}; 
\node[right of=2](4) {$B$};

\draw[->,pro] (1) to node[left,la]{$u$} (2); 
\draw[->] (1) to node[above,la]{$f$} (3); 
\draw[d] (2) to (4); 
\draw[d,pro](3) to (4); 
 
\node[la] at ($(1)!0.5!(4)$) {$\varphi$};

\node[right of=3,xshift=2cm](1) {$A$}; 
\node[below of=1](2) {$A$}; 
\node[right of=1](3) {$A$}; 
\node[right of=2](4) {$B$};

\draw[->,pro] (3) to node[right,la]{$u$} (4); 
\draw[->] (2) to node[below,la]{$f$} (4); 
\draw[d] (1) to (3); 
\draw[d,pro](1) to (2); 
 
\node[la] at ($(1)!0.5!(4)$) {$\psi$};
\end{tz}
    such that $\psi \vert \varphi=e_f$ and $\frac{\psi}{\varphi}=\id_u$.
\end{defn}

\begin{ex} \label{rem:Sqnaivefib}
Given a $2$-category $\cA$, the double category $\Sq\cA$ is such that every horizontal and vertical morphism is part of a companion pair. Indeed, every horizontal or vertical morphism in $\Sq\cA$ is represented by a morphism $f$ in $\cA$ and the tuple $(f,f,\id_f,\id_f)$ is a companion pair.
\end{ex}

\begin{rem} \label{rightadjofSq}
    The right adjoint functor $\bfR\colon\dblcat\to\twocat$ sends a double category~$\bA$ to the $2$-category $\bfR\bA$ whose objects are the objects of $\bA$, whose morphisms are the companion pairs in $\bA$, and whose $2$-morphisms $\alpha\colon (f,u,\varphi,\psi)\Rightarrow (f',u',\varphi',\psi')$ are the horizontal globular squares in $\bA$ of the form 
\begin{tz}
\node[](1) {$A$}; 
\node[below of=1](1') {$A$}; 
\node[right of=1](2) {$B$}; 
\node[below of=2](2') {$B$}; 

\draw[d,pro](1) to (1');
\draw[d,pro](2) to (2');
\draw[->](1) to node[above,la]{$f$} (2);
\draw[->](1') to node[below,la]{$f'$} (2');

\node[la] at ($(1)!0.5!(2')$) {$\alpha$};
\end{tz}
\end{rem}

We conclude this subsection with a useful technical result about lifting adjoint equivalences. In order to do so, we first recall the following definitions.

\begin{defn}
    A \emph{horizontal (adjoint) equivalence} in a double category $\bA$ is an (adjoint) equivalence in the underlying horizontal $2$-category $\bfH\bA$. Analogously, we define \textit{(adjoint) vertical equivalences} using the underlying vertical $2$-category $\bfV\bA$.
\end{defn}

\begin{rem}
Note that specifying an adjoint horizontal equivalence data in a double category $\bA$ amounts to specifying a double functor $\bH\Eadj\to\bA$, where $\Eadj$ denotes the $2$-category freely generated by an adjoint equivalence.
\end{rem}

\begin{prop}\label{rmk:companionsforequivs}
    Let $\bA$ be a double category and let $(f,g,\eta,\epsilon)$ be a horizontal adjoint equivalence data in $\bA$. If $f$ and $g$ have companions, then there is a lift in the following diagram.
    \begin{tz}
\node[](1) {$\bH \Eadj$}; 
\node[below of=1](2) {$\Sq \Eadj$}; 
\node[right of=1,xshift=.5cm](1') {$\bA$}; 
\draw[->] (1) to node[above,la]{$(f,g,\eta,\epsilon)$} (1'); 
\draw[->] (1) to (2); 
\draw[->,dashed] (2) to (1');
\end{tz}
\end{prop}

\begin{proof} 
Suppose that $(f,g,\eta,\epsilon)$ is a horizontal adjoint equivalence between objects $A$ and~$B$ in $\bA$, and let $(f,u,\varphi,\psi)$ and $(g,v,\chi,\omega)$ be companion pairs in $\bA$. We claim that the vertical morphisms $u$ and $v$ form a vertical adjoint equivalence, by constructing a unit and a counit which are compatible with the rest of the data. 

To do this, note that the vertical identity $e_A$ is a companion of the horizontal identity~$\id_A$, and the vertical composite $vu$ is a companion of the horizontal composite $gf$. Then, since $\eta\colon \id_A\cong gf$,  by the universal property of companions from \cite[\textsection 4.1.1]{Grandis} we have that the composite $vu$ is also a companion of $\id_A$ and so there is a unique $2$-isomorphism $\alpha\colon e_A\cong vu$ in $\bfV\bA$ compatible with the companion data $(\varphi,\psi)$ and $(\chi,\omega)$. Similarly, we get a $2$-isomorphism $\beta\colon uv\cong e_B$ compatible with the companion data. As $(\eta,\epsilon)$ satisfy the triangle identities, by construction so do $(\alpha,\beta)$. This determines a double functor $\Sq\Eadj\to \bA$ as the above defined data generates it.
\end{proof}

\subsection{Monoidal structures on double categories}\label{subsec:monoidalbackground}

Similarly to the case of 2-categories, $\dblcat$ admits two symmetric monoidal structures which we now briefly recall: the cartesian product, and the Gray tensor product. The following is due to \cite{TwoEhresmann}.

\begin{prop}
    The category $\dblcat$ is cartesian closed: for every double category~$\bA$, there is an adjunction 
    \begin{tz}
\node[](1) {$\dblcat$}; 
\node[right of=1,xshift=1.7cm](2) {$\dblcat$};

\draw[->,bend left=20] ($(1.east)+(0,4pt)$) to node[above,la]{$\bA\times (-)$} ($(2.west)+(0,4pt)$);
\draw[->,bend left=20] ($(2.west)-(0,4pt)$) to node[below,la]{$\llbracket \bA,-\rrbracket$} ($(1.east)-(0,4pt)$);

\node[la] at ($(1.east)!0.5!(2.west)$) {$\bot$}; 
\end{tz}
\end{prop}

\begin{rem}\label{internalhom}
    The right adjoint functor $\llbracket\bA,-\rrbracket\colon \dblcat\to \dblcat$ sends a double category $\bB$ to the internal hom double category $\llbracket\bA,\bB\rrbracket$ of double functors $\bA\to \bB$, horizontal and vertical natural transformations between them, and modifications. See \cite[\S 3.2.7]{Grandis} for a detailed description. 
\end{rem}

The monoidal structure we will be most interested in is given by B\"ohm's Gray tensor product, whose construction can be found in \cite{Bohm}. This monoidal structure is an enhancement of the (pseudo) Gray tensor product of $2$-categories to the double categorical setting.

\begin{prop}\label{prop:graypseudo}
    There is a symmetric monoidal structure \[ \otimes\colon \dblcat\times \dblcat\to \dblcat, \] 
    called the \emph{Gray tensor product}. Moreover, it is closed: for every double category $\bA$, there is an adjunction 
    \begin{tz}
\node[](1) {$\dblcat$}; 
\node[right of=1,xshift=1.7cm](2) {$\dblcat$};

\draw[->,bend left=20] ($(1.east)+(0,4pt)$) to node[above,la]{$\bA\otimes (-)$} ($(2.west)+(0,4pt)$);
\draw[->,bend left=20] ($(2.west)-(0,4pt)$) to node[below,la]{$\llbracket \bA,-\rrbracket_\mathrm{ps}$} ($(1.east)-(0,4pt)$);

\node[la] at ($(1.east)!0.5!(2.west)$) {$\bot$}; 
\end{tz}
\end{prop}

\begin{rem}\label{internalpseudohom}
    The right adjoint functor $\llbracket\bA,-\rrbracket_\mathrm{ps}\colon \dblcat\to \dblcat$ sends a double category $\bB$ to the internal pseudo-hom double category $\llbracket\bA,\bB\rrbracket_\mathrm{ps}$ of double functors $\bA\to \bB$, horizontal and vertical pseudo-natural transformations between them, and modifications. A full description can be found in \cite[\S 3.8]{Grandis} or \cite[\S 2.2]{Bohm}. 
\end{rem}

\subsection{Pushouts of double categories}\label{subsec:pushoutsdblcat}

In this section we recall some definitions and results from \cite[\S3 and \S4]{FPP}. The end goal is to have notions of free and quotient double categories that allow for an explicit description of colimits in $\dblcat$.

All graphs here are assumed to be directed.

\begin{defn}\label{defn:double_graph}
    A \emph{double graph} is an internal graph in the category of graphs. It consists of sets of objects, horizontal edges, vertical edges, and squares, each equipped with source and target maps.

    A \emph{morphism of double graphs} is an assignment on each of these sets which preserves sources and targets.
\end{defn}

\begin{rem}
    A double graph has both an \emph{underlying horizontal graph} given by the objects and horizontal edges, and an \emph{underlying vertical graph} given by the objects and vertical edges.
\end{rem}

\begin{defn}
    A \emph{double graph with $1$-identities}  is a double graph in which the underlying horizontal and vertical graphs are reflexive.

    A \emph{morphism of double graphs with $1$-identities} is a morphism of double graphs which preserves identity edges.
\end{defn}

There is an intermediate notion between double graphs with $1$-identities and double categories, which is analogous to Street's derivation schemes of \cite{Street_DerSchemes}.

\begin{defn} 
A \emph{double derivation scheme} is a double graph with $1$-identities whose underlying horizontal and vertical reflexive graphs are categories.

A \emph{morphism of double derivation schemes} is a morphism of double graphs with $1$-identities which is a functor on underlying horizontal and vertical categories.

We denote by $\dbldersch$ the category of double derivation schemes and their morphisms.
\end{defn}

\begin{notation}
    We denote the underlying horizontal and vertical categories of a double derivation scheme $\bS$ by $U\bfH\bS$ and $U\bfV\bS$, respectively.
\end{notation}

The forgetful functor from double categories to double derivation schemes admits a left adjoint, giving the \emph{free double category} on a double derivation scheme. This appears as \cite[Proposition 3.8]{FPP}.

\begin{prop}
The forgetful functor admits a left adjoint. 
\begin{tz}
\node[](2) {$\dbldersch$}; 
\node[right of=2,xshift=2cm](3) {$\dblcat$};

\draw[->,bend left=20] ($(2.east)+(0,4pt)$) to node[above,la]{$\bF_\mathrm{dbl}$} ($(3.west)+(0,4pt)$);
\draw[->,bend left=20] ($(3.west)-(0,4pt)$) to node[below,la]{$U$}  ($(2.east)-(0,4pt)$);

\node[la] at ($(2.east)!0.5!(3.west)$) {$\bot$}; 
\end{tz}
\end{prop}

\begin{defn}
A \emph{congruence} on a double category consists of an equivalence relation on each set of squares with a fixed boundary
such that if $\alpha\sim\alpha'$, $\beta\sim\beta'$, and $\gamma\sim\gamma'$ then 
\[ \alpha \vert \beta\sim \alpha' \vert \beta' \hspace{1em}\text{ and }  \hspace{1em}\frac{\alpha}{\gamma}\sim\frac{\alpha'}{\gamma'} \]
whenever the composites are defined.
\end{defn}

These congruence relations can be used to define quotient double categories.

\begin{defn} 
Let $\bD$ be a double category and $\sim$ be a congruence on $\bD$. The \emph{quotient double category} $\bD/_\sim$ has the same objects, horizontal and vertical morphisms as~$\bD$. For a square boundary $\sqboundary{f}{g}{u}{v}$, the set of squares is given by
    \[ (\bD/_\sim)\sqboundary{f}{g}{u}{v}=\left(\bD\sqboundary{f}{g}{u}{v}\right)/_\sim, \]
with horizontal and vertical compositions of squares induced by those in $\bD$. It comes together with a canonical quotient double functor $\bD\to \bD/_\sim$.
\end{defn}

The following is a consequence of \cite[Lemma 4.4 and Theorem 4.5]{FPP}.

\begin{prop}
    Let $F\colon I\to \dbldersch$ be a functor. Then the colimit of $F$ is the double derivation scheme whose underlying horizontal and vertical categories are the colimits of $U\bfH F\colon I\to \cat$ and $U\bfV F\colon I\to \cat$, respectively, and whose set of squares is given by the colimit of the sets of squares. 
\end{prop}

One can then leverage this result to obtain a formula for colimits in double categories; this appears as \cite[Theorem 4.6]{FPP}.

\begin{theorem}
\label{thm:colimits_dblcats}
Let $F\colon I\to \dblcat$ be a functor and denote by $\bS$ the colimit of $UF\colon I\to \dbldersch$. Then the colimit of $F$ is the quotient
\[ \bF_\mathrm{dbl}(\bS)/_\sim \]
where $\sim$ is the smallest congruence relation such that the morphisms 
\[ q_{Fi}
\colon Fi\to \bF_\mathrm{dbl}(\bS)/_\sim \]
 are double functors. Here, the morphisms $q_{Fi}$ are induced by the morphisms of double derivation schemes $UFi\to \bS\to U\bF_\mathrm{dbl}(\bS)$ and the canonical quotient double functor $\bF_\mathrm{dbl}(\bS)\to \bF_\mathrm{dbl}(\bS)/_\sim$.
\end{theorem}

\begin{rem} 
Note that $\bS$, $\bF_\mathrm{dbl}(\bS)$, and $\bF_\mathrm{dbl}(\bS)/_\sim$ have the same underlying horizontal and vertical categories. In particular, the underlying horizontal and vertical categories of the colimit $\bF_\mathrm{dbl}(\bS)/_\sim$ are the colimits of $U\bfH F\colon I\to \cat$ and $U\bfV F\colon I\to \cat$, respectively.
\end{rem}

We will be interested in the description that this formula gives for the case of pushouts of double categories. For convenience, we interpret the theorem in this particular case.

\begin{cor} \label{charpushout}
    Consider a diagram of double categories
    \[
    \bB\leftarrow\bA\rightarrow \bC \]
    and denote by $\bS$ the pushout of the corresponding diagram of underlying double derivation schemes. 
    Then the pushout of $\bB\leftarrow\bA\rightarrow\bC$ is the quotient 
    \[ \bF_\mathrm{dbl}(\bS)/_\sim, \]
    where $\sim$ is the smallest congruence on the double category $\bF_\mathrm{dbl}(\bS)$ such that the induced morphisms $q_\bB\colon \bB\to  \bF_\mathrm{dbl}(\bS)/_\sim$ and $q_\bC\colon \bC\to \bF_\mathrm{dbl}(\bS)/_\sim$ are double functors.
\end{cor}

\section{Left properness}\label{section:proper}
 
 The goal of this section is to prove that any model structure on $\dblcat$ whose trivial fibrations are the canonical ones must be left proper. To do so, we first describe in \cref{subsec:cofibs} a set of generating canonical cofibrations that identifies the canonical trivial fibrations as the right class in a weak factorization system; these are then used in \cref{subsec:properness} to establish our desired left properness result.

\subsection{Canonical cofibrations and trivial fibrations}\label{subsec:cofibs}

Let us first formally define the class of canonical trivial fibrations, which will be a common feature of all the model structures we consider in this work.

\begin{defn}\label{def:trivfibs}
    A double functor is a \emph{canonical trivial fibration} if it is surjective on objects, full on both horizontal and vertical morphisms, and fully faithful on squares.
\end{defn}

We recall the following terminology.

\begin{notation}\label{notation}
Let $\cI$ be a set of morphisms in a cocomplete category $\cC$. A morphism in~$\cC$ is
\begin{rome}
   \item \emph{$\cI$-injective} if it has the right lifting property with respect to every morphism in~$\cI$; the class of all such morphisms is denoted $\inj(\cI)$,
    \item an \emph{$\cI$-cofibration} if it has the left lifting property with respect to every $\cI$-injective morphism; the class of all such morphisms is denoted $\cof(\cI)$.
\end{rome}
\end{notation}

\begin{rem}
  In a locally presentable category $\cC$, the pair $(\cof(\cI),\inj(\cI))$ forms a weak factorization system for any set $\cI$ of morphisms in $\cC$.
\end{rem}

Next, we define a set of generating canonical cofibrations in the category $\dblcat$.

\begin{defn}\label{gen_cofib}
Let $\cI$ be the set consisting of the following double functors: 
\begin{rome}
\item the unique morphism $\emptyset\to \mathbbm{1}$, 
\item the inclusion $\mathbbm{1}\sqcup\mathbbm{1}\to \bH\mathbbm{2}$, 
\item the inclusion $\mathbbm{1}\sqcup\mathbbm{1}\to \bV\mathbbm{2}$, 
\item the inclusion $\partial(\bH\mathbbm{2}\times\bV\mathbbm{2})\to \bH\mathbbm{2}\times \bV\mathbbm{2}$ of the boundary of the square,
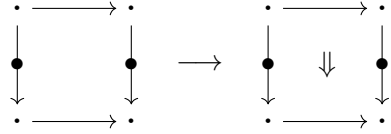
\begin{tz}
    \node[](1) {$\boldsymbol{\cdot}$}; 
    \node[right of=1](2) {$\boldsymbol{\cdot}$};
    \node[below of=1](1') {$\boldsymbol{\cdot}$};
    \node[below of=2](2') {$\boldsymbol{\cdot}$};
    \draw[->] (1) to (2);
    \draw[->] (1') to (2');
    \draw[->,pro] (1) to (1');
    \draw[->,pro] (2) to (2');
    \node[right of=2,xshift=.3cm](1) {$\boldsymbol{\cdot}$}; 
    \node at ($(1)!0.5!(2')$) {$\longrightarrow$}; 
    \node[right of=1](2) {$\boldsymbol{\cdot}$};
    \node[below of=1](1') {$\boldsymbol{\cdot}$};
    \node[below of=2](2') {$\boldsymbol{\cdot}$};
    \draw[->] (1) to (2);
    \draw[->] (1') to (2');
    \draw[->,pro] (1) to (1');
    \draw[->,pro] (2) to (2');
    \node[] at ($(1)!0.5!(2')$) {\rotatebox{270}{$\Rightarrow$}};
    \end{tz}
\item the double functor $\bH\mathbbm{2}\times\bV\mathbbm{2}\sqcup_{\partial(\bH\mathbbm{2}\times\bV\mathbbm{2})}\bH\mathbbm{2}\times\bV\mathbbm{2}\to \bH\mathbbm{2}\times\bV\mathbbm{2}$ sending the two parallel squares to the non-trivial square.
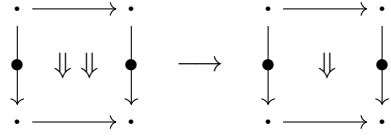
\begin{tz}
    \node[](1) {$\boldsymbol{\cdot}$}; 
    \node[right of=1](2) {$\boldsymbol{\cdot}$};
    \node[below of=1](1') {$\boldsymbol{\cdot}$};
    \node[below of=2](2') {$\boldsymbol{\cdot}$};
    \draw[->] (1) to (2);
    \draw[->] (1') to (2');
    \draw[->,pro] (1) to (1');
    \draw[->,pro] (2) to (2');
    \node[] at ($(1)!0.5!(2')-(5pt,0)$) {\rotatebox{270}{$\Rightarrow$}};
    \node[] at ($(1)!0.5!(2')+(5pt,0)$) {\rotatebox{270}{$\Rightarrow$}};

    \node[right of=2,xshift=.3cm](1) {$\boldsymbol{\cdot}$}; 
    \node at ($(1)!0.5!(2')$) {$\longrightarrow$}; 
    \node[right of=1](2) {$\boldsymbol{\cdot}$};
    \node[below of=1](1') {$\boldsymbol{\cdot}$};
    \node[below of=2](2') {$\boldsymbol{\cdot}$};
    \draw[->] (1) to (2);
    \draw[->] (1') to (2');
    \draw[->,pro] (1) to (1');
    \draw[->,pro] (2) to (2');
    \node[] at ($(1)!0.5!(2')$) {\rotatebox{270}{$\Rightarrow$}};
    \end{tz}
\end{rome}
\end{defn}

Inspecting the right lifting conditions with respect to the double functors in $\cI$, we can identify the $\cI$-injective morphisms with the canonical trivial fibrations.

\begin{prop}\label{prop:trivfibs} 
 A double functor has the right lifting property with respect to the maps in $\cI$ if and only if it is a canonical trivial fibration.
\end{prop}

\begin{defn}\label{defn:cofibs} A double functor is a \emph{canonical cofibration} if it is an $\cI$-cofibration.
\end{defn}

The following description is a combination of \cite[Theorem 3.11, Remark 3.12, and Corollary 3.13]{whi}.

\begin{prop}\label{charcof}
   A double functor $F\colon\bA\to\bB$ is a canonical cofibration if and only if 
   \begin{rome}
       \item it is injective on objects and faithful on both horizontal and vertical morphisms,
       \item there is a double category $\bC$ and double functors $\bB\xrightarrow{I}\bC\xrightarrow{R}\bB$ with $RI=\id_\bB$ such that the underlying categories $U\bfH \bC$ and $U\bfV \bC$ are obtained from the image of $U\bfH (IF)$ and $U\bfV (IF)$ by freely adding objects and then freely adding morphisms between specified objects.
   \end{rome}
   
   In particular, a double category $\bA$ is canonically cofibrant if and only if its underlying horizontal and vertical categories $U\bfH\bA$ and $U\bfV\bA$ are free.
\end{prop}

\subsection{Left properness}\label{subsec:properness}

The left properness of the relevant model structures is a consequence of the following lemma, whose proof will be the main content of this section.

\begin{lem}\label{pushout_trivfib} Given a diagram of pushouts in $\dblcat$ 
\begin{tz}
\node[] (1)     {$\bA$};
\node[right of=1] (2)   {$\bC$};
\node[right of=2] (3) {$\bE$};
\node[below of=1] (4) {$\bB$};
\node[below of=2] (5) {$\bD$};
\node[below of=3] (6) {$\bF$};

\draw[->] (1) to node[left,la]{$I$} (4);
\draw[->] (2) to node[left, la]{$J$} (5);
\draw[->] (3) to node[right, la]{$K$} (6);
\draw[->] (1) to (2);

\draw[->] (2) to node[above, la]{$P$} (3);
\draw[->] (4) to (5);
\draw[->] (5) to node[below, la]{$Q$} (6);

\pushout{5};
\pushout{6};
\end{tz}
where $I\colon\mathbb{A}\to\mathbb{B}$ is a double functor in $\cI$ and $P$ is a canonical trivial fibration, we have that $Q$ is a canonical trivial fibration.
\end{lem}

Before we focus on the proof of this lemma, we explain how one can use it to conclude our desired left properness result. First, using the same model categorical arguments as in the proof of \cite[Theorem 6.3]{Lack2Cat} we deduce the following.

\begin{prop} \label{prop:pushoutoftrivfib}
    Canonical trivial fibrations are stable under pushouts along canonical cofibrations.
\end{prop}

With this in hand, we can prove the desired left properness result.

\begin{theorem} \label{thm:leftproper}
Any model structure on $\dblcat$ whose trivial fibrations are the canonical ones is left proper.
\end{theorem}

\begin{proof} 
    Consider a pushout square in $\dblcat$
    \begin{tz}
\node[] (1)     {$\bA$};
\node[right of=1] (2)   {$\bC$};
\node[below of=1] (4) {$\bB$};
\node[below of=2] (5) {$\bD$};

\draw[->] (1) to node[left,la]{$I$} (4);
\draw[->] (2) to (5);

\draw[->] (1) to node[above, la]{$W$} (2);
\draw[->] (4) to node[below,la]{$V$} (5);

\pushout{5};
\end{tz}
where $I$ is a canonical cofibration and $W$ is a weak equivalence. We need to show that $V$ is also a weak equivalence. Factoring $W=P J$ as a trivial cofibration $J\colon \bA\to \bX$ followed by a canonical trivial fibration $P\colon \bX\to \bC$, we can then consider the following diagram in~$\dblcat$
\begin{tz}
\node[] (1)     {$\bA$};
\node[right of=1] (2)   {$\bX$};
\node[right of=2] (3) {$\bC$};
\node[below of=1] (4) {$\bB$};
\node[below of=2] (5) {$\bY$};
\node[below of=3] (6) {$\bD$};

\draw[->] (1) to node[left,la]{$I$} (4);
\draw[->] (2) to node[right,la]{$I'$} (5);
\draw[->] (3) to (6);
\draw[->] (1) to node[above,la]{$J$} (2);

\draw[->] (2) to node[above, la]{$P$} (3);
\draw[->] (4) to node[below,la]{$K$} (5);
\draw[->,dashed] (5) to node[below, la]{$Q$} (6);

\draw[->,bend left=40] (1) to node[above,la]{$W$} (3);
\draw[->,bend right=40] (4) to node[below,la]{$V$} (6);

\pushout{5};
\pushout{6};
\end{tz}
where the left-hand square is a pushout and the double functor $Q$ is induced by the universal property of $\bY$. Note that the right-hand square is also a pushout by the cancellation property of pushouts. 

We have that $K$ is a trivial cofibration as a pushout of the trivial cofibration $J$, and $I'$ is a canonical cofibration as a pushout of the canonical cofibration $I$. Hence, by \cref{prop:pushoutoftrivfib}, it follows that $Q$ is a canonical trivial fibration as a pushout of the canonical trivial fibration $P$ along~$I'$. We then conclude that $V=Q K$ is a weak equivalence as a composite of the weak equivalences $K$ and $Q$.  
\end{proof}

Finally, we conclude this section with a proof of the lemma.

\begin{proof}[Proof of \cref{pushout_trivfib}] 
We prove the result for each of the generating canonical cofibrations of the set $\cI$ in \cref{gen_cofib}.

\textbf{The double functor $\emptyset\to \mathbbm{1}$:} In this case, we have $\bD \cong\bC \sqcup\mathbbm{1}$ and $\bF \cong\bE \sqcup\mathbbm{1}$. Thus the double functor $Q\colon \bD \to \bF$ is isomorphic to $P\sqcup\mathbbm{1}\colon \bC \sqcup\mathbbm{1}\to \bE \sqcup\mathbbm{1}$ and hence it is a canonical trivial fibration, as $P$ is so. 

\textbf{The double functors $\mathbbm{1}\sqcup\mathbbm{1}\to \bH\mathbbm{2}$ and $\mathbbm{1}\sqcup\mathbbm{1}\to \bV\mathbbm{2}$:} We deal with the case of $\mathbbm{1}\sqcup\mathbbm{1}\to \bH\mathbbm{2}$, as the result for $\mathbbm{1}\sqcup\mathbbm{1}\to \bV\mathbbm{2}$ follows by reversing the role of the horizontal and vertical morphisms. We begin by describing the double functor $Q\colon\bD \to\bF$. 

The double category $\bD$ is obtained from $\bC$ by freely adjoining a horizontal morphism $f\colon A\to B$ between objects $A,B$ in $\bC$. Therefore $\bD$ has the same objects and vertical morphisms as $\bC$, and its horizontal morphisms are composites
\[ c_n\circ f\circ c_{n-1}\circ f\circ\cdots\circ f\circ c_1, \]
where $n\geq 0$ and $c_i$ is a horizontal morphism in $\bC$, for all $0\leq i\leq n$. Similarly, squares in~$\bD$ are horizontal composites of squares
\[ \gamma_1\vert e_f \vert \gamma_{2}\vert e_f\vert\cdots\vert e_f\vert \gamma_n\] 
where $n\geq 0$, each $\gamma_i$ is a square in $\bC$ for all $0\leq i\leq n$, and $e_f$ denotes the vertical identity square at the horizontal morphism $f$. Similarly, the double category $\bF$ is obtained from $\bE$ by freely adjoining a horizontal morphism $g\colon PA\to PB$ between the objects $PA,PB$ in $\bE$ and thus $\bF$ admits a description as above.

We now use these descriptions to unpack the action of the double functor $Q\colon\bD\to \bF$. It acts as $P$ on objects and vertical morphisms. For horizontal morphisms, we have 
\[ Q(c_n\circ f\circ\cdots\circ f\circ c_1)=Pc_n\circ g\circ \cdots\circ g\circ Pc_1 \]
and for squares, we have
\[ Q(\gamma_1\vert e_f \vert \gamma_{2}\vert e_f\vert\cdots\vert e_f\vert \gamma_n)=P\gamma_1\vert e_g \vert P\gamma_{2}\vert e_g\vert\cdots\vert e_g\vert P\gamma_n. \] 
Hence, it is straightforward to check that $Q$ is a canonical trivial fibration, as $P$ is so.
 
\textbf{The double functor $\bH\mathbbm{2}\times\bV\mathbbm{2}\sqcup_{\partial(\bH\mathbbm{2}\times\bV\mathbbm{2})}\bH\mathbbm{2}\times\bV\mathbbm{2}\to \bH\mathbbm{2}\times\bV\mathbbm{2}$:} We begin by giving a different description of the double functor $Q$, that facilitates access to the  information required to prove that $Q$ is a canonical trivial fibration. To describe the domain $\bD$, let $\alpha$ and $\alpha'$ be the squares with the same boundary in $\bC$ that are identified in $\bD$. 

Consider the free double category $\bF_\mathrm{dbl}(\bS_\bD)$, where  $\bS_\bD$ is the double derivation scheme whose objects, horizontal morphisms, and vertical morphisms are those of $\bC$, and whose set of squares is given by the quotient $\mathrm{sq}\bC /\alpha\sim\alpha'$ of the set of squares of $\bC$ resulting by identifying $\alpha$ and $\alpha'$. Now, denote by $\sim_\bC$ the smallest congruence on $\bF_\mathrm{dbl}(\bS_\bD)$ such that 
\[ q_\bC\colon \bC\longrightarrow \bF_\mathrm{dbl}(\bS_\bD)/_{\sim_\bC}\]
is a double functor; by \cref{charpushout} we have that $\bD \cong \bF_\mathrm{dbl}(\bS_\bD)/_{\sim_\bC}$.

Similarly, let $\bF_\mathrm{dbl}(\bS_\bF)$ be the free double category on the double derivation scheme $\bS_\bF$ whose objects, horizontal morphisms, and vertical morphisms are those of $\bE$, and whose set of squares is the quotient $\mathrm{sq}\bE /P\alpha\sim P\alpha'$ of the set of square of $\bE$ resulting by identifying $P\alpha$ and $P\alpha'$. Denote by $\sim_\bE$ the smallest congruence on $\bF_\mathrm{dbl}(\bS_\bF)$ such that 
\[ q_\bE\colon \bE \longrightarrow \bF_\mathrm{dbl}(\bS_\bF)/_{\sim_\bE}\]
is a double functor; by \cref{charpushout} we have that $\bF \cong\bF_\mathrm{dbl}(\bS_\bF)/_{\sim_\bE}$.

Note that the double functor $P\colon\bC\to\bE$ induces a map between double derivation schemes $\hat{Q}\colon \bS_\bD \to\bS_\bF$. This, in turn, gives a double functor $\bF_\mathrm{dbl}(\hat{Q})$ such that the following diagram of underlying derivation schemes commutes.
\begin{diagram} \label{diagram1}
\node[] (2)   {$\bC$};
\node[right of=2,xshift=1.8cm] (3) {$\bE$};

\node[below of=2] (5) {$\bF_\mathrm{dbl}(\bS_\bD)$};
\node[below of=3] (6) {$\bF_\mathrm{dbl}(\bS_\bF)$};

\draw[->] (2) to node[left, la]{} (5);
\draw[->] (3) to node[right, la]{} (6);

\draw[->] (2) to node[above, la]{$P$} (3);
\draw[->] (5) to node[below, la]{$\bF_\mathrm{dbl}(\hat{Q})$} (6);
\end{diagram}
Since $P$ is a canonical trivial fibration, it follows by construction that so is $\bF_\mathrm{dbl}(\hat{Q})$. Moreover, as $Q\colon \bF_\mathrm{dbl}(\bS_\bD)/_{\sim_\bC}\to \bF_\mathrm{dbl}(\bS_\bF)/_{\sim_\bE}$ is the double functor induced by $\bF_\mathrm{dbl}(\hat{Q})$ on quotients, we see that $Q$ must be surjective on objects, full on horizontal and vertical morphisms, and full on squares. Thus, it remains to show that $Q$ is faithful on squares.

To do this, we first define a new congruence $\sim_\bC$ on $\bF_\mathrm{dbl}(\bS_\bF)$ by letting $\partial\sim_\bC \partial'$ in $\bF_\mathrm{dbl}(\bS_\bF)$ if there exist squares $\gamma\sim_\bC \gamma'$ in $\bF_\mathrm{dbl}(\bS_\bD)$ such that $\bF_\mathrm{dbl}(\hat{Q})\gamma=\partial$ and $\bF_\mathrm{dbl}(\hat{Q})\gamma'=\partial'$. The fact that this is a congruence follows from $\bF_\mathrm{dbl}(\hat{Q})$ being full on squares. We thus get an induced double functor $\overline{\bF_\mathrm{dbl}(\hat{Q})}$ on quotients
making the following diagram commute.
\begin{diagram} \label{diagram2}
\node[] (2)   {$\bF_\mathrm{dbl}(\bS_\bD)$};
\node[right of=2,xshift=2.4cm] (3) {$\bF_\mathrm{dbl}(\bS_\bF)$};

\node[below of=2] (5) {$\bF_\mathrm{dbl}(\bS_\bD)/_{\sim_\bC}$};
\node[below of=3] (6) {$\bF_\mathrm{dbl}(\bS_\bF)/_{\sim_\bC}$};

\draw[->] (2) to node[left, la]{} (5);
\draw[->] (3) to node[right, la]{} (6);

\draw[->] (2) to node[above, la]{$\bF_\mathrm{dbl}(\hat{Q})$} (3);
\draw[->] (5) to node[below, la]{$\overline{\bF_\mathrm{dbl}(\hat{Q})}$} (6);
\end{diagram}

We claim that $\overline{\bF_\mathrm{dbl}(\hat{Q})}$ is faithful on squares. Indeed, let $\delta,\delta'$ be squares in $\bF_\mathrm{dbl}(\bS_\bD)$ with the same boundary such that $\overline{\bF_\mathrm{dbl}(\hat{Q})\delta}=\overline{\bF_\mathrm{dbl}(\hat{Q})\delta'}$. By definition, this means that $\bF_\mathrm{dbl}(\hat{Q})\delta\sim_\bC \bF_\mathrm{dbl}(\hat{Q})\delta'$, and so there exist squares $\gamma\sim_\bC \gamma'$ in $\bF_\mathrm{dbl}(\bS_\bD)$ such that $\bF_\mathrm{dbl}(\hat{Q})\gamma=\bF_\mathrm{dbl}(\hat{Q})\delta$ and $\bF_\mathrm{dbl}(\hat{Q})\gamma'=\bF_\mathrm{dbl}(\hat{Q})\delta'$. Using the fact that $\bF_\mathrm{dbl}(\hat{Q})$ is a canonical trivial fibration, we can express $\delta$ as a pasting involving $\gamma$ and invertible squares that relate the boundary morphisms in $\delta$ to those in $\gamma$. But recall that $\delta$ and $\delta'$ have the same boundary, as do $\gamma$ and $\gamma'$; hence, we can use the same invertible squares to express $\delta'$ in terms of $\gamma'$. Since $\gamma\sim_\bC \gamma'$, these pastings must also be congruent and thus $\delta\sim_\bC\delta'$, as desired.

Our final step is to show that $Q=\overline{\bF_\mathrm{dbl}(\hat{Q})}$, which would imply the desired faithfulness on squares for $Q$, by the previous paragraph. By the universal property of the pushout for $\bF_\mathrm{dbl}(\bS_\bD)/_{\sim_\bC}$, it suffices to show that $\bF_\mathrm{dbl}(\bS_\bE)/_{\sim_\bC}=\bF_\mathrm{dbl}(\bS_\bE)/_{\sim_\bE}$. The inclusion $\sim_\bC \subseteq\sim_\bE$ is straightforward, as given any two congruent squares $\gamma\sim_\bC \gamma'$ in $\bF_\mathrm{dbl}(\bS_\bD)$, we know that $\bF_\mathrm{dbl}(\hat{Q})\gamma\sim_\bE \bF_\mathrm{dbl}(\hat{Q})\gamma'$ because $Q$ is a double functor. On the other hand, $\sim_\bE$ is the smallest congruence such that the map
\begin{tz}
\node[](1)      {$\bE$};
\node[right of=1,xshift=1.2cm](2)    {$\bF_\mathrm{dbl}(\bS_\bF)/_{\sim_\bE}$};
\draw[->] (1) to node[above,la]{$q_{\bE }$} (2);
\end{tz}
is a double functor. Then, to show that $\sim_\bE \subseteq\sim_\bC$, it is enough to show that the map 
\begin{tz}
\node[](1)      {$\bE$};
\node[right of=1,xshift=1.2cm](2)    {$\bF_\mathrm{dbl}(\bS_\bF)/_{\sim_\bC}$};
\draw[->] (1) to node[above,la]{$q'_{\bC }$} (2);
\end{tz}
is a double functor. 

Since $P$ is surjective on objects, full on both horizontal and vertical morphisms, and full on squares, it is easy to check that $\bE \xrightarrow{q'_\bC } \bF_\mathrm{dbl}(\bS_\bF)/_{\sim_\bC}$ is a double functor if and only if the composite $\bC \xrightarrow{P}\bE \xrightarrow{q'_\bC } \bF_\mathrm{dbl}(\bS_\bF)/_{\sim_\bC}$ is a double functor. Combining the commutative diagrams \ref{diagram1} and \ref{diagram2}, we see that this composite is equal to \[\bC\to \bF_\mathrm{dbl}(\bS_\bD)/_{\sim_\bC}\xrightarrow{\overline{\bF_\mathrm{dbl}(\hat{Q})}}\bF_\mathrm{dbl}(\bS_\bF)/_{\sim_\bC},\] which is a double functor by construction.

\textbf{The double functor $\partial(\bH\mathbbm{2}\times\bV\mathbbm{2})\to \bH\mathbbm{2}\times \bV\mathbbm{2}$:} The argument is analogous to the previous case. Instead of taking double derivation schemes whose squares identify two squares, we now want to take into account that we are ``adding a square''; for instance the set of squares of the double derivation scheme associated to $\bD$ is given by $\mathrm{sq}\bC \sqcup \{\alpha\}$.
\end{proof}

\section{Canonical equivalences and the gregarious model structure}

In this section we use a result from previous work of Guetta and the authors \cite{fibrantly_induced}, that we recall in \cref{subsec:MSbackground}, to present in \cref{subsec:gregarious} an alternative proof of the existence of the gregarious model structure on the category of double categories of Campbell \cite{Camp}. This forms the ``base case'' of the method we present in \cref{section:generalMS}, as all other model structures we produce will be localizations of the gregarious model structure. This result is established in \cref{subsec:generalcase}, and used as evidence that the gregarious double equivalences are the canonical notion of double categorical equivalence between strict double categories.

\subsection{Recollection: a tool to construct model structures} \label{subsec:MSbackground}

Our strategy for this section and the next will be to use a result from \cite{fibrantly_induced} to construct model structures from a given data. For convenience, we briefly recall the set-up and statement of the relevant theorem.

We start with the data of a locally presentable category $\cC$ and a weak factorization system $(\an,\nfib)$ on $\cC$ generated by a set $\cJ$, that serves a similar role as the \emph{anodyne extensions} and \emph{naive fibrations} of \cite{cisinski}; see also \cref{def:anodyne}. This weak factorization system yields the following notions of ``fibrant objects'' and ``fibrant replacements''.

\begin{defn}
An object $X\in \cC$ is \emph{naive fibrant} if the unique morphism $X\to *$ to the terminal object in $\cC$ is a naive fibration. 
\end{defn}

\begin{defn}
Given an object $X\in\cC$, a \emph{naive fibrant replacement} of $X$ is an anodyne extension $\iota_X \colon X \to X'$ such that $X'$ is naive fibrant.

Similarly, given a morphism $f \colon X\to Y$ in $\cC$, a \emph{naive fibrant replacement} of  $f$ is a commutative square
\begin{diagram} \label{naiverepl}
\node[](1) {$X$}; 
\node[right of=1](2) {$Y$}; 
\node[below of=1](1') {$X'$}; 
\node[below of=2](2') {$Y'$}; 

\draw[->] (1) to node[above,la]{$f$} (2); 
\draw[->] (1) to node[left,la]{$\iota_X$} (1'); 
\draw[->] (2) to node[right,la]{$\iota_Y$} (2'); 
\draw[->] (1') to node[below,la]{$f'$} (2');
\end{diagram}
where $\iota_X$, $\iota_Y$ are anodyne extensions and $X'$, $Y'$ are naive fibrant.
\end{defn}

\begin{rem} \label{rem:naivefibrep}
Note that naive fibrant replacements always exist, as they can be constructed using the factorization system $(\an,\nfib)$.
\end{rem}

Now, suppose that in addition to the factorization system $(\an, \nfib)$ we are given a class $\Wf$ of morphisms between naive fibrant objects, which we interpret as weak equivalences between fibrant objects. We can then use the above definitions to construct a new class $\cW$, which will play the role of our weak equivalences in $\cC$.

\begin{defn}\label{def:we}
A morphism $f \colon X \to Y$ in $\cC$ is a \emph{weak equivalence} if there exists a naive fibrant replacement \eqref{naiverepl} of $f$ such that $f'$ is in $\Wf$. 

We denote by $\cW$ the class of weak equivalences.
\end{defn}

The following can be found as \cite[Theorem 2.8]{fibrantly_induced}.

\begin{theorem}  \label{thm:GMSV}
    Let $\cC$ be a locally presentable category, $\cI$ be a set of morphisms in $\cC$, and $(\an,\nfib)$ be a weak factorization system in $\cC$ generated by a set $\cJ$ such that ${\an\subseteq\cof(\cI)}$. Suppose in addition that we have a class $\Wf$ of morphisms in $\cC$ between naive fibrant objects such that
\begin{enumerate}[label=(\arabic*)]
    \item\label{ax:trivfib} $\inj(\cI)\subseteq\cW$, where $\cW$ is described in \cref{def:we},
    \item\label{2of6Wf} $\Wf$ has $2$-out-of-$6$, 
    \item \label{accessibility} there exists a class $\overline{\cW}$ of morphisms such that $\Wf$ is the restriction of $\overline{\cW}$ to the morphisms between naive fibrant objects and $\overline{\cW}$ considered as a full subcategory of $\cC^{\mathbbm{2}}$ is accessible,
    \item\label{anwe} the morphisms in $\an$ between naive fibrant objects are in $\Wf$,
    \item\label{fibwe} $\nfib\cap\Wf\subseteq\inj(\cI)$.
\end{enumerate}
Then there exists a combinatorial model structure on $\cC$ with cofibrations given by the $\cI$-cofibrations, fibrant objects given by the naive fibrant objects, and weak equivalences given by the morphisms in $\cW$. Furthermore, the weak equivalences (resp.\ fibrations) between fibrant objects are precisely the morphisms in $\Wf$ (resp.\ $\nfib$).
\end{theorem}

\begin{rem} \label{rem:pathobject}
    By \cite[Proposition 2.21]{fibrantly_induced}, condition (4) is equivalent to the following: 
    \begin{enumerate}
        \item[(4')] for every naive fibrant object $X$, there is a factorization of the diagonal morphism 
    \[ X\xrightarrow{w} \Path X\xrightarrow{p} X\times X \]
    such that $w\in \Wf$ and $p\in \nfib$.
    \end{enumerate}
\end{rem}

\subsection{The gregarious model structure}\label{subsec:gregarious}

We now give an alternative proof of the existence of the gregarious model structure of Campbell \cite{Camp}.

\begin{notation} \label{notn:J0}
    Let $\cJ_0$ be the set consisting of the following double functors: 
    \begin{rome}
        \item the inclusion $\mathbbm{1}\to \Sq \Eadj$, where $\Eadj$ denotes the $2$-category free on an adjoint equivalence data, 
        \item the inclusion $\bH \mathbbm{2}\to \bH \Sigma I$, where $\Sigma I$ denotes the $2$-category free on a $2$-isomorphism,
        \item the inclusion $\bV \mathbbm{2}\to \bV \Sigma I$. 
    \end{rome}
    It is straightforward to check that all double functors in $\cJ_0$ are canonical cofibrations in the sense of \cref{defn:cofibs}. We denote by $(\an_0,\nfib_0)$ the corresponding weak factorization system on $\dblcat$. 
\end{notation}

Inspecting the right lifting conditions with respect to the double functors in $\cJ_0$, we can extract an explicit description of the double functors in $\nfib_0$, for which we introduce the following terminology.

\begin{defn}
    A double functor is a \emph{gregarious fibration} if it has the right lifting property with respect to every double functor in $\cJ_0$. 
\end{defn}

\begin{defn}
    A companion pair $(f,u,\varphi,\psi)$ in a double category $\bA$ is a \emph{gregarious (adjoint) equivalence} if it induces an (adjoint) equivalence in the $2$-category $\bfR \bA$ from \cref{rightadjofSq}. 
\end{defn}

\begin{rem}
    Equivalently, the data of a gregarious adjoint equivalence in a double category $\bA$ can be encoded as a double functor $\Sq\Eadj\to \bA$. 
\end{rem}

\begin{lem} \label{lem:descrnfib0}
   A double functor $F\colon\bA\to\bB$ is a gregarious fibration if and only if
   \begin{enumerate}
       \item[(f1)] for every object $A\in\bA$ and every gregarious equivalence $(g,v,\chi,\omega)\colon FA\to B$ in~$\bB$, there is a gregarious equivalence $(f,u,\varphi,\psi)\colon A\to C$ in $\bA$ such that $Ff=g$, $Fu=v$, $F\varphi=\chi$, and $F\psi=\omega$,
       \item[(f2)] the induced $2$-functor $\bfH F\colon \bfH \bA\to \bfH \bB$ between underlying horizontal $2$-categories is locally an isofibration, 
       \item[(f3)] the induced $2$-functor $\bfV F\colon \bfV \bA\to \bfV \bB$ between underlying vertical $2$-categories is locally an isofibration. 
   \end{enumerate}
   In particular, every double category is gregarious fibrant.
\end{lem}

The weak equivalences in this model structure will be given by the following.

\begin{defn}\label{def:gregwe}
   A double functor $F\colon\bA\to\bB$ is a \textit{gregarious double equivalence} if it is
   \begin{enumerate}
       \item[(g1)] \emph{surjective on objects up to gregarious equivalence:} for every object $B\in\bB$, there is an object $A\in\bA$ and a gregarious equivalence $(g,v,\chi,\omega)\colon FA\to B$ in $\bB$,
       \item[(g2)] \emph{essentially full on horizontal morphisms:} the induced $2$-functor $\bfH F\colon \bfH \bA\to \bfH \bB$ between underlying horizontal $2$-categories is essentially full on morphisms, 
    \item[(g3)] \emph{essentially full on vertical morphisms:} the induced $2$-functor $\bfV F\colon \bfV \bA\to \bfV \bB$ between underlying vertical $2$-categories is essentially full on morphisms,  
    \item[(g4)] fully faithful on squares.
\end{enumerate}
   We denote by $\overline{\cW}$ the class of gregarious double equivalences. 
\end{defn}

Placing these notions in the terminology of \cref{thm:GMSV}, we consider the case where
\begin{itemize}[leftmargin=0.8cm]
    \item $\cC$ is the category $\dblcat$ of double categories and double functors, 
    \item $\cI$ is the set of generating canonical cofibrations of \cref{gen_cofib}, 
    \item $(\an_0,\nfib_0)$ is the weak factorization system on $\dblcat$ generated by the set $\cJ_0$ from \cref{notn:J0}, 
    \item $\Wf=\overline{\cW}$ is the class of gregarious double equivalences.
\end{itemize}
Our goal is now to establish all of the requirements of \cref{thm:GMSV}. 

First note that by the characterization of the canonical trivial fibrations in \cref{prop:trivfibs}, it is straightforward to see that canonical trivial fibrations satisfy the conditions of a gregarious double equivalence.

\begin{prop} \label{lem:trivfibarewe}
    Every canonical trivial fibration is a gregarious double equivalence.
\end{prop}

Moreover, the following result is obtained from a careful inspection of the explicit descriptions in \cref{prop:trivfibs,lem:descrnfib0,def:gregwe}.

\begin{prop} \label{lem:fibwe0}
    Every gregarious fibration that is also a gregarious double equivalence is a canonical trivial fibration.
\end{prop}

We now show that gregarious double equivalences are accessible and satisfy $2$-out-of-$6$. For this, we show that the class of gregarious double equivalences is the preimage under the right adjoint functor \[(\bfR,\bfH,\bfV,\bfH\llbracket\bV\mathbbm{2},-\rrbracket)\colon \dblcat\to \twocat^{\times 4}\]  
of the pointwise biequivalences in  $\twocat$. Here $\bfR$ is the right adjoint functor from \cref{rightadjofSq}, $\bfH$ and $\bfV$ are the right adjoint functors from \cref{underlyinghorver2cat}, and $\bfH\llbracket\bV\mathbbm{2},-\rrbracket$ is the composite of the right adjoint functors 
\[ \dblcat\xrightarrow{\llbracket\bV\mathbbm{2},-\rrbracket}\dblcat\xrightarrow{\bfH} \twocat \]
from \cref{underlyinghorver2cat,internalhom}.

\begin{prop}\label{prop:4tupleweakequiv}
    If $F\colon\bA\to\bB$ is a gregarious double equivalence, then the image $(\bfR,\bfH,\bfV,\bfH\llbracket\bV\mathbbm{2},- \rrbracket)F$ is a pointwise biequivalence.
\end{prop}
\begin{proof}
    We first prove that $\bfR F$ is a biequivalence. Condition (b1) in \cref{defn:bieq} is a consequence of (g1) in \cref{def:gregwe} by definition of a gregarious equivalence, and condition (b3) is a direct consequence of (g4) by definition of the $2$-morphisms in $\bfR\bA$ and~$\bfR\bB$.
    
    For (b2), given a morphism $(g,v,\chi,\omega)\colon FA\to FB$ in $\bfR\bB$, we want to find a morphism $(f,u,\varphi,\psi)\colon A\to B$ in $\bfR\bA$ and a $2$-isomorphism $\alpha\colon F(f,u,\varphi,\psi)\cong (g,v,\chi,\omega)$. In other words, we want to find a horizontal morphism $f$ in $\bA$ that has a vertical companion, together with a vertically invertible square in $\bB$
    \begin{tz}
        \node[](1) {$FA$}; 
\node[below of=1](2) {$FA$}; 
\node[right of=1](3) {$FB$}; 
\node[right of=2](4) {$FB$};

\draw[d,pro] (1) to node[left,la]{} (2); 
\draw[->] (1) to node[above,la]{$Ff$} (3); 
\draw[->] (2) to node[below,la]{$g$} (4); 
\draw[d,pro](3) to node[right,la]{} (4); 
 
\node[la] at ($(1)!0.5!(4)+(5pt,0)$) {$\vcong$};
\node[la] at ($(1)!0.5!(4)-(5pt,0)$) {$\beta$};
    \end{tz}
    The existence of the morphism $f$ and invertible square $\alpha$ is guaranteed by (g2), and since $Ff\cong g$, we have that $(Ff,v)$ is also part of the data of a companion pair in $\bB$. Now, by (g3), we have a horizontally invertible square 
    \begin{tz}
        \node[](1) {$FA$}; 
\node[below of=1](2) {$FA'$}; 
\node[right of=1](3) {$FA$}; 
\node[right of=2](4) {$FA'$};

\draw[->,pro] (1) to node[left,la]{$Fu$} (2); 
\draw[d] (1) to node[above,la]{} (3); 
\draw[d] (2) to node[below,la]{} (4); 
\draw[->,pro](3) to node[right,la]{$v$} (4); 
 
\node[la] at ($(1)!0.5!(4)-(5pt,0)$) {$\beta$};
\node[la] at ($(1)!0.5!(4)+(5pt,0)$) {$\cong$};
    \end{tz}
    which implies that $Fu$ is another vertical companion of $Ff$. Since $F$ is fully faithful on squares by (g4), the squares encoding the fact that $(Ff,Fu)$ are a companion pair must come from squares in $\bA$ that express $(f,u)$ as companions, as desired. 

    To check that $\bfH F$ and $\bfH\llbracket\bV\mathbbm{2},F\rrbracket$ are biequivalences, note that $F$ is a double biequivalence in the sense of \cite[Definition 3.6]{MSV}, and so the claim is given by \cite[Proposition 3.11]{MSV} (note that the functor $\bfH\llbracket\bV\mathbbm{2},-\rrbracket$ is denoted by $\cV$ in \cite{MSV}). Finally, since gregarious double equivalences are fully symmetric with respect to the horizontal and vertical directions, the fact that $\bfH F$ is a biequivalence means that so is $\bfV F$. 
\end{proof}

\begin{prop}\label{4tupleconverse}
    If $(\bfR,\bfH,\bfV,\bfH\llbracket \bV\mathbbm{2},-\rrbracket)F$ is a pointwise biequivalence, then $F\colon\bA\to\bB$ is a gregarious double equivalence.
\end{prop}
\begin{proof}
    By condition (b1) for the biequivalence $\bfR F$, we get that $F$ is surjective on objects up to gregarious equivalence; hence $F$ satisfies condition (g1) in \cref{def:gregwe}.

    Condition (g2) asserts that $F$ is full on horizontal morphisms up to $2$-isomorphism in $\bfH \bB$, which is precisely the condition (b2) satisfied by the biequivalence $\bfH F$. Similarly, (g3) is ensured by the condition (b2) of the biequivalence $\bfV F$.

    Finally, since both $\bfH F$ and $\bfH\llbracket\bV\mathbbm{2},F\rrbracket$ are biequivalences, by \cite[Proposition 3.11]{MSV} we have that $F$ is a double biequivalence; in particular, it is fully faithful on squares, which is precisely condition (g4).
\end{proof}

\begin{cor} \label{2of6accessible}
    The class $\overline{\cW}$ of gregarious double equivalences satisfies $2$-out-of-$6$, and is accessible as a full subcategory of $\dblcat^{\mathbbm{2}}$.
\end{cor}

\begin{proof}
    The pointwise biequivalences in $\twocat^{\times 4}$ satisfy $2$-out-of-$6$ and form an accessible class, as they are the weak equivalences of a combinatorial model structure. Hence, the result follows directly from the fact that, by \cref{prop:4tupleweakequiv,4tupleconverse}, the class $\overline{\cW}$ is the preimage under the accessible functor $(\bfR,\bfH,\bfV,\bfH\llbracket\bV\mathbbm{2},-\rrbracket)\colon \dblcat\to \twocat^{\times 4}$ of the pointwise biequivalences.
\end{proof}

We now turn to finding a path object, for which we must establish a series of technical lemmas. The first of these shows that all double categories have the right lifting property with respect to double functors that will eventually be trivial cofibrations in our model structure.

\begin{lem} \label{liftingofcofge2}
    Let $F\colon \bA\to \bB$ be a double functor that is both a canonical cofibration and a gregarious double equivalence. Then, for every double category $\bX$, there is a lift in every diagram of the form
    \begin{tz}
\node[](1) {$\bA$}; 
\node[right of=1](2) {$\bX$}; 
\node[below of=1](1') {$\bB$}; 

\draw[->] (1) to node[above,la]{} (2); 
\draw[->] (1) to node[left,la]{$F$} (1'); 
\draw[->,dashed] (1') to node[right,la,yshift=-5pt]{} (2); 
\end{tz}
\end{lem}

\begin{proof} 
To prove the result, we use the characterization of canonical cofibrations from \cref{charcof}. Note that, without loss of generality, we can assume that $\bB$ is such that its underlying categories $U\bfH \bB$ and $U\bfV\bB$ are obtained from the image of $U\bfH F$ and $U\bfV F$ by freely adding objects and then freely adding morphisms between specified objects, as the class of double functors with the left lifting property against $\bX$ is closed under retracts. Moreover, it is enough to prove the case where $\bX=\bA$, i.e., we need to construct a retraction $G\colon \bB\to \bA$ of $F$.

    We define the double functor $G\colon \bB\to \bA$ as follows.
    \begin{itemize}[leftmargin=0.8cm]
        \item Given an object $b\in \bB$, 
        \begin{itemize}[leftmargin=0.8cm]
            \item if $b=Fa$ for some object $a\in \bA$, we set $Gb\coloneqq a$ and $(f_b,u_b,\varphi_b,\psi_b)$ to be the identity companion pair at $Fa$,
            \item otherwise, by (g1) in \cref{def:gregwe}, there is an object $a\in \bA$ and a gregarious equivalence $(f_b,u_b,\varphi_b,\psi_b)\colon Fa\xrightarrow{\simeq} b$ in $\bB$, and we set $Gb\coloneqq a$. 
        \end{itemize}
        \item Given a horizontal morphism $g\colon b\to d$ in $\bB$, 
        \begin{itemize}[leftmargin=0.8cm]
            \item if $g=Ff$ for some horizontal morphism $f\colon a\to a'$ in $\bA$, we set $Gg\coloneqq f$ and $\alpha_g$ to be the vertical identity square at $Ff$, 
            \item if $g$ is a freely-added horizontal morphism, by (g2), there is a horizontal morphism ${f\colon Gb\to Gd}$ in $\bA$ and a vertically invertible square in $\bB$
            \begin{tz}
                \node[](1) {$FGb$}; 
                \node[below of=1](2) {$FGb$}; 
                \node[right of=2](3) {$b$}; 
                \node[right of=3](4) {$d$}; 
                \node[right of=4](5) {$FGd$}; 
                \node[above of=5](6) {$FGd$}; 

                \draw[d,pro] (1) to (2);
                \draw[d,pro] (6) to (5);

                \draw[->](1) to node[above,la]{$Ff$} (6); 
                \draw[->](2) to node[below,la]{$f_b$} (3);
                \draw[->](3) to node[below,la]{$g$} (4);
                \draw[->](4) to node[below,la]{$f_d^{-1}$} (5);

                \node[la] at ($(1)!0.5!(5)+(5pt,0)$) {$\vcong$}; 
                \node[la] at ($(1)!0.5!(5)-(5pt,0)$) {$\alpha_g$}; 
            \end{tz}
            and we set $Gg\coloneqq f$, 
            \item otherwise, we have that $g$ is a composite $b\xrightarrow{g_1} d_1\xrightarrow{g_2} \ldots \xrightarrow{g_{n-1}} d_{n-1}\xrightarrow{g_n} d$
            of horizontal morphisms in the image of $F$ and of freely-added morphisms;  we set $Gg\coloneqq G(g_n)\circ \ldots \circ G(g_1)$ and we set $\alpha_g$ to be the vertically invertible square from $FGg$ to $f_d^{-1}\circ g\circ f_b$ in $\bB$ obtained by taking the horizontal composite of the squares $\alpha_{g_1},\ldots,\alpha_{g_n}$ and using the counits to cancel out the copies of $f_{d_i}$ and $f_{d_i}^{-1}$. 
        \end{itemize}
    \item The procedure for vertical morphisms is analogous, using property (g3) and defining, for a vertical morphism $v\colon b\arrowdot b'$ in $\bB$, a vertical morphism $Gv$ in $\bA$ and a horizontally invertible square $\beta_v$ from $FGv$ to $u_{b'}^{-1}\circ v\circ u_b$ in $\bB$. 
    \item Given a square $\beta$ in $\bB$ as depicted below left, by (g4), there is a unique square $\alpha$ in~$\bA$ such that $F\alpha$ is equal to the pasting below right, and we set $G\beta\coloneqq \alpha$.
    \begin{tz}

    \node[](1) {$b$}; 
    \node[right of=1](2) {$d$};
    \node[below of=1](1') {$b'$};
    \node[below of=2](2') {$d'$};
    \draw[->] (1) to node[above,la]{$g$} (2);
    \draw[->] (1') to node[below,la]{$g'$} (2');
    \draw[->,pro] (1) to node[left,la]{$v$} (1');
    \draw[->,pro] (2) to node[right,la]{$w$} (2');
    \node[la] at ($(1)!0.5!(2')$) {$\beta$};

        \node[right of=2,xshift=1cm,yshift=3cm](0) {$FGb$}; 
        \node[below of=0](1) {$FGb$}; 
        \node[right of=0](0') {$FGb$}; 
        \node[below of=0'](1') {$FGb$}; 
        \node[below of=1'](2') {$b$};
        \node[below of=2'](3') {$b'$};
        \node[below of=3'](4') {$FGb'$};
        \node[below of=4'](5') {$FGb'$};
        \node[left of=4'](4) {$FGb'$};
        \node[below of=4](5) {$FGb'$};

        \draw[d,pro](0) to (1);
        \draw[->,pro](1) to node[left,la]{$FGv$} (4);
        \draw[d,pro](4) to (5);
        \draw[d,pro](0') to (1');
        \draw[->,pro](1') to node[left,la]{$u_b$} (2');
        \draw[->,pro](2') to node[left,la]{$v$} (3');
        \draw[->,pro](3') to node[left,la]{$u_{b'}^{-1}$} (4');
        \draw[d,pro](4') to (5');
        \draw[d](0) to (0');
        \draw[d](1) to (1');
        \draw[d](4) to (4');
        \draw[d](5) to (5');

        \node[la] at ($(0)!0.5!(1')$) {$=$};
        \node[la] at ($(1)!0.5!(4')-(2pt,0)$) {$\beta_v$};
        \node[la] at ($(4)!0.5!(5')$) {$=$};
 
        \node[right of=1'](1) {$b$}; 
        \node[below of=1](2) {$b$};
        \node[below of=2](3) {$b'$};
        \node[below of=3](4) {$b'$};

        \draw[d,pro](1) to (2);
        \draw[->,pro](2) to node[left,la]{$v$} (3);
        \draw[d,pro](3) to (4);
        \draw[->](1') to node[above,la]{$f_b$} (1); 
        \draw[->](4') to node[below,la]{$f_{b'}$} (4); 
        \draw[d](2') to (2); 
        \draw[d](3') to (3);
        
        \node[la] at ($(1')!0.5!(2)$) {$\varphi_b$};
        \node[la] at ($(2')!0.5!(3)-(2pt,0)$) {$=$};
        \node[la] at ($(3')!0.5!(4)$) {$\varphi_{b'}^{-1,v}$};

        \node[right of=1](1') {$d$}; 
        \node[below of=1'](2') {$d$};
        \node[below of=2'](3') {$d'$};
        \node[below of=3'](4') {$d'$};

        \draw[d,pro](1') to (2');
        \draw[->,pro](2') to node[right,la]{$w$} (3');
        \draw[d,pro](3') to (4');

        \draw[->](1) to node[above,la]{$g$} (1');
         \draw[->](2) to node[above,la]{$g$} (2');
         \draw[->](3) to node[below,la]{$g'$} (3');
         \draw[->](4) to node[below,la]{$g'$} (4');

         \node[la] at ($(1)!0.5!(2')+(0,2pt)$) {\rotatebox{270}{$=$}};
        \node[la] at ($(2)!0.5!(3')$) {$\beta$};
        \node[la] at ($(3)!0.5!(4')-(0,2pt)$) {\rotatebox{270}{$=$}};

    \node[right of=1'](1) {$FGd$}; 
    \node[above of=1](0) {$FGd$};
        \node[below of=1](2) {$d$};
        \node[below of=2](3) {$d'$};
        \node[below of=3](4) {$FGd'$};
        \node[below of=4](5) {$FGd'$};

        \draw[d,pro](0) to (1);
        \draw[->,pro](1) to node[right,la]{$u_d$} (2);
        \draw[->,pro](2) to node[right,la]{$w$} (3);
        \draw[->,pro](3) to node[right,la]{$u_{d'}^{-1}$} (4);
        \draw[d,pro](4) to (5);

        \draw[->](0') to node[above,la]{$FGg$} (0);
        \draw[->](1') to node[above,la]{$f_d^{-1}$} (1);
        \draw[d](2') to (2);
        \draw[d](3') to (3); 
        \draw[->](4') to node[below,la]{$f_{d'}^{-1}$} (4);
        \draw[->](5') to node[below,la]{$FGg'$} (5);

        \node[la] at ($(1')!0.5!(2)$) {$\varphi_d^{-1,h}$};
        \node[la] at ($(2')!0.5!(3)+(2pt,0)$) {$=$};
        \node[la] at ($(3')!0.5!(4)$) {$\varphi_{d'}^{-1,hv}$};
        \node[la] at ($(0')!0.5!(1)+(0,2pt)$) {$\alpha_g$}; 
        \node[la] at ($(5')!0.5!(4)-(0,2pt)$) {$\alpha^{-1}_{g'}$};

        \node[right of=0](0') {$FGd$}; 
        \node[below of=0'](1') {$FGd$};
        \node[right of=4](4') {$FGd'$};
        \node[below of=4'](5') {$FGd'$};

        \draw[d,pro](0') to (1');
        \draw[->,pro](1') to node[right,la]{$FGw$} (4');
        \draw[d,pro](4') to (5');
        \draw[d](0') to (0);
        \draw[d](1') to (1);
        \draw[d](4') to (4);
        \draw[d](5') to (5);

        \node[la] at ($(0')!0.5!(1)$) {$=$};
        \node[la] at ($(1')!0.5!(4)+(2pt,0)$) {$\beta^{-1}_w$};
        \node[la] at ($(4')!0.5!(5)$) {$=$};
    \end{tz}
    \end{itemize}
    Note that $G$ is well-defined since $F$ is injective on objects and faithful on horizontal and vertical morphisms---as it is a canonical cofibration---, and it is fully faithful on squares---since $F$ is a gregarious double equivalence. Moreover, we see that $GF=\id_{\bA}$ as desired, and that $G$ respects composites of horizontal (resp.\ vertical) morphisms by construction. Finally, to see that $G$ preserves compositions of squares, note that composing two of the diagrams as depicted above corresponding to composable squares $\beta$ and $\gamma$ causes several squares to cancel out, and we obtain the diagram corresponding to their composite.
\end{proof}

The next several lemmas prove some useful facts regarding the behavior of the Gray tensor product of double categories with respect to our relevant classes of maps. The motivation behind these lies in the fact that our candidate for path object of a double category $\bA$ will be given by the internal pseudo-hom $\llbracket\Sq\Eadj,\bA\rrbracket_\mathrm{ps}$, which is intimately related to the Gray tensor product as stated in \cref{prop:graypseudo}.

\begin{lem} \label{tensoroftrivfib}
    If $P\colon \bA\to \bB$ is a canonical trivial fibration and $\bX$ is a double category, then the induced double functor $P\otimes \bX\colon \bA\otimes\bX\to \bB\otimes \bX$ is a canonical trivial fibration. 
\end{lem}

\begin{proof}
    Consider the commutative diagram in $\dblcat$,
    \begin{tz}
        \node[](1) {$\bA\otimes \bX$}; 
        \node[below of=1](2) {$\bA\times \bX$}; 
        \node[right of=1,xshift=1.4cm](3) {$\bB\otimes \bX$}; 
        \node[below of=3](4) {$\bB\times \bX$}; 

        \draw[->](1) to node[above,la]{$P\otimes \bX$} (3);
        \draw[->](2) to node[below,la]{$P\times \bX$} (4);
        \draw[->](1) to node[left,la]{$\pi_{\bA,\bX}$} (2);
        \draw[->](3) to node[right,la]{$\pi_{\bB,\bX}$} (4);
    \end{tz}
    where the vertical double functors are the canonical projections and are canonical trivial fibrations by \cite[Lemma 7.3]{whi}. Moreover, the product $P\times \bX$ is a canonical trivial fibration as a pullback of the canonical trivial fibration $P$ along the projection $\bB\times \bX\to \bB$. Then, by $2$-out-of-$3$ and \cref{lem:trivfibarewe}, we get that $P\otimes \bX$ is a gregarious double equivalence. In particular, it is fully faithful on squares. 

    To show that $P\otimes \bX$ is a canonical trivial fibration, it remains to show that it is surjective on objects and full on both horizontal and vertical morphisms. A detailed description of the data of the double category $\bA\otimes \bX$ can be found, for instance, in \cite[Description 7.1]{whi}. Using this, we can see that $\bA\otimes\bX$ and $\bA\times\bX$ have the same objects. Moreover, their horizontal  morphisms are generated by pairs of the form  $(f,\id)$ and $(\id,g)$ for horizontal morphisms $f$ in $\bA$ and $g\in \bB$, the only difference being that $\bA\times\bX$ has an additional relation that makes these two types of generators commute; a similar description can be given for the vertical morphisms. Then, the fact that $P\times \bX$ is surjective on objects and full on horizontal and vertical morphisms implies that so is $P\otimes \bX$.
\end{proof}

\begin{notation}
Given two double functors $J\colon \bA\to \bB$ and $I\colon \bX\to \bY$, we write $J\, \square\, I$ for the pushout-product double functor
\[ J\,\square\, I\colon \bA\otimes \bY\amalg_{\bA\otimes \bX}\bB\otimes \bX\to \bB\otimes \bY \]
\end{notation}

\begin{lem} \label{pushprodofcof}
    The pushout-product of two canonical cofibrations is a canonical cofibration.
\end{lem}

\begin{proof}
    This is contained in the proof of \cite[Theorem 7.8]{whi}, since that theorem deals with the same generating set of cofibrations as in \cref{gen_cofib}.
\end{proof}

\begin{lem} \label{pushprodofspecifictrivcof}
    Let $J\colon \bA\to \bB$ be a canonical cofibration such that $\bA$ is canonically cofibrant and $J$ admits a retraction $P\colon \bB\to \bA$ that is a canonical trivial fibration. Then, for any canonical cofibration $I\colon \bX\to \bY$, the pushout-product 
    \[ J\,\square\, I\colon \bA\otimes \bY\amalg_{\bA\otimes \bX}\bB\otimes \bX\to \bB\otimes \bY \]
    is a canonical cofibration and a gregarious double equivalence. 
\end{lem}

\begin{proof}
 Consider the following commutative diagram
    \begin{tz}
\node[] (1)     {$\bA\otimes \bX$};
\node[right of=1,xshift=1cm] (2)   {$\bB\otimes \bX$};
\node[right of=2,xshift=1cm] (3) {$\bA\otimes \bX$};
\node[below of=1] (4) {$\bA\otimes \bY$};
\node[below of=2] (5) {$\bullet$};
\node[below of=3] (6) {$\bA\otimes \bY$};
\node[below of=5,xshift=1cm] (7) {$\bB\otimes \bY$};

\draw[->] (1) to node[left,la]{$\bA\otimes I$} (4);
\draw[->] (2) to node[right,la]{$K$} (5);
\draw[->] (3) to node[right,la]{$\bA\otimes I$} (6);
\draw[->] (1) to node[above,la]{$J\otimes \bX$} (2);

\draw[->] (2) to node[above, la]{$P\otimes \bX$} (3);
\draw[->] (4) to node[below,la]{} (5);
\draw[->,dashed] (5) to node[below, la]{$Q$} (6);
\draw[->,dashed] (5) to node[left, la,yshift=-2pt]{$J\,\square\, I$} (7);

\draw[->,bend right=15] (4) to node[below,la,xshift=-10pt]{$J\otimes \bY$} (7);
\draw[->] (7) to node[right,la,yshift=-2pt]{$P\otimes \bY$} (6);

\pushout{5};
\pushout{6};
\end{tz} 
First, by \cref{pushprodofcof}, we have that $\bA\otimes I$ and $J\,\square\,I$ are canonical cofibrations. Hence $K$ is also a canonical cofibration as a pushout of $\bA\otimes I$. Then, by \cref{tensoroftrivfib}, we have that $P\otimes \bX$ and $P\otimes \bY$ are canonical trivial fibrations. Hence $Q$ is also a canonical trivial fibration by \cref{prop:pushoutoftrivfib} as the pushout of the canonical trivial fibration $P\otimes \bX$ along the canonical cofibration $K$. Hence, by $2$-out-of-$3$, we conclude that the canonical cofibration $J\,\square\,I$ is also a gregarious double equivalence. 
\end{proof}

\begin{rem} \label{J0isspecific}
Note that the projection double functors $\Sq \Eadj\to \mathbbm{1}$, $\bH\Sigma I\to \bH \mathbbm{2}$, and $\bV\Sigma I\to \bV \mathbbm{2}$ are canonical trivial fibrations, and that the double categories $\mathbbm{1}$, $\bH \mathbbm{2}$, and $\bV \mathbbm{2}$ are canonically cofibrant. Hence, the canonical cofibrations of the set $\cJ_0$ all satisfy the condition of~$J$ in \cref{pushprodofspecifictrivcof}.
\end{rem}

\begin{lem} \label{lem:pathtrivfib}
    For every double category $\bA$, the double functor $\llbracket\Sq\Eadj,\bA\rrbracket_\mathrm{ps}\to \bA$ induced by the inclusion $\mathbbm{1}\to \Sq\Eadj$ is a canonical trivial fibration. 
\end{lem}

\begin{proof}
    The double functor $\llbracket\Sq\Eadj,\bA\rrbracket_\mathrm{ps}\to \bA$ is a canonical trivial fibration if and only if it has the right lifting property with respect to any canonical cofibration $I$. By the adjunction $-\otimes \bB\dashv \llbracket\bB,-\rrbracket_\mathrm{ps}$ which is natural in $\bB$, this is the case if and only if $\bA$ has the right lifting property with respect to the pushout-product $J\,\square\,I$ of $J\colon \mathbbm{1}\to \Sq\Eadj$ and any canonical cofibration $I$. We show the latter. 
    
    Let $I$ be a canonical cofibration and fix $J\colon \mathbbm{1}\to \Sq\Eadj$. By \cref{pushprodofspecifictrivcof,J0isspecific}, the pushout-product $J\,\square\,I$ is both a canonical cofibration and a gregarious double equivalence. Hence, by \cref{liftingofcofge2}, the double category $\bA$ has the right lifting property against $J\,\square\,I$, as desired.
\end{proof}

\begin{lem} \label{lem:pathnaive0}
    For every double category $\bA$, the double functor $\llbracket\Sq\Eadj,\bA\rrbracket_\mathrm{ps}\to \bA\times \bA$ induced by the inclusion $\mathbbm{1}\amalg\mathbbm{1}\to \Sq\Eadj$ is a gregarious fibration. 
\end{lem}

\begin{proof}
    The double functor $\llbracket\Sq\Eadj,\bA\rrbracket_\mathrm{ps}\to \bA\times \bA$ is a gregarious fibration if and only if it has the right lifting property with respect to any canonical cofibration $J$ in the set $\cJ_0$. By the adjunction $-\otimes \bB\dashv \llbracket\bB,-\rrbracket_\mathrm{ps}$ which is natural in $\bB$, this is the case if and only if $\bA$ has the right lifting property with respect to the pushout-product $J\,\square\,I$ of any $J$ in~$\cJ_0$ and the canonical cofibration $I\colon \mathbbm{1}\amalg\mathbbm{1}\to \Sq\Eadj$. The proof of this latter fact now proceeds as in the previous lemma.
\end{proof}

\begin{prop} \label{lem:pathobject0}
    For every double category $\bA$, there is a factorization of the diagonal 
    \[ \bA\xrightarrow{W} \llbracket\Sq\Eadj,\bA\rrbracket_\mathrm{ps}\xrightarrow{P} \bA\times \bA \]
    induced by $\mathbbm{1}\amalg\mathbbm{1}\to \Sq\Eadj\to \mathbbm{1}$
    as a gregarious double equivalence $W$ followed by a gregarious fibration $P$.
\end{prop}

\begin{proof}
    The fact that $P$ is a gregarious fibration is the content of \cref{lem:pathnaive0}. The fact that $W$ is a gregarious double equivalence follows by $2$-out-of-$3$, as $W$ is a section of the canonical trivial fibration $\llbracket\Sq\Eadj,\bA\rrbracket_\mathrm{ps}\to \bA$ from \cref{lem:pathtrivfib}.  
\end{proof}

We have now provided a different proof of Campbell's theorem from \cite{Camp}. 

\begin{theorem}
    There is a cofibrantly generated model structure on $\dblcat$, called the \emph{gregarious model structure} and denoted by $\dblcat_\mathrm{greg}$, in which the weak equivalences are the gregarious double equivalences and the trivial fibrations are the canonical ones. Furthermore, the fibrations are the gregarious fibrations, and all objects are fibrant.
\end{theorem}

\begin{proof}
    We verify the conditions from \cref{thm:GMSV}. \ref{ax:trivfib} follows from \cref{lem:trivfibarewe} together with the fact that fibrant replacements can be taken to be trivial as all objects are naive fibrant by \cref{lem:descrnfib0}. \ref{2of6Wf} and \ref{accessibility} are the content of \cref{2of6accessible}. \ref{anwe} is proven in \cref{lem:pathobject0} using the equivalent characterization from \cref{rem:pathobject}. Finally, \ref{fibwe} is given by \cref{lem:fibwe0}. 
\end{proof}

\begin{rem} \label{rem:an0istrivcof}
    As all objects are fibrant, we have that $(\an_0,\nfib_0)$ coincides with the weak factorization system of trivial cofibrations and fibrations for the gregarious model structure. 
\end{rem}

    The following result now follows directly from \cref{thm:leftproper}.

\begin{theorem} \label{gregleftproper}
    The gregarious model structure $\dblcat_\mathrm{greg}$ is left proper.
\end{theorem}

\begin{theorem}\label{thm:gregariousmonoidal}
    The gregarious model structure $\dblcat_\mathrm{greg}$ is monoidal with respect to the Gray tensor product.
\end{theorem}

\begin{proof}
    This follows from \cref{pushprodofcof,pushprodofspecifictrivcof,J0isspecific,rem:an0istrivcof}. 
\end{proof}

\begin{rem}
    None of the model structures that we consider in this paper---hence, in particular, the gregarious model structure---is compatible with the cartesian product. This is due to the fact that, using this monoidal structure, the pushout-product of two cofibrations is not necessarily a cofibration, as shown in \cite[Remark 7.1]{MSV}.
\end{rem}

\subsection{The canonical double equivalences}\label{subsec:generalcase}

We can now reap the benefits of the homotopical results we have established so far, to answer the following question: \emph{what should it mean for a double functor $F\colon\bA\to\bB$ to be a canonical equivalence of double categories?}

As explained in the introduction, unlike other categorical structures such as categories, 2-categories, or even $n$-categories, it is not clear what the canonical notion of equivalence should be in the setting of double categories. This difficulty is due to the existence of two different types of morphisms, and arises when one tries to determine what type of surjectivity on objects to require of the double functor $F$. While for categories we ask for surjectivity up to isomorphism, and for 2-categories we need surjectivity up to (internal) equivalence, the fact that double categories have both horizontal and vertical equivalences gives rise to a variety of choices and combinations one could ask for, and obscures which is the ``canonical'' one. 

Instead of trying to emulate the ``essentially surjective and fully faithful'' characterization of equivalences, one could alternatively think of equivalences of categories as functors $F\colon\cC\to\cD$ such that there exists another functor $G\colon\cD\to\cC$ together with natural isomorphisms $\id_\cC\cong GF$ and $\id_\cD\cong FG$. Indeed, an analogous description holds for 2-categories: a 2-functor $F\colon\cC\to\cD$ is a biequivalence precisely if there exists a pseudo-functor $G\colon\cD\to\cC$ and pseudo-natural equivalences $\id_\cC\simeq GF$ and $\id_\cD\simeq FG$. However, the same issue arises if we try to translate this description to the setting of double categories, as we now have both \emph{horizontal} and \emph{vertical} (pseudo) natural transformations, and one is faced with the same choices as above. 

\begin{rem}
Before we present our argument, we briefly remind the reader of other equivalences between double categories that have been considered in the literature:
\begin{itemize}[leftmargin=0.8cm]
    \item levelwise equivalences using the fact that double categories are categories internal to $\cat$; this gives a truncation of the equivalences in \cite[Theorem 7.17]{FPP};
    \item equivalences in a $2$-category \cite{LackTriv} using the fact that there is a 2-category of double categories, double functors, and horizontal natural transformations \cite[\S 8.4]{FPP}; this yields a stricter version of the ``horizontal equivalences'' of Grandis--Par\'e \cite[\S 3.5.5]{Grandis};
    \item equivalences of algebras \cite{LackTriv} using the fact that double categories form the algebras for a certain 2-monad over $\cat(\graph)$ \cite[\S 9]{FPP}.
\end{itemize}
While all of these notions are natural and highlight relevant features and perspectives on $\dblcat$, none of them produce a truly canonical notion of double categorical equivalence. Indeed, the surjectivity on objects we obtain from translating these notions is inherently 1-dimensional, as it relies on \emph{isomorphisms}, as opposed to \emph{equivalences}. Our knowledge from the setting of 2-category theory already tells us that this is not what we need.  
\end{rem}

We argue that the canonical equivalences of double categories should be the gregarious double equivalences, which were introduced by Campbell in \cite{Camp} to serve this purpose. To support this claim, we remind the reader of two uncontroversial assumptions that we are making from the start:
\begin{enumerate}[leftmargin=0.8cm]
    \item Any candidate notion of ``canonical double equivalence'' must form the class of weak equivalences in some model structure on $\dblcat$ (as is the case for the equivalences of categories \cite{rezk} and the biequivalences of 2-categories \cite{Lack2Cat,LackBicat}).
    \item The trivial fibrations in such a model structure must be the canonical trivial fibrations of \cref{def:trivfibs} (which are defined in direct analogy to the trivial fibrations in $\cat$ and $\twocat$).
\end{enumerate}

With this in mind, the following result, showing how any model structure with the canonical trivial fibrations is related to the gregarious model structure, gives evidence to our claim. 

\begin{theorem} \label{thm:localization}
    Every model structure on $\dblcat$ whose trivial fibrations are the canonical ones is a localization of the gregarious model structure $\dblcat_\mathrm{greg}$. Moreover, if the model structure is cofibrantly generated, it is further a left Bousfield localization of $\dblcat_\mathrm{greg}$, and so the weak equivalences (resp.~fibrations) between fibrant objects are precisely the gregarious double equivalences (resp.~gregarious fibrations).
\end{theorem}

\begin{proof}
    Suppose that $\dblcat$ is endowed with a model structure whose trivial fibrations are the canonical ones. We show that the identity adjunction
\begin{tz}
\node[](1) {$\dblcat_\mathrm{greg}$}; 
\node[right of=1,xshift=2cm](2) {$\dblcat$};

\draw[->,bend left=20] ($(1.east)+(0,4pt)$) to node[above,la]{$\id$} ($(2.west)+(0,4pt)$);
\draw[->,bend left=20] ($(2.west)-(0,4pt)$) to node[below,la]{$\id$} ($(1.east)-(0,4pt)$);

\node[la] at ($(1.east)!0.5!(2.west)$) {$\bot$}; 
\end{tz}
    is a Quillen pair. First note that $\id\colon \dblcat_\mathrm{greg}\to \dblcat$ preserves cofibrations, as both model structures have  as cofibrations the canonical ones. Moreover, as the double functors $\Sq \Eadj\to \mathbbm{1}$, $\bH\Sigma I\to \bH \mathbbm{2}$, and $\bV\Sigma I\to \bV \mathbbm{2}$ are canonical trivial fibrations, by $2$-out-of-$3$, the double functors in $\cJ_0$ will also be trivial cofibrations in $\dblcat$. Hence $\id\colon \dblcat_\mathrm{greg}\to \dblcat$ preserves trivial cofibrations, and thus it is left Quillen. In addition, the derived counit at a fibrant object $\bA\in \dblcat$ is given by a canonical cofibrant replacement of $\bA$ in $\dblcat_\mathrm{greg}$ and can be chosen to be a canonical trivial fibration. Hence it is also a weak equivalence in $\dblcat$. This shows that $\dblcat$ is a localization of $\dblcat_\mathrm{greg}$

    If the model structure on $\dblcat$ is cofibrantly generated, then by \cref{prop:locvsBousfieldloc}, it coincides with a left Bousfield localization of $\dblcat_\mathrm{greg}$ and so the weak equivalences and fibrations between fibrant objects admit the desired description by \cref{thm:existleftBousfield}.
\end{proof}

In particular, this shows that the gregarious double equivalences give the smallest class of weak equivalences on $\dblcat$ compatible with the canonical trivial fibrations, and hence they can be considered to be the class of ``canonical equivalences'' of double categories.

\begin{rem}
    While \cref{def:gregwe} provides an explicit characterization of the gregarious double equivalences in the spirit of the familiar ``essentially surjective and fully faithful'' description, a categorically-minded reader might wonder whether they could also be described in terms of the existence of a certain (weak) inverse. This will be addressed in future work by the second author.
\end{rem}


\begin{rem}
In the slides where the notion of gregarious double equivalence first appeared \cite{Camp}---albeit under a different name---, Campbell suggests another justification for why these equivalences are canonical: the gregarious model structure is the model structure on $\dblcat$ with the canonical trivial fibrations in which all objects are fibrant, in parallel with the cases of $\cat$ and $\twocat$.
\end{rem}

\begin{rem}
Inspired by a talk on this paper, recent work of Tom Leinster \cite{leinsterspans} gives an independent, additional justification that further supports the conclusion that the gregarious double equivalences are the canonical notion of equivalence for double categories. He shows that in the settings of categories, monoidal categories, and bicategories, two objects $A$ and $B$ are equivalent precisely if there exists a span $A\leftarrow C\rightarrow B$ consisting of trivial fibrations---in his work, these are called ``surjective equivalences''. Then, \cite[Theorem 5.8]{leinsterspans} studies the case of double categories, showing that this happens if and only if the double categories $A$ and $B$ are gregarious double equivalent.
\end{rem}

\section{Constructing model structures on double categories}\label{section:generalMS}

Having settled the question of the canonical double categorical equivalences in the previous section, we now turn our attention to all possible \emph{cofibrantly generated} model structures on double categories whose trivial fibrations are the canonical ones, and aim to give a detailed and streamlined set of guidelines for constructing them in \cref{subsec:recipe}. 
Our recipe has the additional advantage of yielding an explicit description of the fibrant objects in terms of lifting conditions---rather than a local condition, as obtained from its construction as a Bousfield localization. Lastly, in \cref{subsec:gray} we give sufficient conditions for these model structures to be monoidal with respect to the Gray tensor product. 

\subsection{A recipe for cofibrantly generated model structures}\label{subsec:recipe}

We  restrict our attention to cofibrantly generated model structures. Our goal is to provide a user-friendly recipe to construct these model structures, which will rely on \cref{thm:GMSV}. Looking at the statement of this theorem, we find that in our case most of the data is already determined by our motivation:
\begin{itemize}[leftmargin=0.8cm]
    \item The category $\cC$ must be the category $\dblcat$ of double categories and double functors, where we want to construct our model structure.
    \item Since we wish to consider the canonical trivial fibrations, we can let $\cI$ be the set of generating canonical cofibrations of \cref{gen_cofib}.
    \item By \cref{thm:localization}, the weak equivalences between fibrant objects in all these model structures are the gregarious double equivalences; hence $\Wf$ will consist exactly of these double functors.
\end{itemize}

Therefore, the only data left to determine is the weak factorization system $(\an,\nfib)$, which must be cofibrantly generated by a set $\cJ$ of double functors.  At this point, the reader could skip to \cref{thm:genmodel} and read the requirements on a set $\cJ$ that will ensure the existence of a cofibrantly generated model structure on $\dblcat$ whose trivial fibrations are the canonical ones. However, a more likely scenario is that the user has some additional features in mind for the desired model structure. We will now give some guidelines on how to use those features to select a set $\cJ$.

\noindent\textbf{\underline{Step 1: Choosing a candidate for $\cJ$}.}

\begin{itemize}[leftmargin=0.8cm]
\item[(a)] \textit{If the user has a desired class of fibrant double categories,} they should find a set $\cJ$ with the property that a double category is fibrant if and only if it has the right lifting property with respect to all double functors in $\cJ$.
\item[(b)] \textit{If the user has a desired set $S$ of double functors that they wish to localize at}, they should pick $\cJ=S$.
\end{itemize}

\noindent\textbf{\underline{Step 2: Making adjustments to $\cJ$}.}
The double functors in $\cJ$ need to be modified to ensure that they are canonical cofibrations with canonically cofibrant domain. This is not hard to do in practice when the double functors are determined from a lifting property as in item (a) above; in fact, these double functors will be trivial cofibrations if the desired model structure indeed exists. However, it may be less evident that a localizing class $\cJ$ obtained as in item (b) can be similarly modified without altering its intended purpose. This is justified by the following remark.

\begin{rem}\label{rem:cofdomain}
Given a double functor $J\colon \bA\to \bB$ in $\cJ$, consider a commutative square of the form 
\begin{tz}
\node[](1) {$\bA^\mathrm{cof}$}; 
\node[right of=1,xshift=.3cm](2) {$\bB^\mathrm{cof}$}; 
\node[below of=1](1') {$\bA$}; 
\node[below of=2](2') {$\bB$}; 

\draw[->] (1) to node[above,la]{$J^\mathrm{cof}$} (2); 
\draw[->] (1) to node[left,la]{$\mathrm{inj}(\cI)\ni$} (1'); 
\draw[->] (2) to node[right,la]{$\in \mathrm{inj}(\cI)$} (2'); 
\draw[->] (1') to node[below,la]{$J$} (2');(2); 
\end{tz}
with $J^\mathrm{cof}\in \mathrm{cof}(\cI)$. Such a square can be obtained by first choosing a canonical cofibrant replacement $\bA^\mathrm{cof}\to \bA$ in $\mathrm{inj}(\cI)$ and then factoring the resulting composite using the weak factorization system $(\mathrm{cof}
(\cI),\mathrm{inj}(\cI))$; hence $J^\mathrm{cof}$ is a canonical cofibration with canonically cofibrant domain. Moreover, since maps in $\mathrm{inj}(\cJ)$ are trivial fibrations, we have that $J$ is a weak equivalence in the resulting model structure if and only if $J^\mathrm{cof}$ is so, and thus localizing by $J$ or by $J^\mathrm{cof}$ produces the same result.
\end{rem}

\begin{rem} \label{rem:J0vsJ}
The double functors in $\cJ_0$ will always be trivial cofibrations in this new model structure we want to build, as it will be a localization of the gregarious model structure by \cref{thm:localization}. Hence it is not restrictive to assume that $\cJ_0\subseteq \an$. It follows that we have inclusions 
\[ \an_0\subseteq \an \quad \text{and} \quad \nfib\subseteq \nfib_0. \]
In particular, one could always ask that $\cJ_0\subseteq \cJ$ in order to ensure the inclusion $\cJ_0\subseteq \an$.
\end{rem}

We now prove that, in this setting, conditions \ref{ax:trivfib}-\ref{accessibility} and \ref{fibwe} of \cref{thm:GMSV} hold automatically. We start with the following lemma: a general, model categorical argument expressing a relation between the two factorization systems we consider.

\begin{lem}\label{lem:pushfibobj}
If $P\colon \bA\to \bB$ is a canonical trivial fibration and $\bA$ is naive fibrant, then $\bB$ is naive fibrant.
\end{lem}

\begin{proof}
    We need to show that for every double functor $J\colon \bX\to \bY$ in the generating set~$\cJ$ of the weak factorization system $(\an,\nfib)$, there is a lift $\widehat{F}$ in the diagram below left.
\begin{tz}
\node[](1) {$\bX$}; 
\node[right of=1](2) {$\bB$}; 
\node[below of=1](1') {$\bY$}; 

\draw[->] (1) to node[above,la]{$F$} (2); 
\draw[->] (1) to node[left,la]{$J$} (1'); 
\draw[->,dashed] (1') to node[right,la,yshift=-5pt]{$\widehat{F}$} (2); 

\node[right of=2,xshift=1cm](1) {$\emptyset$}; 
\node[right of=1](2) {$\bA$}; 
\node[below of=1](1') {$\bX$}; 
\node[below of=2](2') {$\bB$}; 

\draw[->] (1) to (2); 
\draw[->] (1) to (1'); 
\draw[->] (2) to node[right,la]{$P$} (2'); 
\draw[->] (1') to node[below,la]{$F$} (2');
\draw[->,dashed] (1') to node[left,la,yshift=5pt]{$G$} (2); 

\node[right of=2,xshift=1cm](1) {$\bX$}; 
\node[right of=1](2) {$\bA$}; 
\node[below of=1](1') {$\bY$}; 

\draw[->] (1) to node[above,la]{$G$} (2); 
\draw[->] (1) to node[left,la]{$J$} (1'); 
\draw[->,dashed] (1') to node[right,la,yshift=-5pt]{$\widehat{G}$} (2); 
\end{tz}
However, since $\bX$ is canonically cofibrant and $P$ is a canonical trivial fibration, we get a lift $G$ in the commutative diagram in the center. Then, as $\bA$ is naive fibrant, we further get a lift in the diagram above right. Setting $\widehat{F}\coloneqq P \widehat{G}$ we get the desired lift, since
\[ \widehat{F} J=P\widehat{G} J=P G=F. \qedhere \]
\end{proof}

We can use this to show that condition \ref{ax:trivfib} of \cref{thm:GMSV} is satisfied. Adapting \cref{def:we} to our setting, the class $\cW$ of weak equivalences is given by those double functors which admit a naive fibrant replacement that is a gregarious double equivalence. 

\begin{prop} \label{trivfibincW}
    Every canonical trivial fibration is in the class of weak equivalences~$\cW$.
\end{prop}

\begin{proof}
    Let $P\colon \bA\to \bB$ be a canonical trivial fibration in $\dblcat$, and let $\iota_\bA\colon \bA\to \bA'$ be a naive fibrant replacement with respect to the weak factorization system $(\an,\nfib)$. We consider the following pushout in $\dblcat$ and show that it provides a naive fibrant replacement of $P$ such that $P'$ is a gregarious double equivalence. 
    \begin{tz}
\node[](1) {$\bA$}; 
\node[right of=1](2) {$\bB$}; 
\node[below of=1](1') {$\bA'$}; 
\node[below of=2](2') {$\bB'$}; 

\pushout{2'};

\draw[->] (1) to node[above,la]{$P$} (2); 
\draw[->] (1) to node[left,la]{$\iota_\bA$} (1'); 
\draw[->] (2) to node[right,la]{$\iota_\bB$} (2'); 
\draw[->] (1') to node[below,la]{$P'$} (2');
\end{tz}
Since $\an$ is saturated, the double functor $\iota_\bB\colon \bB\to \bB'$ is also an anodyne extension. Moreover, since $\an\subseteq \cof(\cI)$, the pushout $P'$ of the canonical trivial fibration $P$ is also a canonical trivial fibration by \cref{prop:pushoutoftrivfib} and so also a gregarious double equivalence by \cref{lem:trivfibarewe}. Finally, as $P'$ is a canonical trivial fibration and $\bA'$ is naive fibrant, then $\bB'$ is also naive fibrant by \cref{lem:pushfibobj}. 
\end{proof}

While conditions (2), (3), and (5)  of \cref{thm:GMSV} will always be satisfied as well (see the proof of \cref{thm:genmodel}), the same cannot be said of condition \ref{anwe}, and this is the place where our choice of $\cJ$ requires some care. To better understand this condition, we present equivalent formulations in the next result. 

\begin{prop} \label{equivalentcond4}
    The following conditions are equivalent:
    \begin{rome}
        \item Every double functor $J\in\an$ between naive fibrant double categories is a gregarious double equivalence.
        \item Every gregarious fibration between naive fibrant double categories is in $\nfib$.
        \item For every naive fibrant double category $\bA$, the double functor \[ \llbracket\Sq\Eadj,\bA\rrbracket_\mathrm{ps}\to \bA\times \bA \]
        induced by the canonical inclusion $\mathbbm{1}\amalg \mathbbm{1}\to \Sq \Eadj$ is in $\nfib$.
    \end{rome}
\end{prop}

\begin{proof}
    To show that (i) implies (ii), let $P\colon \bA\to \bB$ be a gregarious fibration such that $\bA$ and $\bB$ are naive fibrant. Factor $P$ as 
    \[ \bA\xrightarrow{J}\bC\xrightarrow{Q} \bB \]
    with $J\in \an$ and $Q\in \nfib$. By (i), the double functor $J$ is both a canonical cofibration and a gregarious double equivalence, and so it is a trivial cofibration in the gregarious model structure. As $P$ is a gregarious fibration, it has the right lifting property with respect to $J$, and so by the retract argument it is in $\nfib$. 

   Next we show that (ii) implies (iii). Let $\bA$ be a naive fibrant double category. By \cref{lem:pathtrivfib}, the double functor $\llbracket\Sq\Eadj,\bA\rrbracket_\mathrm{ps}\to \bA$ is a canonical trivial fibration, thus also a naive fibration. Hence it follows that $\llbracket\Sq\Eadj,\bA\rrbracket_\mathrm{ps}$ is also naive fibrant. This fact, together with \cref{lem:pathnaive0}, imply that  $\llbracket\Sq\Eadj,\bA\rrbracket_\mathrm{ps}\to \bA\times \bA$ is a gregarious fibration between naive fibrant double categories, and so by (ii), it is also in $\nfib$.

    Finally, we prove that (iii) implies (i). Note that (i) corresponds to  condition \ref{anwe} of \cref{thm:GMSV} in this context, so we can instead show the equivalent condition (4') from \cref{rem:pathobject}. Let $\bA$ be a naive fibrant double category. Then the following is a factorization of the diagonal 
    \[ \bA\xrightarrow{W} \llbracket\Sq\Eadj,\bA\rrbracket_\mathrm{ps}\xrightarrow{P} \bA\times\bA \]
    where $W$ is a gregarious double equivalence by \cref{lem:pathobject0} and $P$ is in $\nfib$ by (iii), as desired.
\end{proof}

We now show that this condition does not hold for all choices of $\cJ$. 

\begin{ex}
    Consider the set of double functors
    \[ \cJ\coloneqq \cJ_0\amalg \{\mathbbm{1}\to \bH\mathbbm{2} \} \]
    inducing a weak factorization system $(\an,\nfib)$ on $\dblcat$. Then, we have strict inclusions $\an_0\subset \an$ and $\nfib\subset \nfib_0$. Moreover, note that every double category is naive fibrant as, for every double category $\bA$, there is a lift in every diagram of the form
    \begin{tz}
\node[](1) {$\mathbbm{1}$}; 
\node[right of=1](2) {$\bA$}; 
\node[below of=1](1') {$\bH\mathbbm{2}$}; 

\draw[->] (1) to node[above,la]{$A$} (2); 
\draw[->] (1) to node[left,la]{} (1'); 
\draw[->,dashed] (1') to node[right,la,yshift=-5pt]{$\id_A$} (2); 
\end{tz}
    Hence the strict inclusion $\nfib\subset \nfib_0$ shows that there are double functors in $\nfib_0$ (between naive fibrant double categories) that are not in $\nfib$, and so $\cJ$ does not satisfy the equivalent properties of \cref{equivalentcond4}.

    However, to produce the same resulting model structure where the double functor $\mathbbm{1}\to \bH\mathbbm{2}$ is made into a weak equivalence, one can consider instead the set
    \[ \cJ\coloneqq \cJ_0\cup \{\bH\mathbbm{2}\to\Sq\Eadj \}. \]
   Indeed, since $\mathbbm{1}\to \Sq\Eadj$ is in $\cJ_0$ and hence is always a weak equivalence, by $2$-out-of-$3$ the double functor $\mathbbm{1}\to \bH\mathbbm{2}$ is a weak equivalence if and only if $\bH\mathbbm{2}\to \Sq\Eadj$ is so. Moreover, this new set $\cJ$ satisfies the equivalent properties of \cref{equivalentcond4}, which can be deduced from \cref{prop:compsetJ,prop:gpdsetJ,rmk:companionsforequivs} using that $\bH\mathbbm{2}\to\Sq\Eadj$ factors as $\bH\mathbbm{2}\to \Sq\mathbbm{2}\to \Sq\Eadj$. 
\end{ex}

Under the equivalent conditions of \cref{equivalentcond4}, we have proven the existence of the desired model structure, as we now summarize.

\begin{theorem} \label{thm:genmodel}
   Let $(\an,\nfib)$ be a weak factorization system in $\dblcat$ generated by a set $\cJ$ of double functors with canonically cofibrant domain such that $\cJ_0\subseteq\cJ\subseteq\cof(\cI)$ and the equivalent conditions of \cref{equivalentcond4} are satisfied. Then, there is a cofibrantly generated model structure on $\dblcat$ whose trivial fibrations are the canonical ones, and whose fibrant objects are the naive fibrant double categories. Furthermore, the weak equivalences (resp.~fibrations) between fibrant objects are precisely the gregarious double equivalences (resp.~gregarious fibrations). 
\end{theorem}

\begin{proof}
    We check the conditions from \cref{thm:GMSV}. Condition \ref{ax:trivfib} is \cref{trivfibincW}; conditions \ref{2of6Wf} and  \ref{accessibility} are ensured by \cref{2of6accessible}. Condition \ref{anwe} holds by assumption as one of the equivalent conditions in \cref{equivalentcond4}. Finally, as every naive fibration is in particular a gregarious fibration by \cref{rem:J0vsJ}, condition \ref{fibwe} follows from \cref{lem:fibwe0}. 
\end{proof}

As a direct consequence of \cref{thm:leftproper}, we get the following.

\begin{theorem}\label{thm:leftproper2}
    The model structure on $\dblcat$ from \cref{thm:genmodel} is left proper.
\end{theorem}

We are also able to refine our result from \cref{thm:localization}.

\begin{theorem} \label{thm:MSislocwrtJ}
    The model structure on $\dblcat$ from \cref{thm:genmodel} is the left Bousfield localization of the gregarious model structure at the set $\cJ\setminus \cJ_0$. 
\end{theorem}

\begin{proof}
    By \cref{thm:localization}, the model structure on $\dblcat$ from \cref{thm:genmodel} is a localization of the gregarious model structure. Hence, by \cref{prop:locvsBousfieldloc}, it coincides with the left Bousfield localization of the gregarious model structure at the set $\cJ$ of generating anodyne extensions. Since the double functors in $\cJ_0$ are already weak equivalences in the gregarious model structure, we do not need to localize with respect to those. 
\end{proof}

Finally, we show that every cofibrantly generated model structure on $\dblcat$ with the desired trivial fibrations can be obtained using our recipe. 

\begin{theorem}\label{thm:cleverJ}
    Every cofibrantly generated model structure on $\dblcat$ whose trivial fibrations are the canonical ones can be obtained through the method of \cref{thm:genmodel} by finding a convenient set $\cJ$.
\end{theorem}

\begin{proof}
    By \cref{thm:localization}, such a model structure is a localization of the gregarious model structure, and thus, as it is combinatorial, it is also a left Bousfield localization of $\dblcat_\mathrm{greg}$ at a set $\cJ$ of anodyne extensions in $\dblcat$ by \cref{prop:locvsBousfieldloc}. Without loss of generality, as every $\cJ_0$-cofibration is a trivial cofibration in $\dblcat$, we can assume that $\cJ_0\subseteq \cJ$. Moreover, by \cref{rem:cofdomain}, we can further assume that the double functors in $\cJ$ have canonically cofibrant domain. Then the weak factorization system  $(\an,\nfib)$ generated by $\cJ$ satisfies condition (ii) of \cref{equivalentcond4}, as the fibrations between fibrant objects in the left Bousfield localization agree with the original fibrations in $\dblcat_\mathrm{greg}$.
\end{proof}

\subsection{Compatibility with Gray tensor product}\label{subsec:gray}

We conclude this section with a study of when the model structure on $\dblcat$ from \cref{thm:genmodel} is monoidal for the Gray tensor product. Since \cref{pushprodofcof} already established that the pushout-product of two canonical cofibrations is again a canonical cofibration, our goal is to find a minimal list of conditions that ensures the pushout-product of a canonical cofibration with a trivial cofibration yields a trivial cofibration.

\begin{defn}
    A class $\cA$ of canonically cofibrant double categories is \emph{saturated by canonical cofibrations} if it satisfies the following conditions:
    \begin{rome}
         \item for every pushout diagram in $\dblcat$
         \begin{tz}
\node[](1) {$\bA$}; 
\node[right of=1](2) {$\bB$}; 
\node[below of=1](1') {$\bC$}; 
\node[below of=2](2') {$\bD$}; 

\pushout{2'};

\draw[->] (1) to node[above,la]{} (2); 
\draw[->] (1) to node[left,la]{$I$} (1'); 
\draw[->] (2) to node[right,la]{} (2'); 
\draw[->] (1') to node[below,la]{} (2');
\end{tz}
        where $I$ is a canonical cofibration and $\bA$, $\bB$, and $\bC$ are in $\cA$, then $\bD$ is in $\cA$, 
        \item for every regular cardinal $\lambda$ and every diagram $\bA_{(-)}\colon \lambda\to \dblcat$ such that for every $\kappa\leq \mu<\lambda$ the map $\bA_\kappa\to \bA_\mu$ is a canonical cofibration, and that for every $\kappa<\lambda$, $\bA_\kappa$ is in $\cA$, then $\colim_{\kappa<\lambda}\bA_\kappa$ is also in $\cA$, 
        \item for every retract diagram in $\dblcat$
        \[ \bA\xrightarrow{I} \bB\to \bA \]
        where $I$ is a canonical cofibration and $\bB$ is in $\cA$, then $\bA$ is also in $\cA$.
    \end{rome}
\end{defn}

\begin{lem} \label{prop:Cisinskiargument}
    Let $\cA$ be a class of canonically cofibrant double categories that is saturated by canonical cofibrations. If the double categories $\emptyset$, $\mathbbm{1}$, $\bH\mathbbm{2}$, $\bV\mathbbm{2}$, and $\bH\mathbbm{2}\times \bV\mathbbm{2}$ are in $\cA$, then $\cA$ is the class of all canonically cofibrant double categories.
\end{lem}

\begin{proof}
    This follows from the small object argument using the generating set of canonical cofibrations from \cref{gen_cofib}. 
\end{proof}

\begin{lem} \label{prop:satbycof}
    Let $\cJ$ be a set as in \cref{thm:genmodel}, and let $\cA$ be the class of all canonically cofibrant double categories $\bX$ such that $J\otimes \bX$ is in $\cW$, for all $J\in \cJ$. Then $\cA$ is saturated by canonical cofibrations.
\end{lem}

\begin{proof}
    This follows from the fact that, for any canonically cofibrant double category $\bX$, the functor $-\otimes\bX\colon \dblcat\to \dblcat$ preserves canonical cofibrations by \cref{pushprodofcof}, and preserves colimits as it is a left adjoint. 
\end{proof}

The conditions we seek are identified in the following lemma.

\begin{prop} \label{lemma:monoidalconditions} 
    For a set $\cJ$ as in \cref{thm:genmodel}, the following conditions are equivalent: 
    \begin{rome}
        \item For all $J\in \cJ$ and $\bX\in \{ \bH\mathbbm{2}, \bV\mathbbm{2}, \bH\mathbbm{2}\times \bV\mathbbm{2}\}$, the double functor $J\times \bX$ is in $\cW$.
        \item For all $J\in \cJ$ and $\bX\in \{ \bH\mathbbm{2}, \bV\mathbbm{2}, \bH\mathbbm{2}\times \bV\mathbbm{2}\}$, the double functor $J\otimes \bX$ is in $\cW$.
        \item For all $J\in \an$ and $\bX\in \dblcat$ canonically cofibrant, the double functor $J\otimes \bX$ is in $\cW$. 
        \item For all $J\in \cof(\cI)\cap \cW$ and $\bX\in \dblcat$ canonically cofibrant, the double functor $J\otimes \bX$ is in~$\cW$.
        \item For all $\bX\in \dblcat$  canonically cofibrant and every gregarious fibration $P\colon \bA\to \bB$ between (naive) fibrant double categories, the induced double functor \[ P_*\colon \llbracket \bX,\bA\rrbracket_\mathrm{ps}\to \llbracket \bX,\bA\rrbracket_\mathrm{ps} \] is a gregarious fibration between (naive) fibrant double categories.
        \item For all $\bX\in \dblcat$ canonically cofibrant and $\bA\in \dblcat$ (naive) fibrant, the pseudo-hom double category $\llbracket \bX,\bA\rrbracket_\mathrm{ps}$ is (naive) fibrant.  
    \end{rome}
\end{prop}

\begin{proof}
    The fact that (i) and (ii) are equivalent follows from the $2$-out-of-$3$ property of weak equivalences and the fact that the projection double functors $\bA\otimes\bX\to \bA\times \bX$ are canonical trivial fibrations for all double categories $\bA$ and $\bX$ by \cite[Lemma 7.3]{whi}. Moreover, it is clear that  (iv) implies (iii) and that (iii) implies (ii).   
    
    To show that (ii) implies (iv), let $\cA$ be the class of all canonically cofibrant double categories $\bX$ such that $J\otimes \bX$ is in $\cW$; then, by \cref{prop:satbycof}, we have that $\cA$ is saturated by canonical cofibrations. By assumption, we also know that the double categories $\emptyset$, $\mathbbm{1}$, $\bH\mathbbm{2}$, $\bV\mathbbm{2}$, $\bH\mathbbm{2}\times \bV\mathbbm{2}$ are all in $\cA$. Hence $\cA$ is the class of all canonically cofibrant double categories by \cref{prop:Cisinskiargument}.

    Then, for any canonically cofibrant double category $\bX$, the left adjoint functor 
    \[ -\otimes \bX\colon \dblcat\to \dblcat \] preserves all canonical cofibrations by \cref{pushprodofcof} and all double functors in $\cJ$ by the above argument. Hence, by \cite[Proposition E.2.14]{JoyalVolumeII}, it preserves all trivial cofibrations, as desired. With this, we have an equivalence between statements (i)-(iv). 

The fact that the double functor $J\otimes \bX$ is a canonical cofibration for all $J\in\an$ also allows us to show the equivalence between (iii) and (v),  using \cite[Lemma E.2.13]{JoyalVolumeII} and the description of the fibrations between fibrant objects from \cref{thm:genmodel}. As (v) clearly implies (vi), it only remains to show that (vi) implies (v). For this, let $X$ be a canonically cofibrant double category and $P\colon \bA\to \bB$ be a gregarious fibration between (naive) fibrant double categories. Then $\llbracket \bX,\bA\rrbracket_\mathrm{ps}$ and $\llbracket \bX,\bB\rrbracket_\mathrm{ps}$ are (naive) fibrant by (vi), and \cref{thm:gregariousmonoidal} ensures that $P_*$ is a gregarious fibration, as desired.
\end{proof}

\begin{theorem}\label{thm:gray}
Suppose we have a model structure on $\dblcat$ built from a set $\cJ$ as in \cref{thm:genmodel}. If $\cJ$ satisfies any of the equivalent conditions of \cref{lemma:monoidalconditions}, then the model structure is monoidal with respect to the Gray tensor product.
\end{theorem}

\begin{proof}
    The fact that the pushout-product of two canonical cofibrations is a cofibration is given by \cref{pushprodofcof}. Hence, it remains to show that if $J\colon\bA\to\bB$ is a trivial cofibration and $I\colon \bX\to \bY$ is a canonical cofibration in the generating set $\cI$ of \cref{gen_cofib}, then their pushout-product $J\, \square\, I$ is a weak equivalence. Note that $I$ is of the form $I\colon \bX\to \bY$ with $\bX$ canonically cofibrant, and consider the pushout diagram below.
\begin{tz}
\node[](1) {$\bA\otimes\bX$};
\node[right of=1,xshift=1.5cm](2) {$\bA\otimes\bY$};
\node[below of=1](3) {$\bB\otimes\bX$};
\node[below of=2](4) {$\bP$};
\node[below right of=4,xshift=1cm](5) {$\bB\otimes\bY$};
\draw[->] (1) to node[above,la]{$\bA\otimes I$} (2);
\draw[->] (1) to node[left,la]{$J\otimes \bX$} (3);
\draw[->] (2) to (4);
\draw[->] (3) to node[above,la]{$K$} (4);
\draw[->,dashed] (4) to node[right,la,yshift=12pt,xshift=-15pt]{$J\, \square\, I$} (5);
\draw[->,bend right] (3) to node[below,la,xshift=-10pt] {$\bB\otimes I$} (5);
\draw[->,bend left] (2) to node[right,la,xshift=2pt] {$J\otimes \bY$} (5);
  \node at ($(4)+(-.3cm,.25cm)$) {$\ulcorner$};
\end{tz}
Since $\bX$ is canonically cofibrant and $J$ is a canonical cofibration, we know that $J\, \square\, (\emptyset\to \bX)=J\otimes \bX$ is also a canonical cofibration by \cref{pushprodofcof}. As $J$ is a trivial cofibration by assumption, and $\bX$ is canonically cofibrant, the double functor $J\otimes \bX$ is a weak equivalence by \cref{lemma:monoidalconditions}. Then $J\otimes \bX$ is a trivial cofibration, and therefore so is $K$ since these are stable under pushouts. The conditions in \cref{lemma:monoidalconditions} also guarantee that $J\otimes \bY$ is a weak equivalence, and then so is $J\, \square\, I$ by $2$-out-of-$3$.
\end{proof}

\section{Application: Model structures from companions and conjoints}\label{section:ex1}

In this section we illustrate how our recipe can be used, by constructing several model structures whose homotopy theories capture a variety of 2-dimensional structures. Before we describe them, we start by recalling the dual notion to the companion pairs of \cref{def:companion}.

\begin{notation}
    Given a double category $\bA$, we write $\bA^\mathrm{v,op}$ for its \emph{vertical opposite}, i.e., the double category obtained from $\bA$ by reversing the vertical morphisms, and reversing the squares vertically; that is, their source horizontal morphism becomes their target, and conversely.
\end{notation}

\begin{defn}
    A \emph{conjoint pair} in a double category $\bA$ is the data of a double functor ${\Sq\mathbbm{2}^{\mathrm{v,op}}\to \bA}$. It consists of a tuple $(f,u,\varphi,\psi)$ of a horizontal morphism $f\colon A\to B$, a vertical morphism $u\colon B\arrowdot A$, and two squares $\varphi$ and $\psi$ in $\bA$ of the form
\begin{tz}
\node[](1) {$B$}; 
\node[below of=1](2) {$A$}; 
\node[right of=1](3) {$B$}; 
\node[right of=2](4) {$B$};

\draw[->,pro] (1) to node[left,la]{$u$} (2); 
\draw[d] (1) to (3); 
\draw[->] (2) to node[below,la]{$f$} (4); 
\draw[d,pro](3) to (4); 
 
\node[la] at ($(1)!0.5!(4)$) {$\varphi$};

\node[right of=3,xshift=2cm](1) {$A$}; 
\node[below of=1](2) {$A$}; 
\node[right of=1](3) {$B$}; 
\node[right of=2](4) {$A$};

\draw[->,pro] (3) to node[right,la]{$u$} (4); 
\draw[d] (2) to (4); 
\draw[->] (1) to node[above,la]{$f$} (3); 
\draw[d,pro](1) to (2); 
 
\node[la] at ($(1)!0.5!(4)$) {$\psi$};
\end{tz}
    such that $\psi \vert \varphi=e_f$ and $\frac{\varphi}{\psi}=\id_u$.
\end{defn}

\begin{rem} \label{conjvscomp}
    A conjoint pair in $\bA$ is a companion pair in $\bA^\mathrm{v,op}$.
\end{rem}

To describe our model structures, we introduce the following convenient notation.

\begin{notation} \label{notn:properties1}
    Given a double category $\bA$, we consider the following classes of morphisms: 
    \begin{itemize}[leftmargin=0.8cm]
        \item the class $\mathrm{hMor}$ of horizontal morphisms,
        \item the class $\mathrm{vMor}$ of vertical morphisms,
        \item the class $\mathrm{hEq}$ of horizontal equivalences, i.e., equivalences in the underlying horizontal $2$-category of $\bA$.
    \end{itemize}
    Given a class $S$ of horizontal or vertical morphisms in $\bA$, we consider the following properties: 
    \begin{itemize}[leftmargin=0.8cm]
        \item the property $\mathrm{comp}(S)$ that every morphism in $S$ has a companion, 
        \item the property $\mathrm{equip}(S)$ that every morphism in $S$ has both a companion and a conjoint.\footnote{The notation ``equip'' comes from the term \emph{equipments}, which are double categories such that all horizontal (or vertical) morphisms have both a companion and a conjoint.}
    \end{itemize}
    For example, the property $\mathrm{comp}(\mathrm{hMor})$ means that every horizontal morphism has a companion. 
\end{notation}

We aim to show that the following model structures on $\dblcat$ exist, all of which are endowed with the canonical trivial fibrations. The second and third column of the table below specify the fibrant objects of each model structure and describe the properties that these satisfy, which can all be expressed in terms of the existence of companions and conjoints for certain classes of morphisms. 

\begin{table}[h!]
    \centering
    \renewcommand{\arraystretch}{1.3}
    \begin{tabular}{l|C{2.5cm}|C{2.5cm}|C{3cm}|C{3cm}}
    Model & Fibrant double categories & Properties of fibrant objects & Homotopy \ \ theory of ... & Additional properties \\
    \hline
    \hline
       $\dblcat_\mathrm{greg}$ & All &  & Double \newline categories & Initial case \\
       \hline
        $\dblcat_\mathrm{whi}$ & Weakly \newline horizontally invariant & $\mathrm{comp}(\mathrm{hEq})$  & Reedy fibrant double categories & Double $(\infty,1)$-categorical nerve  \\
        \hline
        $\dblcat_\mathrm{h,eqp}$ & Horizontal equipments & $\mathrm{equip}(\mathrm{hMor})$  &  Horizontal equipments & Formal category theory \\
        \hline
     $\dblcat_\mathrm{tr}$  & Transposable & $\mathrm{comp}(\mathrm{hMor})\cap \mathrm{comp}(\mathrm{vMor})$ & $2$-categories &   \\
     \hline
        $\dblcat_\mathrm{tr,ladj}$ &  & $\mathrm{equip}(\mathrm{hMor})\cap \mathrm{comp}(\mathrm{vMor})$ &  $2$-categories with left adjoints &\\
        \hline
        $\dblcat_\mathrm{tr,adj}$ & Transposable equipments & $\mathrm{equip}(\mathrm{hMor})\cap \mathrm{equip}(\mathrm{vMor})$ &  $2$-categories with all adjoints & \\
    \hline
    \end{tabular}
    \vspace{.2cm}
    \caption{Model structures from companions and conjoints --- summary}\label{table:MScompconjpurposes}
\end{table}
Of course, each of these model structures also admits a transposed version where the roles of horizontal and vertical morphisms are reversed. We will not address these in any detail, but a full diagram of model structures and their interactions is included later in \cref{RQdiagram}.

Before we move on to the construction of these model structures, let us briefly discuss the advantages of each model, and the homotopy theory they represent, which are summarized in the last two columns of the table above:

  \begin{itemize}[leftmargin=0.8cm]
  \item We have shown in \cref{thm:localization} that the gregarious model structure $\dblcat_\mathrm{greg}$ is the initial case in the sense that every other model structure with trivial fibrations given by the canonical ones is a localization of it. 
\item The model structure $\dblcat_\mathrm{whi}$ for weakly horizontally invariant double categories was originally constructed in \cite{whi} when looking for a model structure on $\dblcat$ compatible with the horizontal inclusion $\bH\colon \twocat\to \dblcat$ of \cref{def:bbH}. Notably, this is the model structure used in \cite{lyne} to build a nerve functor $\bN$ from double categories to double $(\infty,1)$-categories. To understand the link with ``Reedy fibrant double categories'', note that a double category $\bA$ is weakly horizontally invariant (i.e., every horizontal equivalence has a companion) if and only if the induced $2$-functor $\bfH\llbracket\bV\mathbbm{2},\bA\rrbracket\to \bfH\bA\times \bfH\bA$ is an equifibration by \cite[Proposition 5.6]{fibrantly_induced}. This is a sort of Reedy fibrancy condition. Further evidence for this link is given by the fact that a double category $\bA$ is weakly horizontally invariant if and only if its nerve $\bN\bA$ is Reedy fibrant as a bisimplicial space; see \cite[Theorem 5.30]{lyne}.
\item For the model structure $\dblcat_{\mathrm{h,eqp}}$, the fibrant double categories are those such that every horizontal morphism has both a companion and a conjoint. This is the notion of \emph{equipments}, introduced by Wood in the setting of bicategories \cite{WoodI,WoodII} and later studied in the double categorical setting e.g.\ in \cite{Verity,Shulman_Cruttwell2010}.\footnote{Note that in those references the convention interchanges the role of horizontal and vertical morphisms.} 
Building on the fact that equipments provide a natural framework for formal category theory, the transposed version of this model structure is used in \cite{MSequipments} to prove a result on equivalence invariance of formal category theory. 
\item The model structure $\dblcat_\mathrm{tr}$ for transposable double categories (that is, the ones where all horizontal and vertical morphisms have a companion) models the homotopy theory of $2$-categories, in the sense that the square functor $\Sq\colon \twocat\to \dblcat_\mathrm{tr}$ of \cref{def:square} is a Quillen equivalence; we prove this claim in \cref{rem:2catbothleftandrightind}. 
\item The model structure $\dblcat_\mathrm{tr,ladj}$ (respectively,  $\dblcat_\mathrm{tr,adj}$) models the homotopy theory of $2$-categories such that every morphism has a left adjoint (respectively, both a left and a right adjoint). This is made precise by showing that the square functor is a Quillen equivalence from model structures on $\twocat$ whose fibrant objects are the $2$-categories having the properties just described, to the model structures $\dblcat_\mathrm{tr,ladj}$ and $\dblcat_\mathrm{tr,adj}$; see \cref{prop:SqQEadj}.
\end{itemize}

This section is devoted to the study of these model structures: we establish their existence in \cref{subsec:existenceTable1}, and prove our remaining claims regarding their homotopy theories and describe their interactions through Quillen adjunctions in \cref{subsec:propertiesTable1}.

\subsection{Existence of the model structures}\label{subsec:existenceTable1}

To construct each of the model structures $\dblcat_\bigstar$ listed in the table above,  we must find sets $\cJ_\bigstar$ of generating anodyne extensions with the following properties:
\begin{itemize}[leftmargin=0.8cm]
    \item we have inclusions $\cJ_0\subseteq \cJ_\bigstar\subseteq \cof(\cI)$,
    \item all double functors in $\cJ_\bigstar$ have canonically cofibrant domains,
    \item a double category $\bA$ satisfies the desired properties of the fibrant objects if and only if it has the right lifting property with respect to the double functors in $\cJ_\bigstar$,
    \item the conditions of \cref{equivalentcond4} are satisfied; i.e., every gregarious fibration  between naive fibrant double categories has the right lifting property with respect to the double functors in $\cJ_\bigstar$.
\end{itemize}

\begin{defn} \label{def:Jxxx}
    We define the sets $\cJ_\bigstar$ as follows:
    \begin{itemize}[leftmargin=0.8cm]
        \item $\cJ_\mathrm{whi}\coloneqq \cJ_0\cup\{\bH\Eadj\to\Sq\Eadj\}$,
        \item $\cJ_\mathrm{h,eqp}\coloneqq \cJ_0\cup\{\bH\mathbbm{2}\to\Sq\mathbbm{2}, \,\bH\mathbbm{2}\to\Sq\mathbbm{2}^\mathrm{v,op}\}$,
        \item $\cJ_\mathrm{tr}\coloneqq \cJ_0\cup\{\bH\mathbbm{2}\to\Sq\mathbbm{2}, \,\bV\mathbbm{2}\to\Sq\mathbbm{2}\}$,
        \item $\cJ_\mathrm{tr,ladj}\coloneqq \cJ_\mathrm{tr}\cup\{\bH\mathbbm{2}\to\Sq\mathbbm{2}^\mathrm{v,op}\}$,
        \item $\cJ_\mathrm{tr,adj}\coloneqq \cJ_\mathrm{tr,ladj}\cup\{\bV\mathbbm{2}\to\Sq\mathbbm{2}^\mathrm{v,op}\}$.
    \end{itemize}
\end{defn}

\begin{prop} \label{prop:compsetJ}
Let $\cJ'$ be a set of morphisms in $\dblcat$ which satisfies the hypotheses of \cref{equivalentcond4}. Then, a set $\cJ$ obtained by adding any of the following double functors to $\cJ'$
    \[ \bH\Eadj\to \Sq\Eadj, \;\bH\mathbbm{2}\to \Sq\mathbbm{2}, \;\bV\mathbbm{2}\to \Sq\mathbbm{2}, \; \bH\mathbbm{2}\to \Sq\mathbbm{2}^\mathrm{v,op},\; \bV\mathbbm{2}\to \Sq\mathbbm{2}^\mathrm{v,op},   \]
     also satisfies the hypotheses of \cref{equivalentcond4}.
\end{prop}

\begin{proof}
    We prove the result for $\cJ=\cJ'\cup \{\bH\mathbbm{2}\to \Sq\mathbbm{2}\}$. Suppose that $\bA$ and $\bB$ are double categories which have the right lifting property with respect to every double functor in~$\cJ$, and $F\colon\bA\to\bB$ is a gregarious fibration. By assumption, the double functor $F$ has the right lifting property with respect to every double functor in $\cJ'$. Thus, it remains to show that there is a lift in every commutative diagram as below left.
    \begin{tz}
\node[](1) {$\bH\mathbbm{2}$}; 
\node[right of=1,xshift=1cm](2) {$\bA$}; 
\node[below of=1](1') {$\Sq\mathbbm{2}$}; 
\node[below of=2](2') {$\bB$}; 

\draw[->] (1) to node[above,la]{$f$} (2); 
\draw[->] (1) to (1'); 
\draw[->] (2) to node[right,la]{$F$} (2'); 
\draw[->] (1') to node[below,la]{$(Ff,v,\chi,\omega)$} (2');
\draw[->,dashed] (1') to (2); 

\node[right of=2,xshift=2cm](1) {$\bH\mathbbm{2}$}; 
\node[right of=1,xshift=.5cm](2) {$\bA$}; 
\node[below of=1](1') {$\Sq\mathbbm{2}$}; 

\draw[->] (1) to node[above,la]{$f$} (2); 
\draw[->] (1) to (1');  
\draw[->,dashed] (1') to node[below,la,xshift=20pt]{$(f,u,\varphi,\psi)$} (2); 
\end{tz}

    By our assumption on $\bA$, we know there is a lift in the diagram above right. Hence, the tuples $(Ff,v,\chi,\omega)$ and $(Ff,Fu,F\varphi,F\psi)$ are two companion pairs for the same horizontal morphism $Ff$ in~$\bB$. The universal property of companions from \cite[\textsection 4.1.1]{Grandis} then guarantees that there is a unique invertible vertical globular square $\beta\colon Fu\cong v$ in $\bB$ satisfying the following pasting equality. 
    \begin{tz}
    \node[](1) {$FA$}; 
\node[below of=1](2) {$FA$}; 
\node[right of=1](3) {$FA$}; 
\node[right of=2](4) {$FC$};

\draw[d,pro] (1) to node[left,la]{} (2); 
\draw[d] (1) to node[above,la]{} (3); 
\draw[->] (2) to node[below,la]{$Ff$} (4); 
\draw[->,pro](3) to node[right,la]{$v$} (4); 
 
\node[la] at ($(1)!0.5!(4)$) {$\omega$};

        \node[right of=3,xshift=.5cm](1) {$FA$}; 
        \node at ($(4)!0.5!(1)$) {$=$};
\node[below of=1](2) {$FA$}; 
\node[right of=1](3) {$FA$}; 
\node[right of=2](4) {$FC$};
\node[right of=3](3') {$FA$}; 
\node[right of=4](4') {$FC$};

\draw[d,pro] (1) to node[left,la]{} (2); 
\draw[d] (1) to node[above,la]{} (3); 
\draw[->] (2) to node[below,la]{$Ff$} (4); 
\draw[->,pro](3) to node[left,la]{$Fu$} (4); 
\node[la] at ($(1)!0.5!(4)-(5pt,0)$) {$F\psi$};

\draw[d] (3) to node[above,la]{} (3'); 
\draw[d] (4) to node[below,la]{} (4'); 
\draw[->,pro](3') to node[right,la]{$v$} (4'); 
 
\node[la] at ($(3)!0.5!(4')-(5pt,0)$) {$\beta$};
\node[la] at ($(3)!0.5!(4')+(5pt,0)$) {$\cong$};
    \end{tz}
  Using the same universal property, we have that $\chi=\beta^{-1}\vert F\varphi$. Since $F$ is a gregarious fibration, by (f3) of \cref{lem:descrnfib0} applied to the vertical globular square $\beta$, there is a horizontally invertible square in $\bA$
    \begin{tz}
        \node[](1) {$A$}; 
\node[below of=1](2) {$C$}; 
\node[right of=1](3) {$A$}; 
\node[right of=2](4) {$C$};

\draw[->,pro] (1) to node[left,la]{$u$} (2); 
\draw[d] (1) to node[above,la]{} (3); 
\draw[d] (2) to node[below,la]{} (4); 
\draw[->,pro](3) to node[right,la]{$w$} (4); 
 
\node[la] at ($(1)!0.5!(4)-(5pt,0)$) {$\alpha$};
\node[la] at ($(1)!0.5!(4)+(5pt,0)$) {$\cong$};
    \end{tz}  
    such that $F\alpha=\beta$. Hence $(f,w,\varphi',\psi')$ with $\varphi'\coloneqq \alpha^{-1}\vert \varphi$ and $\psi'\coloneqq \psi\vert \alpha$ is a companion pair in $\bA$ such that $F(f,w,\varphi',\psi')=(Ff,v,\chi,\omega)$, providing the desired lift. 

     The result for $\cJ=\cJ'\cup \{\bH\Eadj\to \Sq\Eadj\}$ proceeds in the exact same way, keeping track of the additional data of the equivalences. The result for $\cJ=\cJ'\cup \{\bV\mathbbm{2}\to \Sq\mathbbm{2}\}$ is obtained by reversing the role of the horizontal and vertical morphisms. Lastly, the result for $\cJ=\cJ'\cup \{\bH\mathbbm{2}\to \Sq\mathbbm{2}^{\mathrm{v,op}}\}$ proceeds by reasoning in the vertical opposite $\bA^{\mathrm{v,op}}$ using \cref{conjvscomp}, and the one for $\cJ=\cJ'\cup \{\bV\mathbbm{2}\to \Sq\mathbbm{2}^{\mathrm{v,op}}\}$ then follows by reversing the role of the horizontal and vertical morphisms.
\end{proof}

With the above result in hand, we can now prove that all of the described model structures exist.

\begin{theorem}
    The model structures from \cref{table:MScompconjpurposes} exist. Moreover, they are all left proper and cofibrantly generated. 
\end{theorem}

\begin{proof}
    By definition, the sets $\cJ_\bigstar$ from \cref{def:Jxxx} are such that $\cJ_0\subseteq\cJ_\bigstar$, and it is immediate to see that all the double functors in $\cJ_\bigstar$ are canonical cofibrations with canonically cofibrant domain using the descriptions in \cref{charcof}. Moreover, each double functor added to $\cJ_\bigstar$ is designed to capture the new requirements on fibrant objects, as revealed by a straightforward inspection of the right lifting properties. The fact that each of these sets satisfies the hypotheses of \cref{equivalentcond4} is the content of \cref{prop:compsetJ}. Hence, by \cref{thm:genmodel}, the desired model structures exist.

    The fact that all of these model structures are cofibrantly generated is also guaranteed by \cref{thm:genmodel}, and they are left proper by \cref{thm:leftproper2}.
\end{proof}

\subsection{The models: purpose and interactions}\label{subsec:propertiesTable1}

We start by proving that the square functor is left Quillen when the target is the gregarious model structure.

\begin{prop} \label{prop:SqLQ}
    The adjunction
       \begin{tz}
\node[](1) {$\twocat$}; 
\node[right of=1,xshift=1.8cm](2) {$\dblcat_\mathrm{greg}$};

\draw[->,bend left=20] ($(1.east)+(0,4pt)$) to node[above,la]{$\Sq$} ($(2.west)+(0,4pt)$);
\draw[->,bend left=20] ($(2.west)-(0,4pt)$) to node[below,la]{$\bfR$} ($(1.east)-(0,4pt)$);

\node[la] at ($(1.east)!0.5!(2.west)$) {$\bot$};
\end{tz}
is a Quillen pair whose derived unit is a biequivalence.
\end{prop}

\begin{proof}
    Since the right adjoint $\bfR$ preserves all weak equivalences by \cref{prop:4tupleweakequiv}, to show that it is right Quillen it suffices to prove that it preserves fibrations. This is straightforward: given a fibration $F\colon\bA\to\bB$ in $\dblcat_\mathrm{greg}$, we see that conditions (f1) and (f2) in \cref{Lackfib} for $\bfR F$ correspond to conditions (f1) and (f2) in \cref{lem:descrnfib0} satisfied by $F$.

    For the claim regarding the derived unit, first note that it can be chosen to agree with the unit of the adjunction, as every double category in $\dblcat_\mathrm{greg}$ is fibrant. Finally, the unit $\cA\to \bfR\Sq\cA$ is a biequivalence. Indeed, it is the identity on objects; essentially full on morphisms since for any companion pair $(f,u,\varphi,\psi)$ in $\Sq\cA$, the identity at $f$ gives a $2$-isomorphism $(f,f,\id_f,\id_f)\cong (f,u,\varphi,\psi)$ in $\bfR\Sq\cA$; and fully faithful on $2$-morphisms by definition of the horizontal globular squares in $\Sq\cA$.
\end{proof}

We summarize in the following diagram the left Quillen functors relating the canonical model structure on $\twocat$ with the different model structures on $\dblcat$. 

\begin{prop} \label{LQdiagram}
    We have the following (non-commutative) diagram of left Quillen functors 
    \begin{tz}
    \node[](1) {$\dblcat_\mathrm{greg}$}; 
    \node[below of=1,xshift=-3cm](2) {$\dblcat_\mathrm{whi}$}; 
    \node[below of=1,xshift=3cm](2') {$\dblcat_\mathrm{wvi}$}; 
    \node[below of=1](0) {$\twocat$}; 
    \node[below of=0](0') {$\dblcat_\mathrm{tr}$}; 
    \node[below of=2,xshift=-3cm](4) {$\dblcat_\mathrm{h,eqp}$}; 
    \node[below of=2',xshift=3cm](4') {$\dblcat_\mathrm{v,eqp}$}; 

\draw[->](0) to node[left,la]{$\Sq$} (1);

\draw[->](0) to node[above,la]{$\bH$} (2);    
\draw[->](0) to node[above,la]{$\bV$} (2');

\draw[->](1) to node[above,la]{$\id$} (2); 
\draw[->](1) to node[above,la]{$\id$} (2'); 
\draw[->](2) to node[above,la]{$\id$} (4);
\draw[->](2') to node[above,la]{$\id$} (4');
\draw[->](2) to node[below,la]{$\id$} (0'); 
\draw[->](2') to node[below,la]{$\id$} (0'); 
\end{tz} 
    where the functors from $\twocat$ are left Quillen embeddings.
\end{prop}

\begin{proof}
To show that the identity functors are left Quillen, since the model structures involved all have the same cofibrations, it suffices to prove that we have inclusions 
\[\cJ_0\subseteq\cJ_\mathrm{whi}\subseteq \mathrm{cof}(\cJ_\mathrm{h,eqp}), \mathrm{cof}(\cJ_\mathrm{tr});\] the vertical counterparts will follow analogously.
 The fact that $\cJ_0\subseteq \cJ_\mathrm{whi}$ is immediate by definition of these sets. For the other two inclusions, we claim that the double functor $\bH\Eadj\to \Sq\Eadj$ in~$\cJ_\mathrm{whi}$ is an anodyne extension for every set $\cJ_\bigstar$ that contains the double functor $\bH\mathbbm{2}\to \Sq\mathbbm{2}$. Indeed, this follows from the fact that the companion of a horizontal adjoint equivalence can always be constructed as a vertical adjoint equivalence, by \cref{rmk:companionsforequivs}.

    For the functors involving $\twocat$, we have that  $\Sq\colon \twocat\to \dblcat_\mathrm{greg}$ is a left Quillen embedding by \cref{prop:SqLQ}, and  that $\bH\colon \twocat\to \dblcat_\mathrm{whi}$ is a left Quillen embedding by \cite[Theorem 6.3]{whi}; the case of $\bV$ is analogous.
\end{proof}

We now turn to proving that the square functor is both a left and a right Quillen equivalence when the target is the model structure for transposable double categories. In particular, this justifies our claim that $\dblcat_\mathrm{tr}$ models the homotopy theory of 2-categories.

\begin{prop} \label{rem:2catbothleftandrightind}
    The adjunctions
    \begin{tz}
\node[](1) {$\twocat$}; 
\node[right of=1,xshift=1.6cm](2) {$\dblcat_\mathrm{tr}$};

\draw[->,bend right=35] ($(2.west)+(0,8pt)$) to node[above,la]{$L_{Sq}$} ($(1.east)+(0,8pt)$);
\draw[->] ($(1.east)$) to node[over,la]{$\Sq$} ($(2.west)$);
\draw[->,bend left=35] ($(2.west)-(0,8pt)$) to node[below,la]{$\bfR$} ($(1.east)-(0,8pt)$);

\node[la] at ($(1.east)!0.5!(2.west)+(0,10pt)$) {$\bot$};
\node[la] at ($(1.east)!0.5!(2.west)-(0,10pt)$) {$\bot$};
\end{tz}
    are Quillen equivalences. Moreover, the model structure on $\twocat$ is both left- and right-transferred from the model structure on $\dblcat_\mathrm{tr}$ along $\Sq$.
\end{prop}

\begin{proof}
    The fact that $\Sq\colon \twocat\to \dblcat_\mathrm{tr}$ is left Quillen follows from \cref{LQdiagram}. To see that $\Sq\colon \twocat\to \dblcat_\mathrm{tr}$ is right Quillen, first note that by \cref{rem:Sqnaivefib} the functor $\Sq$ takes every $2$-category to a transposable double category. Then, it is straightforward to verify that $\Sq$ sends trivial fibrations and equifibrations in $\twocat$ to canonical trivial fibrations and gregarious fibrations in $\dblcat_\mathrm{tr}$ using the description of these classes of morphisms from \cref{Lackfib,thm:lackMS,def:trivfibs,lem:descrnfib0}. 

    To see that $\Sq$ is a Quillen equivalence, we prove that the adjunction $\Sq\dashv \bfR$ is a Quillen equivalence. As the proof of \cref{prop:SqLQ} above already shows that the unit of this adjunction is pointwise a biequivalence, by \cite[Lemma 3.3]{mehmet} it suffices to show that $\Sq\colon \twocat\to \dblcat_\mathrm{tr}$ both preserves and reflects weak equivalences. Since $\Sq$ is right Quillen and all objects in $\twocat$ are fibrant, it preserves all weak equivalences. For the converse, if $F\colon \cA\to \cB$ is a $2$-functor such that $\Sq F\colon \Sq\cA\to \Sq\cB$ is a weak equivalence in $\dblcat_\mathrm{tr}$, then $\Sq F$ is a gregarious biequivalence, as $\Sq\cA$ and $\Sq\cB$ are transposable. By \cref{prop:4tupleweakequiv}, this implies that $\mathbf{H}\Sq F=F$ is a biequivalence.
    
    Using that the functor $\bfH\colon \dblcat\to \twocat$ is a retraction of the square functor $\Sq$, we deduce the ``moreover'' part from \cref{rem:inductionQE}.
\end{proof}

Next, we deal with the cases of $\dblcat_\mathrm{tr,ladj}$ and $\dblcat_\mathrm{tr,adj}$. First let us note the following. 

\begin{rem}\label{rem:tr_localizations}
By \cref{thm:MSislocwrtJ}, the model structures $\dblcat_\mathrm{tr,ladj}$ and $\dblcat_\mathrm{tr,adj}$ are the localization of the model structure $\dblcat_\mathrm{tr}$ at the sets 
\[ \mathbb{S}_\mathrm{ladj}\coloneqq\{\bH\mathbbm{2}\to\Sq\mathbbm{2}^\mathrm{v,op}\} \quad \text{and}\quad  \mathbb{S}_\mathrm{adj}\coloneq \mathbb{S}_\mathrm{ladj}\cup \{\bV\mathbbm{2}\to\Sq\mathbbm{2}^\mathrm{v,op}\}. \]
\end{rem}

\begin{prop}\label{prop:right_transferred_alladjoints}
    The right-transferred model structures on $\twocat$ along the right adjoint
    \[ \Sq\colon \twocat\to \dblcat_\mathrm{tr,ladj} \quad \text{and} \quad \Sq\colon \twocat\to \dblcat_\mathrm{tr,adj} \]
    exist. We denote them by $\twocat_\mathrm{ladj}$ and $\twocat_\mathrm{adj}$, respectively.
\end{prop}

\begin{proof}
    Since the model structure on $\twocat$ is left proper and combinatorial, we can consider its left Bousfield localization with respect to the sets 
    \[ S_\mathrm{ladj}\coloneqq \{ r\colon \mathbbm{2}\to \mathrm{Adj} \} \quad \text{and} \quad S_\mathrm{adj}\coloneqq S_\mathrm{ladj}\cup \{ \ell\colon \mathbbm{2}\to \mathrm{Adj} \}, \]
    where $\mathrm{Adj}$ denotes the free $2$-category on an adjunction $\ell\dashv r$. Note that the image under $L_\Sq$ of the double functor $\bH\mathbbm{2}\to \Sq\mathbbm{2}^{\mathrm{v,op}}$ (resp.\ $\bV\mathbbm{2}\to \Sq\mathbbm{2}^{\mathrm{v,op}}$) is precisely the 2-functor $r\colon \mathbbm{2}\to \mathrm{Adj}$ (resp.\ $\ell\colon \mathbbm{2}\to \mathrm{Adj}$). Hence, by \cref{rem:fibindloc,rem:2catbothleftandrightind,rem:tr_localizations}, we get that the localizations of $\twocat$ at $S_\mathrm{ladj}$ and $S_\mathrm{adj}$ coincide with the right-transferred model structures along the right adjoint $\Sq$ from the model structures $\dblcat_\mathrm{tr,ladj}$ and $\dblcat_\mathrm{tr,adj}$, respectively. 
\end{proof}

\begin{rem}\label{fibrantladj}
    Given a $2$-category $\cA$, we have that $\Sq\cA$ is fibrant in $\dblcat_\mathrm{tr,ladj}$ if and only if every morphism in $\cA$ has a left adjoint. Indeed, let $f\colon A\to B$ be a morphism in $\cA$; then $(f,f,\id_f,\id_f)$ is a companion pair, as mentioned in \cref{rem:Sqnaivefib}. If the horizontal morphism $f$ further has a conjoint with data $(f,u,\varphi,\psi)$, we can then consider the squares
\begin{tz}
    \node[](1) {$B$}; 
\node[right of=1](2) {$B$}; 
\node[below of=1](3) {$A$}; 
\node[right of=3](4) {$B$}; 
\node[below of=3](5) {$B$};
\node[right of=5](6) {$B$};
\node[left of=3, xshift=.8cm](7) {$\eta\coloneqq$}; 

\draw[d,pro] (2) to node[left,la]{} (4); 
\draw[d,pro] (4) to node[left,la]{} (6); 
\draw[d] (1) to node[above,la]{} (2); 
\draw[d] (5) to node[above,la]{} (6);
\draw[->] (3) to node[above,la]{$f$} (4); 
\draw[->,pro](1) to node[left,la]{$u$} (3); 
\draw[->,pro](3) to node[left,la]{$f$} (5); 
 
 \node[la] at ($(1)!0.5!(4)$) {$\varphi$};
  \node[la] at ($(3)!0.5!(6)$) {$\id_f$};

\node[right of=1,xshift=4cm](1) {$A$}; 
\node[right of=1](2) {$A$}; 
\node[below of=1](3) {$A$}; 
\node[right of=3](4) {$B$}; 
\node[below of=3](5) {$A$};
\node[right of=5](6) {$A$};
\node[left of=3, xshift=.8cm](7) {$\epsilon\coloneqq$}; 

\draw[d,pro] (1) to node[left,la]{} (3); 
\draw[d,pro] (3) to node[left,la]{} (5); 
\draw[d] (1) to node[above,la]{} (2); 
\draw[d] (5) to node[above,la]{} (6);
\draw[->] (3) to node[above,la]{$f$} (4); 
\draw[->,pro](2) to node[right,la]{$f$} (4); 
\draw[->,pro](4) to node[right,la]{$u$} (6); 
 
 \node[la] at ($(1)!0.5!(4)$) {$\id_f$};
  \node[la] at ($(3)!0.5!(6)$) {$\psi$};
    \end{tz}
    in $\Sq\cA$, which must correspond to a unit $\eta\colon\id\Rightarrow fu$ and counit $\epsilon\colon uf\Rightarrow \id$ in $\cA$ exhibiting $u$ as a left adjoint of $f$. Similarly, starting from the data of a left adjoint in $\cA$ it is possible to construct the data of a conjoint in $\Sq\cA$. A similar analysis is carried out in \cite[Proposition 4.1.4]{Grandis}.

    The same idea can be used to show that a $2$-category $\cA$ is such that $\Sq\cA$ is fibrant in $\dblcat_\mathrm{tr,adj}$ if and only if every morphism in $\cA$ has both a left and a right adjoint.
\end{rem}

The above remark, together with the following result, justify our claim that $\dblcat_\mathrm{tr,ladj}$ (respectively,  $\dblcat_\mathrm{tr,adj}$) models the homotopy theory of $2$-categories such that every morphism has a left adjoint (respectively, both a left and a right adjoint).

\begin{prop} \label{prop:SqQEadj}
    The adjunctions
    \begin{tz}
\node[](1) {$\twocat_\mathrm{ladj}$}; 
\node[right of=1,xshift=2.2cm](2) {$\dblcat_\mathrm{tr,ladj}$};

\draw[->,bend right=35] ($(2.west)+(0,8pt)$) to node[above,la]{$L_{Sq}$} ($(1.east)+(0,8pt)$);
\draw[->] ($(1.east)$) to node[over,la]{$\Sq$} ($(2.west)$);
\draw[->,bend left=35] ($(2.west)-(0,8pt)$) to node[below,la]{$\bfR$} ($(1.east)-(0,8pt)$);

\node[la] at ($(1.east)!0.5!(2.west)+(0,10pt)$) {$\bot$};
\node[la] at ($(1.east)!0.5!(2.west)-(0,10pt)$) {$\bot$};

\node[right of=2,xshift=1.5cm](1) {$\twocat_\mathrm{adj}$}; 
\node[right of=1,xshift=2.1cm](2) {$\dblcat_\mathrm{tr,adj}$};

\draw[->,bend right=35] ($(2.west)+(0,8pt)$) to node[above,la]{$L_{Sq}$} ($(1.east)+(0,8pt)$);
\draw[->] ($(1.east)$) to node[over,la]{$\Sq$} ($(2.west)$);
\draw[->,bend left=35] ($(2.west)-(0,8pt)$) to node[below,la]{$\bfR$} ($(1.east)-(0,8pt)$);

\node[la] at ($(1.east)!0.5!(2.west)+(0,10pt)$) {$\bot$};
\node[la] at ($(1.east)!0.5!(2.west)-(0,10pt)$) {$\bot$};
\end{tz}
    are Quillen equivalences.
\end{prop}

\begin{proof}
    This follows directly from \cref{rem:2catbothleftandrightind} and \cite[Theorem 3.3.20(1)(b)]{Hirsch}, using the descriptions of the model structures $\twocat_\mathrm{ladj}$ and $\twocat_\mathrm{adj}$ as Bousfield localizations of $\twocat$ from the proof of \cref{prop:right_transferred_alladjoints}. 
\end{proof}

\begin{rem} In \cite{adjoining_adjoints} the authors construct a $2$-category by freely adjoining right adjoints for each morphism. Taking fibrant replacements in the (analogously defined) model category $\twocat_\mathrm{radj}$ provides a way to do this, as well as a generalization to allow $2$-categories with non-trivial $2$-morphisms as the starting point.
\end{rem}

We summarize the results from this subsection in the following diagram. 

\begin{prop} \label{RQdiagram}
    We have the following (non-commutative) diagram of right Quillen embeddings.
    \begin{tz}
    \node[](1) {$\dblcat_\mathrm{greg}$}; \node[below of=1,xshift=-3cm](2) {$\dblcat_\mathrm{whi}$}; 
    \node[below of=1,xshift=3cm](2') {$\dblcat_\mathrm{wvi}$};
    \node[below of=1](0) {$\twocat$}; 
    \node[below of=0](0') {$\dblcat_\mathrm{tr}$}; 
    \node[below of=0',yshift=-1.5cm](0'') {$\dblcat_\mathrm{tr,adj}$}; 
    \node[above of=0''](0''') {$\twocat_\mathrm{adj}$}; 
    \node[below of=2,xshift=-3cm](4) {$\dblcat_\mathrm{h,eqp}$}; 
    \node[below of=2',xshift=3cm](4') {$\dblcat_\mathrm{v,eqp}$}; 
     \node[below of=4,xshift=3cm](5) {$\dblcat_\mathrm{tr,ladj}$};
    \node[above of=5](05) {$\twocat_\mathrm{ladj}$};  
     \node[below of=4',xshift=-3cm](5') {$\dblcat_\mathrm{tr,radj}$};
    \node[above of=5'](05') {$\twocat_\mathrm{radj}$}; 

\draw[->](0) to node[left,la]{$\Sq$} node[right,la]{$\simeq$} (0');
\draw[->](05) to node[left,la]{$\Sq$} node[right,la]{$\simeq$} (5);
\draw[->](05') to node[left,la]{$\Sq$} node[right,la]{$\simeq$} (5');
\draw[->](0''') to node[left,la]{$\Sq$} node[right,la]{$\simeq$} (0'');

\draw[->](0) to node[above,la]{$\bH^{\simeq}$} (2);    
\draw[->](0) to node[above,la]{$\bV^\simeq$} (2');

\draw[->](2) to node[above,la]{$\id$} (1); 
\draw[->](2') to node[above,la]{$\id$} (1); 
\draw[->](4) to node[above,la]{$\id$} (2);
\draw[->](4') to node[above,la]{$\id$} (2');
\draw[->](0') to node[below,la]{$\id$} (2); 
\draw[->](0') to node[below,la]{$\id$} (2'); 
\draw[->](5) to node[below,la]{$\id$} (0');
\draw[->](5') to node[below,la]{$\id$} (0');
\draw[->](0'') to node[below,la]{$\id$} (5); 
\draw[->](0'') to node[below,la]{$\id$} (5');
\draw[->](5) to node[below,la]{$\id$} (4); 
\draw[->](5') to node[below,la]{$\id$} (4');  
\end{tz}
    where all instances of $\Sq$ are both left and right Quillen equivalences.
\end{prop}

\begin{proof}
From \cref{LQdiagram} we know that the (opposite) identity functors involving $\dblcat_\mathrm{greg}$, $\dblcat_\mathrm{whi}$, $\dblcat_\mathrm{tr}$ and $\dblcat_\mathrm{h,eqp}$ are left  Quillen; hence the corresponding identities displayed in the above diagram are right Quillen. The proof for the remaining identity functors in the bottom half of the diagram follows the same strategy, as one can verify that we have inclusions 
\[ \cJ_\mathrm{h,eqp}, \cJ_\mathrm{tr} \subseteq \cJ_\mathrm{tr,ladj} \subseteq \cJ_\mathrm{tr,adj}. \]
    Moreover, the derived counits of all these identity Quillen pairs are weak equivalences since all the model structures on $\dblcat$ have the same trivial fibrations. 
    
    The fact that $\bH^\simeq\colon \twocat\to \dblcat_\mathrm{whi}$ is a right Quillen embedding is proven in \cite[Theorem 6.6]{whi}, and $\Sq\colon \twocat\to \dblcat_\mathrm{tr}$ is a Quillen equivalence by \cref{rem:2catbothleftandrightind}. Finally, the fact that $\Sq\colon \twocat_\mathrm{tr,ladj}\to \dblcat_\mathrm{tr,ladj}$ and $\Sq\colon \twocat_\mathrm{tr,adj}\to \dblcat_\mathrm{tr,adj}$ are Quillen equivalences is proved in \cref{prop:SqQEadj}.
\end{proof}

\subsection{Monoidality} 
We end the section with a discussion of whether the model structures on $\dblcat$ of \cref{table:MScompconjpurposes} are monoidal with respect to B\"ohm's Gray tensor product.

We have already shown that the gregarious model structure is monoidal in \cref{thm:gregariousmonoidal}; and we also know that the model structure $\dblcat_\mathrm{whi}$ is monoidal by \cite[Theorem 7.8]{whi}. We will show that the remaining model structures,
$$\dblcat_\mathrm{h,eqp},\ \dblcat_\mathrm{tr},\  \dblcat_\mathrm{tr,ladj},\  \dblcat_\mathrm{tr, adj},$$ (and their transposed versions) are \emph{not} monoidal. To prove this claim, recall from \cref{thm:gray} that a model structure $\dblcat_\bigstar$ is monoidal with respect to the Gray tensor product if and only if, for every canonically cofibrant double category $\bX$ and every fibrant double category $\bA$, the pseudo-hom double category $\llbracket\bX,\bA\rrbracket_\mathrm{ps}$ is fibrant, which is condition (vi) of \cref{lemma:monoidalconditions}.

Since the fibrant objects in many of our model structures can be characterized by the existence of certain companions and/or conjoints, it will be useful to recognize these in a pseudo-hom double category $\llbracket\bX,\bA\rrbracket_\mathrm{ps}$. For this, we recall and adapt to the strict case the following results \cite[Definition 5.1 and Theorem 5.4]{Ruit_comp&conj} of Ruit.

\begin{defn} \label{companionable}
A square $\alpha$ in a double category $\bA$ is \emph{companionable} if its horizontal boundary $f$ and $f'$ are part of companion data $(f,w,\varphi,\psi)$ and $(f',w',\varphi',\psi')$ and the square given by the pasting
    \begin{diagram} \label{pasting}
        \node[](1) {$A$};
        \node[right of=1](2) {$A$};
    
        \node[below of=1](3) {$A$}; 
        \node[below of=3](4) {$A'$}; 
        \node[right of=3](5) {$B$}; 
        \node[below of=5](6) {$B'$}; 

        \node[below of=4](7) {$B'$};
        \node[right of=7](8) {$B'$};
        
        \draw[->](4) to node[below,la]{$f'$} (6); 
        \draw[->](3) to node[above,la]{$f$} (5);
        \draw[->,pro](3) to node[left,la]{$u$} (4);
        \draw[->,pro](5) to node[right,la]{$v$} (6); 
        
        \draw[->,pro](2) to node[right,la]{$w$} (5);
        \draw[d] (1) to (2);
        \draw[d,pro] (1) to (3);
        \draw[->,pro] (4) to node[left,la]{$w'$} (7);
        \draw[d] (7) to (8);
        \draw[d,pro] (6) to (8);

        \node[la] at ($(1)!0.5!(5)$) {$\psi$}; 
        \node[la] at ($(3)!0.5!(6)$) {$\alpha$}; 
        \node[la] at ($(4)!0.5!(8)$) {$\varphi'$}; 
    \end{diagram}
is a $2$-isomorphism in $\bfV\bA$. 
\end{defn}

\begin{theorem} \label{thm:companion_pshom}
Let $F,G\colon \bA\to \bB$ be double functors and $\tau\colon F\Rightarrow G$ be a horizontal pseudo natural transformation. Then the following are equivalent:
\begin{rome}
    \item the horizontal morphism $\tau$ is a companion in $\llbracket \bA,\bB\rrbracket_\mathrm{ps}$;
    \item for every vertical morphism $u\colon A\arrowdot A'$ in $\bA$, the square $\tau_u$ is companionable in $\bB$. 
    
        \begin{tz}
        \node[](3) {$FA$}; 
        \node[below of=3](4) {$FA'$}; 
        \node[right of=3](5) {$GA$}; 
        \node[below of=5](6) {$GA'$}; 

        \draw[->](4) to node[below,la]{$\tau_{A'}$} (6); 
        \draw[->](3) to node[above,la]{$\tau_A$} (5);
        \draw[->,pro](3) to node[left,la]{$Fu$} (4);
        \draw[->,pro](5) to node[right,la]{$Gu$} (6); 

        \node[la] at ($(3)!0.5!(6)$) {$\tau_u$}; 
    \end{tz}

\end{rome}
\end{theorem}

\begin{prop}\label{prop:counterex}
The model structures 
$$\dblcat_\mathrm{h,eqp},\ \dblcat_\mathrm{tr},\  \dblcat_\mathrm{tr,ladj},\  \dblcat_\mathrm{tr, adj},$$ 
are not monoidal with respect to B\"ohm's Gray tensor product.
\end{prop}
\begin{proof}
Note that for all the model structures we are considering, the double category $\bV \mathbbm{2}$ is canonically cofibrant by \cref{charcof}, and the fibrant objects are such that all horizontal morphisms have a companion. Thus, in view of the results above, we seek for each model structure a fibrant object $\bB$ such that when we consider the pseudo-hom double category $\llbracket \bV\mathbbm{2}, \bB\rrbracket_\mathrm{ps}$, the condition (vi) of \cref{thm:companion_pshom} does not hold.

Consider $\bB=\Sq\mathbbm{2}^\mathrm{dual}$, where $\mathbbm{2}^\mathrm{dual}$ is the free $2$-category on a morphism $f\colon A\to B$ with the property that every morphism has both left and right adjoints; see \cite[Section 4]{Mimram_compact2cat} for an explicit description. Note that $\Sq\mathbbm{2}^\mathrm{dual}$ is fibrant in all the relevant model structures.

Next, let $F,G\colon \bV\mathbbm{2}\to \Sq\mathbbm{2}^\mathrm{dual}$ be the double functors defined on the unique vertical morphism $v\colon 0\arrowdot 1$ of $\bV\mathbbm{2}$ as $F(v)=f$ and $G(v)=e_A$.
We define a horizontal pseudo natural transformation $\tau\colon F\Rightarrow G$ whose component horizontal morphisms are given by $\tau_0=\id_A$ and $\tau_{1}=u^L$, where $u^L\colon A'\to A$ is a left adjoint of $u$ with unit and counit given by $\eta$ and $\varepsilon$. Then $(u^L,u,\eta,\varepsilon)$ is a conjoint data in $\Sq\mathbbm{2}^\mathrm{dual}$ by \cref{fibrantladj}. Finally, we define $\tau_v=\eta$.  Note that $\tau_v$ is not companionable, as the pasting 
 \begin{tz} 
        \node[](3) {$A$}; 
        \node[below of=3](4) {$A'$}; 
        \node[right of=3](5) {$A$}; 
        \node[below of=5](6) {$A$}; 

        \node[below of=4](7) {$A$};
        \node[right of=7](8) {$A$};
        
        \draw[->](4) to node[below,la]{$u^L$} (6); 
        \draw[d](3) to (5);
        \draw[->,pro](3) to node[left,la]{$u$} (4);
        \draw[d,pro](5) to  (6); 
        
        \draw[->,pro] (4) to node[left,la]{$u^L$} (7);
        \draw[d] (7) to (8);
        \draw[d,pro] (6) to (8);

        \node[la] at ($(3)!0.5!(6)$) {$\eta$}; 
        \node[la] at ($(4)!0.5!(8)$) {$\id_{u^L}$}; 
    \end{tz}
is not a $2$-isomorphism in $\bfV\bB$, and so $\tau$ does not have a companion by \cref{thm:companion_pshom}. 
\end{proof}

\section{Application: Groupoid-like model structures}\label{section:ex2}

In this section we provide additional examples of our main result, this time constructing model structures on $\dblcat$ with a groupoidal flavor. We start by recalling the following notions.

\begin{defn}
    A square in a double category $\bA$ is \emph{weakly horizontally invertible} if it is an equivalence in the $2$-category $\bfH\llbracket\bV\mathbbm{2},\bA\rrbracket$. See \cite[Definition 2.5]{whi} for more details. 

    Similarly, a \emph{weakly vertically invertible} square can be defined by reversing the role of the horizontal and vertical morphisms. 
\end{defn}

\begin{defn}
    A \emph{double groupoid} is a double category in which all horizontal morphisms are horizontal equivalences, all vertical morphisms are vertical equivalences, and all squares are both weakly horizontally invertible and weakly vertically invertible.
\end{defn}

\begin{ex}
    The double category $\Sq\Eadj$ is a double groupoid. 
\end{ex}

We aim to show that the model structures on $\dblcat$ from \cref{table:MSgpd} exist, all of which are endowed with the canonical trivial fibrations. As before, the second and third column of the table describe the fibrant objects of these model structures, which are double groupoids satisfying the outlined properties, in the language of \cref{notn:properties1}.

\begin{table}[h!]
    \centering
    \renewcommand{\arraystretch}{1.3}
    \begin{tabular}{l|C{2.5cm}|C{3.5cm}|C{2.5cm}|C{2.5cm}}
    Model & Fibrant double groupoids & Properties of fibrant objects & Homotopy \ \ theory of ... & Additional properties \\
    \hline
    \hline
       $\dblcat_\mathrm{tr,gpd}$ & Transposable & $\mathrm{comp}(\mathrm{hMor})\cap \mathrm{comp}(\mathrm{vMor})$ &  $2$-groupoids &  \\
       \hline
         $\dblcat_{\emptyset\mathrm{/ctr}}$ & Empty or contractible & $\bA=\emptyset$ or $\bA\simeq\mathbbm{1}$ in $\dblcat_\mathrm{greg}$  & homotopy $(-1)$-types &  \\
        \hline
    $\dblcat_\mathrm{ctr}$ &  Contractible & $\bA\simeq\mathbbm{1}$ in $\dblcat_\mathrm{greg}$  &  homotopy $(-2)$-types & Terminal case \\
        \hline
    \end{tabular}
    \vspace{.2cm}
    \caption{Groupoid-like model structures --- summary}\label{table:MSgpd}
\end{table}

Before we move on to the construction of these model structures, let us briefly dis-
cuss the advantages of each model, and the homotopy theory they represent, which are
summarized in the last two columns of the table:

\begin{itemize}[leftmargin=0.8cm]
\item The model structure $\dblcat_\mathrm{tr,gpd}$ for transposable double groupoids (that is, the double groupoids where all horizontal and vertical morphisms have a companion) models the homotopy theory of 2-groupoids, in the sense that the square functor $\Sq\colon\twogpd\to\dblcat_\mathrm{tr,gpd}$ is a Quillen equivalence; we prove this claim in \cref{rem:2catbothleftandrightindgpd}.
\item The model structure $\dblcat_{\emptyset\mathrm{/ctr}}$ models homotopy (-1)-types; that is, truth values. We make this claim explicit in \cref{rem:homotopy-1types}.
\item The model structure $\dblcat_\mathrm{ctr}$ models homotopy (-2)-types. Indeed, all double categories here are weakly equivalent, and hence its homotopy category is equivalent to the terminal category, as we prove in  \cref{rem:2catbothleftandrightindctr}. This sits at the opposite end to the gregarious model structure in the poset of model structures: as we prove in \cref{prop:ctrterminal}, $\dblcat_\mathrm{ctr}$ is the terminal case, in the sense that it is a localization of every model structure on $\dblcat$ with trivial fibrations given by the canonical ones.
\end{itemize}

We establish the existence of these model structures in \cref{subsec:existenceTable2}, and prove our claims regarding their homotopy theories in \cref{subsec:purposesTable2}.

\subsection{Existence of the model structures}\label{subsec:existenceTable2}

Once again, we follow our general recipe to construct each of these model structures $\dblcat_\bigstar$ by finding convenient sets $\cJ_\bigstar$ of generating anodyne extensions.

\begin{defn} \label{def:Jxxxgpd}
    We define the sets $\cJ_\bigstar$ as follows:
    \begin{itemize}[leftmargin=0.8cm]
        \item $\cJ_{\mathrm{tr,gpd}}\coloneqq \cJ_
        \mathrm{tr}\cup\{\Sq\mathbbm{2}\to\Sq\Eadj,\,\bH\Sigma\mathbbm{2}\to\bH\Sigma I,\,\bV\Sigma\mathbbm{2}\to\bV\Sigma I\}$, 
        \item $\cJ_{\emptyset\mathrm{/ctr}}\coloneqq \cJ_0\cup (\cI\setminus\{\emptyset\to \mathbbm{1}\})$,
        \item $\cJ_{\mathrm{ctr}}\coloneqq \cJ_0\cup \cI$.
    \end{itemize}
\end{defn}

We first show that the set $\cJ_\mathrm{tr,gpd}$ gives us the desired description of the fibrant objects.

\begin{lem} \label{lem:fibtrdblgpd}
    A double category $\bA$ has the right lifting property with respect to every double functor in $\cJ_{\mathrm{tr,gpd}}$ if and only if it is a transposable double groupoid.
\end{lem}

\begin{proof}
    If $\bA$ is a transposable double groupoid, then it clearly has the desired lifting property. Conversely, if $\bA$ has the right lifting property with respect to every double functor in $\cJ_{\mathrm{tr,gpd}}$, then it is transposable since $\cJ_\mathrm{tr}\subseteq \cJ_\mathrm{tr,gpd}$; it remains to show that it is a double groupoid. Since every horizontal morphism in $\bA$ is part of a companion pair $\Sq\mathbbm{2}\to \bA$, and $\bA$ has the right lifting property against $\Sq\mathbbm{2}\to \Sq\Eadj$, then every horizontal morphism is part of a gregarious equivalence $\Sq\Eadj\to \bA$. In particular, it is a horizontal equivalence. Similarly, every vertical morphism is a vertical equivalence. 

    Finally, to see that every square $\alpha$ in $\bA$ is weakly horizontally invertible, consider the following pasting,
    \begin{tz}
        \node[](1) {$A$}; 
        \node[below of=1](2) {$A$}; 
        \node[right of=1](3) {$A$}; 
        \node[below of=3](4) {$A'$}; 
        \node[right of=3](5) {$B$}; 
        \node[below of=5](6) {$B'$}; 
        \node[right of=5](7) {$B'$}; 
        \node[below of=7](8) {$B'$}; 

        \draw[d](1) to (3); 
        \draw[d](6) to (8); 
        \draw[d,pro](1) to (2); 
        \draw[d,pro](7) to (8); 

        \draw[->](2) to node[below,la]{$\simeq$} (4); 
        \draw[->](4) to node[above,la]{$\simeq$} node[below,la]{$f'$} (6); 
        \draw[->](3) to node[below,la]{$\simeq$} node[above,la]{$f$} (5);
        \draw[->,pro](3) to node[left,la]{$u$} (4);
        \draw[->,pro](5) to node[right,la]{$v$} (6); 
        \draw[->](5) to node[above,la]{$\simeq$} (7); 

        \node[la] at ($(1)!0.5!(4)$) {$\simeq$}; 
        \node[la] at ($(3)!0.5!(6)$) {$\alpha$}; 
        \node[la] at ($(5)!0.5!(8)$) {$\simeq$}; 
    \end{tz}
    where the left-hand and right-hand squares are the weakly horizontally invariant squares of some gregarious equivalence data $\Sq\Eadj\to \bA$ associated with the vertical morphisms $u$ and $v$. Since $\bA$ has the right lifting property with respect to $\bH\Sigma\mathbbm{2}\to \bH\Sigma I$, the horizontal composite of the squares depicted above is vertically invertible. By \cite[Lemma A.2]{lyne}, this composite is  weakly horizontally invertible; hence this must also be the case for $\alpha$.  One can similarly show that $\alpha$ is weakly vertically invertible. 
\end{proof}

\begin{prop} \label{prop:gpdsetJ}
Let $\cJ'$ be a set of morphisms in $\dblcat$ which satisfies the hypotheses of \cref{equivalentcond4}. Then, a set $\cJ$ obtained by adding any of the following double functors to $\cJ'$
    \[ \Sq\mathbbm{2}\to\Sq\Eadj,\;\bH\Sigma\mathbbm{2}\to\bH\Sigma I,\;\bV\Sigma\mathbbm{2}\to\bV\Sigma I \]
    also satisfies the hypotheses of \cref{equivalentcond4}.
\end{prop}

\begin{proof}
    We prove the result for $\cJ=\cJ'\cup \{\Sq\mathbbm{2}\to\Sq\Eadj\}$. Suppose that $\bA$ and $\bB$ are double categories which have the right lifting property with respect to every double functor in $\cJ$, and $F\colon \bA\to \bB$ is a gregarious fibration. By assumption, the double functor~$F$ has the right lifting property with respect to every double functor in $\cJ'$. It remains to show that there is a lift $L$ in every commutative diagram as below left. 
    \begin{tz}
\node[](1) {$\Sq\mathbbm{2}$}; 
\node[right of=1,xshift=1cm](2) {$\bA$}; 
\node[below of=1](1') {$\Sq\Eadj$}; 
\node[below of=2](2') {$\bB$}; 

\draw[->] (1) to node[above,la]{$(f,u,\varphi,\psi)$} (2); 
\draw[->] (1) to (1'); 
\draw[->] (2) to node[right,la]{$F$} (2'); 
\draw[->] (1') to node[below,la]{$G$} (2');
\draw[->,dashed] (1') to node[below,la,xshift=3pt]{$L$} (2); 

\node[right of=2,xshift=2cm](1) {$\Sq\mathbbm{2}$}; 
\node[right of=1,xshift=1cm](2) {$\bA$}; 
\node[below of=1](1') {$\Sq\Eadj$}; 

\draw[->] (1) to node[above,la]{$(f,u,\varphi,\psi)$} (2); 
\draw[->] (1) to (1');  
\draw[->,dashed] (1') to node[below,la,xshift=3pt]{$H$} (2); 
\end{tz}
    By assumption, there is a lift in the diagram above right. Similarly to the proof of \cref{prop:compsetJ}, the double functors $G,F\circ H\colon \Sq\Eadj\to \bB$ give us two different data witnessing the fact that $(Ff,Fu)$ is a gregarious equivalence. We can relate these data by invertible squares, which we can then lift using (f2) and (f3) of \cref{lem:descrnfib0} for the gregarious fibration $F$ to rectify the data of $H$ into a data $L\colon \Sq\Eadj\to \bA$ such that $F\circ L=G$, providing the desired lift.

    For the case of $\cJ=\cJ'\cup \{\bH\Sigma\mathbbm{2}\to \bH\Sigma I\}$, note that every double functor whose domain has the right lifting property with respect to $\bH\Sigma\mathbbm{2}\to \bH\Sigma I$ must itself have the right lifting property with respect to $\bH\Sigma\mathbbm{2}\to\bH\Sigma I$, as inverses for horizontal globular squares are unique when they exist. The case of $\cJ=\cJ'\cup \{\bV\Sigma\mathbbm{2}\to\bV\Sigma I\}$ proceeds analogously. 
\end{proof}

We can now use the above result to prove the existence of the model
structures of \cref{table:MSgpd}.

\begin{theorem}
   The model structures from \cref{table:MSgpd} exist. Moreover, they are all left proper and cofibrantly generated.
\end{theorem}

\begin{proof}
    By definition, the sets $\cJ_\bigstar$ from \cref{def:Jxxxgpd} are such that $\cJ_0\subseteq\cJ_\bigstar$, and it is immediate to see that all the double functors in $\cJ_\bigstar$ are canonical cofibrations with canonically cofibrant domain using the descriptions in \cref{charcof}. 
    
    We first show the existence of the model structure $\dblcat_\mathrm{tr,gpd}$. By \cref{lem:fibtrdblgpd}, the set $\cJ_\mathrm{tr,gpd}$ captures the desired fibrant objects. Furthermore, this set satisfies the hypotheses of \cref{equivalentcond4} by \cref{prop:compsetJ,prop:gpdsetJ}. Hence, by \cref{thm:genmodel}, the desired model structure exists.

 Next we show the existence of the model structure $\dblcat_{\emptyset/\mathrm{ctr}}$. For this note that, as $\cJ_0\subseteq \mathrm{cof}(\cI\setminus\{\emptyset\to \mathbbm{1}\})$, a double category $\bA$ has the right lifting property with respect to the set $\cJ_{\emptyset/\mathrm{ctr}}=\cJ_0\cup (\cI\setminus \{\emptyset\to \mathbbm{1}\})$ if and only if $\bA$ has the right lifting property with respect to the set $\cI\setminus\{\emptyset\to \mathbbm{1}\}$. This is the case when $\bA=\emptyset$ or when $\bA\to \mathbbm{1}$ is a canonical trivial fibration, as the latter always lifts against $\emptyset\to \mathbbm{1}$ when $\bA\neq\emptyset$. Since the double functor $\bA\to \mathbbm{1}$ is always a gregarious fibration, this holds if and only if $\bA=\emptyset$ or $\bA\to \mathbbm{1}$ is a gregarious biequivalence, as desired. 

 To see that the set $\cJ_{\emptyset/\mathrm{ctr}}$ satisfies the hypotheses of \cref{equivalentcond4}, consider a gregarious fibration $F\colon \bA\to \bB$ between naive fibrant objects. Then there are three possible cases: either $\bA=\bB=\emptyset$; or $\bA=\emptyset$ and $\bB$ is contractible; or both $\bA$ and $\bB$ are contractible. In the first case we have that $F$ is the identity at $\emptyset$ and hence a fibration. In the second case, we can trivially verify that $F$ is a fibration using the right lifting property, as there are no commutative diagrams expressing lifts of $F$ with respect to the maps in $\cJ_{\emptyset/\mathrm{ctr}} $ to even consider. Finally, in the third case, we have a commutative diagram 
    \begin{tz}
        \node[](1) {$\bA$}; 
        \node[below right of=1](2) {$\mathbbm{1}$}; 
        \node[above right of=2](3) {$\bB$};

        \draw[->](1) to node[above,la]{$F$} (3); 
        \draw[->](1) to node[left,la,yshift=-2pt]{$\simeq$} (2);
        \draw[->](3) to node[right,la,yshift=-2pt]{$\simeq$} (2);
    \end{tz}
    where the unique double functors are gregarious biequivalences, and so by $2$-out-of-$3$, the double functor $F$ is also a gregarious biequivalence. Hence $F$ is a canonical trivial fibration by \cref{lem:fibwe0}. Then, by \cref{thm:genmodel}, the desired model structure exists.

    The proof of the existence of the model structure $\dblcat_\mathrm{ctr}$ is similar, noticing that a double category $\bA$ has the right lifting property with respect to the set $\cJ_\mathrm{ctr}=\cJ_0\cup \cI$ if and only if the unique double functor $\bA\to \mathbbm{1}$ is a canonical trivial fibration. 

    The fact that all of these model structures are cofibrantly generated is guaranteed by \cref{thm:genmodel}, and they are left proper by \cref{thm:leftproper2}.
\end{proof}

\subsection{The models: purpose and interactions}\label{subsec:purposesTable2}

 To prove that $\dblcat_\mathrm{tr,gpd}$ is a model for the homotopy theory of $2$-groupoids, we consider the composite of the adjunctions in \cref{MS2gpdvs2cat,adj:sq} and prove that they yield Quillen equivalences. 

\begin{theorem} \label{rem:2catbothleftandrightindgpd}
    The adjunctions
    \begin{tz}
\node[](1) {$\twogpd$}; 
\node[right of=1,xshift=2cm](2) {$\dblcat_\mathrm{tr,gpd}$};

\draw[->,bend right=35] ($(2.west)+(0,8pt)$) to node[above,la]{$\mathrm{loc}\circ L_{Sq}$} ($(1.east)+(0,8pt)$);
\draw[->] ($(1.east)$) to node[over,la]{$\Sq$} ($(2.west)$);
\draw[->,bend left=35] ($(2.west)-(0,8pt)$) to node[below,la]{$\mathrm{core}\circ \bfR$} ($(1.east)-(0,8pt)$);

\node[la] at ($(1.east)!0.5!(2.west)+(0,10pt)$) {$\bot$};
\node[la] at ($(1.east)!0.5!(2.west)-(0,10pt)$) {$\bot$};
\end{tz}
    are Quillen equivalences. 
\end{theorem}

\begin{proof}
    First note that $\Sq\colon \twogpd\to \dblcat_\mathrm{tr,gpd}$ is left Quillen, as it is the composite of the left Quillen functors 
    \[ \twogpd\xrightarrow{\iota}\twocat\xrightarrow{\Sq} \dblcat_\mathrm{tr}\xrightarrow{\id}\dblcat_\mathrm{tr,gpd} \]
    from \cref{MS2gpdvs2cat,rem:2catbothleftandrightind}, and using that $\dblcat_\mathrm{tr,gpd}$ is a localization of $\dblcat_\mathrm{tr}$ as $\cJ_\mathrm{tr}\subseteq \cJ_\mathrm{tr,gpd}$. To see that $\Sq\colon \twogpd\to \dblcat_\mathrm{tr,gpd}$ is right Quillen, recall that the composite 
    \[ \twogpd\xrightarrow{\iota}\twocat\xrightarrow{\Sq} \dblcat_\mathrm{tr} \]
    is right Quillen by \cref{MS2gpdvs2cat,rem:2catbothleftandrightind}, and so it suffices to show that the functor $\Sq\colon \twogpd\to \dblcat_\mathrm{tr,gpd}$ preserves fibrant objects. But if $\cG$ is a $2$-groupoid, then $\Sq\cG$ is transposable by \cref{rem:Sqnaivefib} and is a double groupoid by direct inspection.

To see that $\Sq$ is a Quillen equivalence, it suffices to show that the adjunction $\Sq\dashv \mathrm{core}\circ \bfR$ is a Quillen equivalence; this proceeds just like the proof of \cref{rem:2catbothleftandrightind}.
\end{proof}

Next, let $\{\bot\to \top\}$ denote the category free on a morphism. There is a canonical functor $P\colon \dblcat\to \{\bot\to \top\}$ sending the initial double category $\emptyset$ to $\bot$ and any other double category to $\top$. This admits a right adjoint $\{\bot\to \top\}\to \dblcat$ given by picking the unique double functor $\emptyset\to \mathbbm{1}$. We show that this forms a Quillen equivalence, proving that $\dblcat_{\emptyset\mathrm{/ctr}}$ models homotopy $(-1)$-types.

\begin{theorem} \label{rem:homotopy-1types}
    The adjunction
     \begin{tz}
\node[](1) {$\dblcat_{\emptyset\mathrm{/ctr}}$}; 
\node[right of=1,xshift=2.3cm](2) {$\{\bot\to \top\}$};

\draw[->,bend left=20] ($(1.east)+(0,4pt)$) to node[above,la]{$P$} ($(2.west)+(0,4pt)$);
\draw[->,bend left=20] ($(2.west)-(0,4pt)$) to node[below,la]{$\emptyset\to \mathbbm{1}$} ($(1.east)-(0,4pt)$);

\node[la] at ($(1.east)!0.5!(2.west)$) {$\bot$};
\end{tz}
    is a Quillen equivalence, where $\{\bot\to \top\}$ is endowed with the trivial model structure. 
\end{theorem}

\begin{proof}
    First note that all double categories except for the initial double category $\emptyset$ are weakly equivalent in $\dblcat_{\emptyset\mathrm{/ctr}}$, since every double functor between contractible double groupoids is a gregarious biequivalence by $2$-out-of-$3$. Then, it is straightforward to see that the functor $P\colon \dblcat_{\emptyset\mathrm{/ctr}}\to \{\bot\to \top\}$ is both left and right Quillen, as it preserves weak equivalences and the model structure on $\{\bot\to \top\}$ is the trivial one. 
    
    Finally, for $X\in\{\bot,\top\}$ and any double category $\bA$, a direct verification reveals that a morphism $P\bA\to X$ is an isomorphism if and only if its adjunct is a weak equivalence in $\dblcat_{\emptyset\mathrm{/ctr}}$; we then conclude that this Quillen pair is a Quillen equivalence.
\end{proof}

We now prove that $\dblcat_\mathrm{ctr}$ models homotopy $(-2)$-types.

\begin{theorem} \label{rem:2catbothleftandrightindctr}
    The adjunctions
    \begin{tz}
\node[](1) {$\dblcat_\mathrm{ctr}$}; 
\node[right of=1,xshift=1.7cm](2) {$\{\top\}$};

\draw[->,bend right=35] ($(2.west)+(0,8pt)$) to node[above,la]{$\emptyset$} ($(1.east)+(0,8pt)$);
\draw[->] ($(1.east)$) to node[over,la]{$!$} ($(2.west)$);
\draw[->,bend left=35] ($(2.west)-(0,8pt)$) to node[below,la]{$\mathbbm{1}$} ($(1.east)-(0,8pt)$);

\node[la] at ($(1.east)!0.5!(2.west)+(0,10pt)$) {$\bot$};
\node[la] at ($(1.east)!0.5!(2.west)-(0,10pt)$) {$\bot$};
\end{tz}
    are Quillen equivalences, where $\{\top\}$ is the terminal category endowed with the trivial model structure. 
\end{theorem}

\begin{proof}
    It is straightforward to check that the unique functor $\dblcat_{\mathrm{ctr}}\to \{\top\}$ is both left and right Quillen. Moreover, all double categories are weakly equivalent in $\dblcat_{\mathrm{ctr}}$, since every double functor between contractible double groupoids is a gregarious biequivalence by $2$-out-of-$3$. Hence the derived units and counits are all weak equivalences.  
\end{proof}

We conclude this section by proving that the model structure for contractible double groupoids is terminal among model structures on $\dblcat$ endowed with the canonical trivial fibrations.

\begin{prop}\label{prop:ctrterminal}
    The model category $\dblcat_\mathrm{ctr}$ is a localization of every model structure on $\dblcat$ whose trivial fibrations are the canonical ones.
\end{prop} 

\begin{proof}
Suppose that $\dblcat$ is endowed with a model structure whose trivial fibrations are the canonical ones; we will first show that $\id\colon\dblcat\to\dblcat_\mathrm{ctr}$ is left Quillen. The fact that it preserves cofibrations is immediate, as both model structures have as cofibrations the canonical ones. To see that it preserves trivial cofibrations, recall from \cref{def:Jxxxgpd} that $\dblcat_\mathrm{ctr}$ was constructed from the generating set of anodyne extensions $\cJ=\cJ_0\cup\cI$. As anodyne extensions are trivial cofibrations, this implies that all canonical cofibrations are weak equivalences in $\dblcat_\mathrm{ctr}$; in particular, this applies to the trivial cofibrations in $\dblcat$. 

Lastly, the derived counit at a fibrant object $\bA\in \dblcat_\mathrm{ctr}$ is given by a canonical cofibrant replacement of $\bA$ in $\dblcat$ and can be chosen to be a canonical trivial fibration. Hence it is also a weak equivalence in $\dblcat_\mathrm{ctr}$.
\end{proof}

\subsection{Monoidality} 
We end the section by showing that the model structures on $\dblcat$ of \cref{table:MSgpd} are monoidal with respect to Böhm's Gray tensor product. 

\begin{prop}
    The model structures $\dblcat_{\emptyset\mathrm{/ctr}}$ and $\dblcat_\mathrm{ctr}$ are monoidal with respect to the Gray tensor product.    
\end{prop}

\begin{proof}
   This is guaranteed by \cref{thm:gray}, as condition (ii) of \cref{lemma:monoidalconditions} for the sets $\cJ_\mathrm{ctr}=\cI$ and $\cJ_{\emptyset\mathrm{/ctr}}=\cI\setminus \{ \emptyset\to\mathbbm{1}\}$ follows directly from \cref{pushprodofcof}.
\end{proof}

To show that $\dblcat_{\mathrm{tr,gpd}}$ is also monoidal with respect to the Gray tensor product, we use the following lemmas. 
\begin{lem} \label{inviscompanionable}
    Consider a weakly horizontally invertible square $\alpha$ in a double category $\bA$
    \begin{tz}
        \node[](3) {$A$}; 
        \node[below of=3](4) {$A'$}; 
        \node[right of=3](5) {$B$}; 
        \node[below of=5](6) {$B'$}; 

        \draw[->](4) to node[below,la]{$f'$} (6); 
        \draw[->](3) to node[above,la]{$f$} (5);
        \draw[->,pro](3) to node[left,la]{$u$} (4);
        \draw[->,pro](5) to node[right,la]{$v$} (6); 

        \node[la] at ($(3)!0.5!(6)$) {$\alpha$}; 
    \end{tz}
    such that $f$ and $f'$ are part of companion data $(f,w,\varphi,\psi)$ and $(f',w',\varphi',\psi')$. Then $\alpha$ is companionable (see \cref{companionable}). 
\end{lem}

\begin{proof}
    Since $\alpha$ is weakly horizontally invertible, the horizontal morphisms $f$ and $f'$ are horizontal equivalences and so the squares of the companion data of $(f,w,\varphi,\psi)$ and $(f',w',\varphi',\psi')$ are weakly horizontally invertible. Hence the pasting \eqref{pasting} is also weakly horizontally invertible. As it has trivial horizontal boundaries, it is in fact a $2$-isomorphism in $\bfV \bA$, as desired. 
\end{proof}

\begin{lem} \label{transposablehom}
    For every canonically cofibrant double category $\bX$ and every transposable double groupoid $\bA$, the pseudo-hom double category $\llbracket \bX,\bA\rrbracket_\mathrm{ps}$ is a transposable double groupoid. 
\end{lem}

\begin{proof}
    It is clear that $\llbracket \bX,\bA\rrbracket_\mathrm{ps}$ is a double groupoid; we show that it is transposable.
    Since all squares in $\bA$ are weakly horizontally invertible and all horizontal morphisms in $\bA$ are part of a companion data, it follows from \cref{thm:companion_pshom,inviscompanionable} that every horizontal pseudo natural transformation in $\llbracket \bX,\bA\rrbracket_\mathrm{ps}$ is part of a companion data. By interchanging the horizontal and vertical directions, we obtain that every vertical morphism in $\llbracket \bX,\bA\rrbracket_\mathrm{ps}$ is part of a companion data, as well. 
\end{proof}

We can now prove the desired result. 

\begin{theorem}
    The model structure $\dblcat_\mathrm{tr,gpd}$ is monoidal with respect to the Gray tensor product. 
\end{theorem}

\begin{proof}
    This follows from \cref{thm:gray} using that condition (vi) of \cref{lemma:monoidalconditions} is satisfied by \cref{transposablehom}. 
\end{proof}

\bibliographystyle{alpha}
\bibliography{references}

\end{document}